\documentclass[letterpaper, 11pt,  reqno]{amsart}

\usepackage{amsmath,amssymb,amscd,amsthm,amsxtra, esint}

\usepackage[implicit=true]{hyperref}

\allowdisplaybreaks[2]

\usepackage{color}

\definecolor{gr}{rgb}   {0.,   0.69,   0.23 }
\definecolor{bl}{rgb}   {0.,   0.5,   1. }
\definecolor{mg}{rgb}   {0.85,  0.,    0.85}
\definecolor{yl}{rgb}   {0.8,  0.7,   0.}
\definecolor{or}{rgb}  {0.7,0.2,0.2}

\newtheorem{theorem}{Theorem} [section]
\newtheorem{maintheorem}{Theorem}
\newtheorem{lemma}[theorem]{Lemma}
\newtheorem{proposition}[theorem]{Proposition}
\newtheorem{remark}[theorem]{Remark}

\DeclareMathOperator*{\intt}{\int}

\DeclareMathOperator*{\supp}{supp}

\DeclareMathOperator{\Id}{Id}

\newcommand{\I}{\hspace{0.5mm}\text{I}\hspace{0.5mm}}
\newcommand{\II}{\text{I \hspace{-2.8mm} I} }
\newcommand{\III}{\text{I \hspace{-2.9mm} I \hspace{-2.9mm} I}}

\newcommand{\IV}{\text{I \hspace{-2.9mm} V}}

\newcommand{\noi}{\noindent}
\newcommand{\Z}{\mathbb{Z}}
\newcommand{\R}{\mathbb{R}}
\newcommand{\C}{\mathbb{C}}
\newcommand{\T}{\mathbb{T}}

\let\Re=\undefined\DeclareMathOperator*{\Re}{Re}
\let\Im=\undefined\DeclareMathOperator*{\Im}{Im}

\let\P= \undefined
\newcommand{\P}{\mathbf{P}}

\newcommand{\E}{\mathbb{E}}

\newcommand{\FL}{\mathcal{F}L}

\newcommand{\EE}{\mathcal{E}}
\newcommand{\FF}{\mathcal{F}}
\newcommand{\In}{\mathcal{I}}

\newcommand{\F}{\mathcal{F}}

\newcommand{\al}{\alpha}
\newcommand{\be}{\beta}
\newcommand{\dl}{\delta}

\newcommand{\eps}{\varepsilon}
\newcommand{\kk}{\kappa}
\newcommand{\g}{\gamma}
\newcommand{\G}{\Gamma}
\newcommand{\ld}{\lambda}

\newcommand{\s}{\sigma}
\newcommand{\Si}{\Sigma}
\newcommand{\ft}{\widehat}

\newcommand{\wt}{\widetilde}
\newcommand{\cj}{\overline}
\newcommand{\dx}{\partial_x}
\newcommand{\dt}{\partial_t}
\newcommand{\dd}{\partial}

\newcommand{\ta}{\theta}

\renewcommand{\l}{\ell}
\renewcommand{\o}{\omega}
\renewcommand{\O}{\Omega}

\newcommand{\les}{\lesssim}
\newcommand{\ges}{\gtrsim}

\newcommand{\jb}[1]
{\langle #1 \rangle}

\newcommand{\ind}{\mathbf 1}

\renewcommand{\S}{\mathcal{S}}

\newcommand{\N}{\mathbb{N}}

\newcommand{\NN}{\mathcal{N}}

\newcommand{\fN}{\mathfrak{N}}

\newcommand{\GG}{\mathcal{G}}

\newcommand{\J}{\mathcal{J}}
\newcommand{\A}{\mathcal{A}}

\renewcommand{\H}{\mathcal{H}}

\newtheorem*{ackno}{Acknowledgements}

\numberwithin{equation}{section}
\numberwithin{theorem}{section}

\newcommand{\z}{\textsf{z}}

\begin{document}
\baselineskip = 14pt
\title[4NLS with white noise initial data]{Solving the 4NLS with white noise initial data} 
\author[T.~Oh, N.~Tzvetkov, and  Y.~Wang]{Tadahiro Oh, Nikolay Tzvetkov, and Yuzhao Wang}
\address{
Tadahiro Oh, School of Mathematics\\
The University of Edinburgh\\
and The Maxwell Institute for the Mathematical Sciences\\
James Clerk Maxwell Building\\
The King's Buildings\\
Peter Guthrie Tait Road\\
Edinburgh\\ 
EH9 3FD\\
 United Kingdom}

\email{hiro.oh@ed.ac.uk}

\address{
Nikolay Tzvetkov\\
Universit\'e de Cergy-Pontoise\\
 2, av.~Adolphe Chauvin\\
  95302 Cergy-Pontoise Cedex \\
  France}

\email{nikolay.tzvetkov@u-cergy.fr}


\address{
Yuzhao Wang, School of Mathematics\\
The University of Edinburgh\\
and The Maxwell Institute for the Mathematical Sciences\\
James Clerk Maxwell Building\\
The King's Buildings\\
Peter Guthrie Tait Road\\
Edinburgh\\ 
EH9 3FD\\
 United Kingdom\\
 and
 School of Mathematics\\
Watson Building\\
University of Birmingham\\
Edgbaston\\
Birmingham\\
B15 2TT\\
United Kingdom
 }

\email{y.wang.14@bham.ac.uk}

%
%

\begin{abstract}
We construct global-in-time singular dynamics for
 the (renormalized) cubic  fourth order nonlinear Schr\"odinger equation 
 on the circle, having the white noise 
 measure as an invariant measure.
For this purpose, we introduce
the ``random-resonant\,/\,nonlinear decomposition'', 
which allows us to single out the singular component of the solution.
Unlike the classical McKean, Bourgain, Da Prato-Debussche type argument, 
this singular component is nonlinear, 
consisting of arbitrarily high powers of the random initial data.
We also employ a random gauge transform,  leading to random Fourier restriction norm spaces.
For this problem, a contraction argument does not work and 
we instead establish convergence of smooth approximating solutions
by studying the partially iterated Duhamel formulation
under the random gauge transform.
We reduce the crucial nonlinear estimates
to boundedness properties
of certain random multilinear functionals
of the white noise.

\end{abstract}
\maketitle

\tableofcontents

\section{Introduction}\label{SEC:1}
\subsection{White noise on the circle and Hamiltonian partial differential equations}

A white noise on the circle $\T = \R/(2\pi \Z)$ is defined as 
the following infinite-dimensional  random variable:\footnote 
{By convention, we endow $\T$ with the normalized Lebesgue measure $(2\pi)^{-1}dx$.}
\begin{equation}\label{grand_publique}
u^\o(x)=\sum_{n\in\Z} g_{n}(\o) e^{inx}, 
\end{equation}

\noi
where  $\{ g_n \}_{n \in \Z}$ is a family of independent standard complex-valued Gaussian random variables.
On the other hand, using the representation of the $L^2(\T)$-norm in terms of the Fourier coefficients, one may formally define the white noise measure induced by \eqref{grand_publique}  as
\begin{equation*}
\text{``} Z^{-1} e^{-\frac 12 \| u\|_{L^2(\T)}^2} du\text{''}. 
 \end{equation*}

 \noi
 There are many important Hamiltonian PDEs such as the Korteweg-de Vries equation
 (KdV) and  the nonlinear Schr\"odinger equations (NLS),
 under which 
 the $L^2$-norm of a solution is conserved. 
 Therefore, for this type of equations, thanks to the general globalization argument introduced 
 by Bourgain in \cite{BO94, BO96}, 
 if one can solve the equation {\it locally in time} with data distributed according to \eqref{grand_publique}, then 
 one can almost surely extend the solutions for {\it all times} and the white noise would be an {\it invariant measure} of the resulting flow.

 It is easy to check that the white noise measure induced by  \eqref{grand_publique} is  supported 
 in the space of distributions  
 $H^s(\T)\setminus H^{-\frac 12}(\T)$, $s < -\frac{1}{2}$.  
 It is this low regularity which makes it very difficult to solve locally in time a Hamiltonian PDE with 
 the white noise initial data defined in~\eqref{grand_publique}. It is remarkable that this severe  difficulty was overcome in the context of the KdV equation; see \cite{QV,OH4, OH5, OHRIMS, OH6}.
 An important property of the KdV equation heavily exploited 
 in these works  is   the {\it absence of resonant interactions}
  when restricted to solutions with a fixed zero Fourier mode (which is a conserved quantity for the KdV equation). 
As we shall see below, in the case of NLS-type equations,
one may remove a part of the resonant interactions by a gauge transform. 
Even after such a transformation, however,  there are remaining resonant interactions.  
The main goal of this work is to show how, by exploiting an intricate mixture of probabilistic and deterministic analysis, one may deal with such  resonant interactions in the context of the  
cubic fourth order  nonlinear Schr\"odinger equation on the circle with 
the white noise initial data   \eqref{grand_publique}. 
In our construction,  the main random part of the solutions will be a nonlinear object (in fact, of infinite degree),
which is 
in sharp contrast with the simple random linear evolution appearing in the previous random data well-posedness results such as~\cite{BO96, BT1}. 
 This difference between our main result and \cite{BO96, BT1} is similar in spirit with the difference between ``scattering" and ``modified scattering" appearing in the analysis of dispersive PDEs posed on the Euclidean space.  
 See Remarks~\ref{REM:scattering} 
 and~\ref{mod_scat_det} 
 below.

We succeeded to make our method work only for an NLS equation with a sufficiently strong dispersion. The generalization of our result to the  more standard  (in particular because of its  integrability) 
 NLS with the second order dispersion remains as a  challenging open problem.  

\subsection{The cubic fourth order nonlinear Schr\"odinger equation and a soft formulation of the main result}

In this work,  we consider  the 
cubic fourth order  nonlinear Schr\"odinger equation (4NLS) 
 on the circle $\T$:
\begin{align}
\begin{cases}
i \dt u =   \dx^4 u +  |u|^{2}u \\
u|_{t = 0} = u_0,
\end{cases}
\qquad (x, t) \in \T\times \R,
\label{NLS1}
\end{align}
\noi
where  $u$ is complex-valued. 
The equation \eqref{NLS1} is also called  the biharmonic NLS
and it was studied for instance in \cite{IK, Turitsyn} in the context of stability of solitons in magnetic materials.
The $L^2$-norm is formally conserved by the dynamics of \eqref{NLS1} and therefore, as discussed in the previous subsection,
 one may hope to 
 construct global dynamics of \eqref{NLS1} with data given by  \eqref{grand_publique}. 
 This is a delicate problem for many reasons, 
 the most basic one being that it is not clear how to interpret the nonlinearity for such low-regularity solutions.

Let us now briefly go over the deterministic well-posedness theory of \eqref{NLS1}. A simple fixed point argument via the Fourier restriction norm method 
introduced by Bourgain~\cite{BO93} yields local well-posedness of \eqref{NLS1} in $H^s(\T)$, $s\geq 0$. The main ingredient is the following $L^4$-Strichartz estimate:
\begin{align}
\| u \|_{L^4(\T\times \R)} \les \| u \|_{X^{0, \frac{5}{16}}}, 
\label{L4a}
\end{align}

\noi
where  $X^{s,b}$ denotes the  Fourier restriction norm space adapted
to  \eqref{NLS1}.
See \cite{OTz1} for the proof of \eqref{L4a}.
 Thanks to the $L^2$-conservation law,  this local result immediately implies  global well-posedness of \eqref{NLS1} in $H^s(\T)$, $s\geq 0$.
The equation \eqref{NLS1} is known to be ill-posed
in negative Sobolev spaces
in the sense of non-existence of solutions~\cite{GO, OW}.
See also \cite{OW0, CP}
for ill-posedness by norm inflation.
We point out that the ill-posedness results in~\cite{OW0, CP} 
also apply to the renormalized equation \eqref{NLS2} below.

Taking into account that we have a well-defined flow of \eqref{NLS1} for smooth initial data,
 one may formulate the problem of solving \eqref{NLS1} with the white noise initial data 
\eqref{grand_publique}  
as that of studying the limiting behavior of smooth solutions to \eqref{NLS1} 
with initial data given by  suitable regularizations of  \eqref{grand_publique}.  
We do not know the answer to this question in  full generality 
but 
we can answer it in a satisfactory manner  
for the natural regularizations by mollification.

Let $\{u_{0,m}^\o\}_{m=1}^\infty$ be a sequence of random 
smooth  functions defined as the regularization of   
$u^\o$ in~\eqref{grand_publique} by mollification, i.e.
\begin{align}
u_{0,m}^\o = u^\o* \rho_m=\sum_{n\in\Z}\ft{\rho_m}(n) g_{n}(\o) e^{inx} ,
\label{smooth1}
\end{align}

\noi
where $\ft{\rho_m}(n)=\theta(n/m)$ with 
 a bump function $\ta $ on $\R$ which equals one near the origin.\footnote
 {We also allow $\ta$ to be a sharp cutoff function $\ind_{[-1, 1]}(n)$, in which case
the resulting $ u_{0,m}^\o$ corresponds to the frequency truncated version of the white noise
onto the frequencies $\{|n|\leq m\}$.} 
Denote by $u_m$ the smooth solution to \eqref{NLS1} with  smooth initial data 
$u_m|_{t = 0} = u_{0,m}^\o$ constructed in \cite{OTz1}. 
If we could solve the equation \eqref{NLS1} with data given by  \eqref{grand_publique},
 then the sequence $\{u_m\}_{m = 1}^\infty$ would converge to the solution in an appropriate sense. 
The ill-posedness result in~\cite{GO, OW}, however,  implies that there is no hope to make 
$\{u_m\}_{m = 1}^\infty$ converge in any Sobolev space of negative regularity. 
It turns out that a ``renormalization" of $u_m$ is convergent.  
Here is a precise statement.

\begin{maintheorem}\label{THM:soft_version}
The sequence
$
\Big\{\exp\big(2  i t  \|u_m(t)\|^2_{L^2}\big) \,u_m(t)\Big\}_{m = 1}^\infty
$
converges almost surely in\footnote {Here, we endow
$C(\R; H^s(\T))$ with the compact-open topology in time.} $C(\R; H^s(\T))$, $s<-\frac 12$.
If we denote the limit by $u$,   then we have
$$
u=\sum_{n\in\Z} g_{n}(t,\o) e^{inx}\,,
$$
where for every $t\in\R$,  $\{ g_n(t,\o) \}_{n \in \Z}$ is a family of independent standard complex-valued Gaussian random variables.
Furthermore, the limit $u$ does not depend on the choice of the bump function $\theta$. 

\end{maintheorem}

Theorem~\ref{THM:soft_version} is a satisfactory qualitative statement.
It,  however,  does not explain in which sense the obtained limit $u$ satisfies a limit equation and it does not give any description of the obtained limit. 
This will be the purpose of the next two subsections.

\begin{remark} \rm 

It is worthwhile to note  that in a similar discussion for the KdV equation, 
one can show convergence of the sequence of regularized solutions 
for {\it any} regularization of the white noise initial data. This is because  local  well-posedness analysis in \cite{KT,OH5} is purely deterministic. 
Furthermore,  renormalization is not necessary for the KdV equation. 
It would be of interest  to investigate whether
 the result of Theorem \ref{THM:soft_version}
 holds for a more general class of regularizations of the white noise than 
 those given by mollification \eqref{smooth1}.

\end{remark}

%
%


\subsection{Renormalized equation}
We now derive the equation satisfied by the limiting distribution  derived in Theorem~\ref{THM:soft_version}.  Given a  global solution $u \in C(\R;  L^2(\T))$ to \eqref{NLS1}, 
we define the following invertible gauge transform:
\begin{equation}
u(t) \longmapsto
\mathcal{G}(u)(t) : = e^{2 i t \fint  |u(x, t)|^2 dx} u(t),
\label{gauge0}
\end{equation}

\noi
where $ \fint  f(x) dx := \frac{1}{2\pi} \int_\T f(x) dx$ 
denotes integration with respect to the normalized Lebesgue measure $(2\pi)^{-1} dx$ on $\T$.
A direct computation  with the mass conservation shows that the gauged function,  which we still denote by $u$,  solves the following renormalized 4NLS:
\begin{equation}\label{NLS2} 
\textstyle		i \dt u  =  \dx^4 u + \Big( |u|^2 -2  \fint |u|^2 dx\Big) u . 
\end{equation}

\noi
Note that the gauge transform $\GG$ is invertible. In particular,  
we can freely convert solutions to \eqref{NLS1} into solutions to \eqref{NLS2} and vice versa as long as they are in $C(\R; L^2(\T))$.
Clearly, the definition \eqref{gauge0}  does not make sense outside $L^2(\T)$ (in space)
and hence the original 4NLS \eqref{NLS1} and the renormalized 4NLS \eqref{NLS2}
are no longer equivalent outside $L^2(\T)$.
As it turns out, 
 the renormalized equation \eqref{NLS2} is the one satisfied 
 by the limiting distribution~$u$ appearing in the
statement of Theorem~\ref{THM:soft_version}.

Just like the original 4NLS \eqref{NLS1},  the $L^4$-Strichartz estimate \eqref{L4a} along with the mass conservation  yields global well-posedness of  the renormalized   4NLS \eqref{NLS2} in $L^2(\T)$.
The important point is that  the renormalization  removes a certain singular component from the cubic nonlinearity;
see \eqref{nonlin1} and \eqref{nonlin2} below.
This allows us to study well-posedness
of the renormalized 4NLS \eqref{NLS2} in negative Sobolev spaces.
In recent papers \cite{Kwak, OW}, 
the renormalized 4NLS \eqref{NLS2} was shown to be locally well-posed in $H^s(\T)$
for $s \geq - \frac 13$ and globally well-posed for $s > -\frac 13$.
Note that the white noise  in \eqref{grand_publique}
lies almost surely in  $H^s(\T)\setminus H^{-\frac 12}(\T)$, $s < -\frac{1}{2}$, 
which is beyond the scope of the known deterministic well-posedness results in~\cite{Kwak, OW}.
For this reason,  the main part of our analysis is devoted to the 
{\it probabilistic} construction of  local-in-time and global-in-time solutions to \eqref{NLS2} with the white noise as initial data.

Note that the renormalization of the nonlinearity in \eqref{NLS2} is canonical in the Euclidean quantum field theory (see, for example, \cite{Simon}).\footnote{To be precise, it is an equivalent formulation 
to the Wick renormalization in handling rough Gaussian initial data.}
This formulation 
first appeared  in the work of Bourgain \cite{BO96}
for studying  the invariant Gibbs measure for the defocusing cubic NLS on $\T^2$.  
See \cite{CO, OS, GO, OTh} for more discussion in the context of the (usual) nonlinear Schr\"odinger equations.
See also Remark \ref{REM:Wick} below.

\subsection{Statements of the well-posedness results}

In the following, we
consider the Cauchy problem for the renormalized 4NLS
\eqref{NLS2} with Gaussian random data in a more general form than  \eqref{grand_publique}. 
For this purpose,  we introduce a family of mean-zero Gaussian measures on 
periodic distributions on $\T$.
Given $\al \in \R$,  consider the  Gaussian measure $\mu_\al$
with formal density: 
\begin{align}
 d \mu_\al   = Z_\al ^{-1} e^{-\frac 12 \| u\|_{H^\al }^2} du  = Z_\al ^{-1} \prod_{n \in \Z} e^{-\frac 12 \jb{n}^{2\al } |\ft u_n|^2}   d\ft u_n .
\label{gauss0}
 \end{align}

\noi
We can indeed view  $\mu_\al $ as  the induced probability measure
under the map $\Xi_\al$ given by
\begin{align}
\Xi_\al: \o \in \O \longmapsto
\Xi_\al(\o)(x): =  \sum_{n \in \Z} \frac{g_n(\o)}{\jb{n}^\al}e^{inx} \in \mathcal D'(\T), 
\label{gauss1}
 \end{align}

\noi
where  $\jb{\,\cdot\,} = (1+|\cdot|^2)^\frac{1}{2}$ and  $\{ g_n \}_{n \in \Z}$ is a sequence of independent standard\footnote 
{By convention, we set  $\text{Var}(g_n) = 1$.} 
complex-valued Gaussian random variables on a probability space $(\O, \mathcal{F}, P)$.
 An easy computation shows that  $\Xi_\al$ in \eqref{gauss1} lies  in $H^{s}(\T)$ for  
\begin{equation}\label{gauss0a}
s < \al -\frac 12
\end{equation}
but not in $H^{\al-\frac 12}(\T)$ almost surely.  In particular, $\mu_\al$ is a Gaussian measure on $H^s(\T)$ and   the triplet $(H^\al, H^s, \mu_\al)$ forms an abstract Wiener space,  provided that $(\al, s)$ satisfies~\eqref{gauss0a}. 
For more details, see \cite{GROSS, Kuo2}.
When $\al = 0$,  
the random Fourier series~\eqref{gauss1} reduces
to that in~\eqref{grand_publique}
and hence 
 the Gaussian measure $\mu_0$ in \eqref{gauss0} corresponds to the white noise
measure. 

Our first step is to construct local-in-time dynamics for 
the renormalized 4NLS \eqref{NLS2} almost surely  with respect to the random initial data of the form:
\begin{align}
 u_0^\o(x)  = \sum_{n \in \Z} \frac{g_n(\o)}{\jb{n}^\al}e^{inx}
\label{IV1}
 \end{align}

\noi
with $\al  \ge 0$. 
For this purpose, we first introduce the following nonlinear operator $Z$ (of infinite degree)
by setting
\begin{align}
\label{z}
Z (f) (t):= \sum_{n\in \Z} 
e^{i(nx - n^4 t)}
\sum_{k = 0}^\infty \frac{(it)^k}{k!}  |\ft f(n)|^{2k}\ft f(n), 
\end{align}

\noi
a priori  defined 
 for smooth functions  $f = \sum_{n \in \Z} \ft f(n) e^{inx}$ on $\T$.
The following theorem addresses almost sure local  well-posedness of 
the renormalized 4NLS \eqref{NLS2} for $\al \geq 0$.

\begin{maintheorem}[Almost sure local well-posedness]
\label{THM:LWP}
Let $\al \geq 0$.   
Then, the renormalized  cubic 4NLS \eqref{NLS2} on $\T$ is  locally well-posed almost surely with respect to the Gaussian measure $\mu_\al$.
More precisely,  there exist $C, c > 0$ such that for each sufficiently small $\dl >0$,  there exists a set $\Omega_\dl \subset \O$ with the following properties:

\smallskip

\begin{itemize}
\item[\textup{(i)}]

$P(\Omega_\dl^c) = \mu_\al \circ \Xi_\al (\Omega_\dl^c) < C e^{-\frac{1}{\dl^c}}$,
where $\mu_\al$ and $\Xi_\al$ are as in \eqref{gauss0} and \eqref{gauss1}.

\smallskip

\item[\textup{(ii)}]
 For each $\o \in \Omega_\dl$, there exists a \textup{(}unique\textup{)} solution $u$ to  \eqref{NLS2} 
with $u|_{t = 0} = u_0^\o$ given by the random Fourier series~\eqref{IV1}
in the class:
\begin{align}
z^\o + C([-\dl, \dl];L^2(\T))
\subset C([-\dl, \dl];H^{s}(\T)), 
\label{class1}
\end{align}

\noi
where $z^\o =  Z(u_0^\o)$ is as in \eqref{z}
and 
\textup{(i)} $s = 0 $ if $\al > \frac 12$
and 
\textup{(ii)} $s = \al - \frac{1}{2}-\eps$
for any  $\eps > 0$,  if $\al \leq \frac 12$.

\end{itemize}

\end{maintheorem}

In the next subsections, we discuss an outline of the proof
of Theorem \ref{THM:LWP}.

\begin{remark}\label{REM:uniq0}
\rm 
When $\al > \frac 12$, the random initial data $u_0^\o$ in \eqref{IV1}
belongs almost surely to $L^2(\T)$ and hence the deterministic uniqueness statements apply.
In particular, when  $\al > \frac 23$, 
one can easily modify  the argument in \cite{GuoKO}
to conclude that 
the solution to \eqref{NLS2} is almost surely unconditionally unique, 
namely, uniqueness holds in the entire $C([-\dl, \dl]; H^{\frac 16}(\T))$.
For $\frac 12 < \al \leq \frac 23$, 
the solution is almost surely conditionally unique.  Namely, 
uniqueness holds in an auxiliary function space (the $X^{0, b}$-space for some $b > \frac 12$ in this case) contained in 
$C([-\dl, \dl]; L^2(\T))$.
As for the uniqueness statements
for $0 \leq \al \leq \frac 12$,
  see Remark \ref{REM:unique1} for $0 < \al \leq \frac 12$
and Remark~\ref{REM:uniq2} for $\al = 0$.

\end{remark}

Theorem~\ref{THM:LWP} with $\alpha =0$ shows that 
the renormalized 4NLS 
\eqref{NLS2} is almost surely locally well-posed 
with the white noise in \eqref{grand_publique} as initial data. 
In constructing almost sure global-in-time dynamics, 
we adapt Bourgain's invariant measure argument
 \cite{BO94, BO96} to our setting.
 More precisely, 
we use invariance of the white noise measure under the finite-dimensional approximation of the 4NLS flow 
to obtain a uniform control on the solutions, 
and then apply  a PDE approximation argument
to extend the local solutions to \eqref{NLS2} obtained from Theorem~\ref{THM:LWP} to global ones. 
As a byproduct, we also obtain  invariance of the white noise under the resulting global flow of 
the renormalized 4NLS \eqref{NLS2}.

\begin{maintheorem}[Almost sure global well-posedness and invariance of the white noise]
\label{THM:GWP}

Let $\al = 0$.
Then, the renormalized  4NLS \eqref{NLS2} on $\T$ is globally well-posed almost surely 
with the random initial data $u_0^\o$ given by \eqref{IV1}.
 More precisely,  for almost every $\o\in \Omega$, 
 there exists a unique solution $u$ to \eqref{NLS2} with $u|_{t = 0} = u_0^\o$, satisfying
\begin{align*}
u \in z^\o   + C(\R; L^{ 2}(\T)) \subset C(\R; H^{-\frac12 -\eps}(\T))
\end{align*}

\noi
for any $\eps > 0$, 
where $z^\o = Z(u_0^\o)$.
Furthermore, the white noise measure $\mu_0$  is invariant under the flow.
\end{maintheorem}

\begin{remark}
\rm

When $\al > \frac 16$, 
the deterministic global well-posedness \cite{OW} of 
the renormalized 4NLS \eqref{NLS2} in  $H^s(\T)$, $s > -\frac 13$,  implies
almost sure global well-posedness of \eqref{NLS2}
with the random initial data $u_0^\o$ in \eqref{IV1}
since the random initial data $u_0^\o$ almost surely belongs to $H^s(\T)$ for some $s >  - \frac 13$.

The proof of Theorem \ref{THM:GWP}
heavily depends on (formal) invariance of the white noise measure
and hence is not applicable
for the case $\al \in (0, \frac{1}{6}]$.
In \cite{CO}, Colliander and the first author 
adapted Bourgain's high-low decomposition method \cite{BO98}
to prove almost sure global well-posedness
of the renormalized NLS (with the second order dispersion)
with the random initial data of the form \eqref{IV1} below $L^2(\T)$
(without relying on any invariant measure).
The same approach is expected to  yield
almost sure global well-posedness of the renormalized 4NLS~\eqref{NLS2}
for some range  of $\al  \in (0, \frac{1}{6}]$.
 We do not pursue this analysis here.

\end{remark}

\begin{remark}
\rm
The solution $u$ constructed in Theorems \ref{THM:LWP} and \ref{THM:GWP}
 has a structure: 
 \[\text{$u =$ random {\it nonlinear} term $+$ smoother term}.\]
 
 \noi
 See \eqref{expxx}.
 This  is quite different from the standard probabilistic well-posedness results
 as in  \cite{BO96,BT1}, 
 where a solution $u$ has the structure: 
 \begin{align}
 \text{$u =$ random linear term $+$ smoother term}.
 \label{DD}
 \end{align}

\noi
In the field of stochastic PDEs, 
a well-posedness argument based on 
the decomposition~\eqref{DD}
is usually referred to 
as the Da Prato-Debussche trick.
When the decomposition~\eqref{DD} is not sufficient, 
one may try to write
a solution as the sum of  {\it finitely many} stochastic terms
plus a smoother remainder.
See for example \cite{GIP, Hairer}.

In the context of nonlinear dispersive PDEs, 
there are recent works
\cite{BOP3, OPT}, 
where  a solution theory was built,  based
on the decomposition 
of a solution as the sum of   finitely many stochastic terms
plus a smoother remainder.
A remarkable new feature of the decomposition
used in Theorems~\ref{THM:LWP} and~\ref{THM:GWP}
is that the series expansion \eqref{z} 
 for $ Z(u_0^\o)$ 
consists
not only  of the free solution (i.e.~$k=0$ in~\eqref{z}) but also of 
{\it infinitely many} higher order corrections 
terms $k \geq 1$.
As a consequence, 
$z^\o = Z(u_0^\o)$ depends on arbitrarily high powers of Gaussian random variables
and hence it does {\it not} belong to Wiener chaoses $\H_{\leq k}$,
defined in \eqref{Wiener}, of any finite order.
See also Remark \ref{REM:infinite}.

\end{remark}

\begin{remark}\label{REM:scattering}
\rm

A decomposition such as \eqref{DD}
is not only useful in establishing well-posedness of a given equation, but also provides a finer regularity description of a solution thus obtained.
For example, 
The decomposition \eqref{DD} states that 
 in the high frequency regime (i.e.~at small spatial scales on the physical side), 
 the dynamics is essentially governed by that of the random linear solution.
See also Remark \ref{REM:infinite}\,(ii).
In  \cite[Page~62]{BOPCMI}, Bourgain made an ``analogy'' of the decomposition \eqref{DD}
to  {\it scattering}
  (i.e.~a nonlinear solution behaving like a linear solution asymptotically as $t \to \pm \infty$) 
by saying 
``This property [namely the decomposition~\eqref{DD}] reminds of ``scattering" occurring in certain dispersive models"
in the sense that 
in both
the decomposition \eqref{DD}
and scattering, the dominant part of dynamics is given by the linear dynamics.

In our solution theory, we have the decomposition
\[\text{$u = z^\o +$ smoother term,}\] 

\noi
where $z^\o = Z(u_0^\o)$.
Namely, the dominant part is {\it nonlinear} (with an explicit structure).
In this context, one may wish to say that the results of Theorems~\ref{THM:LWP} and~\ref{THM:GWP}
  remind of {\it modified scattering} occurring in certain dispersive models \cite{Ozawa, HN, HPTV},
where the asymptotic dominant dynamics  is given not by a linear dynamics
but by a certain nonlinear dynamics.
See Remark~\ref{mod_scat_det} below for more details on this analogy.  

\end{remark}

\begin{remark}\label{REM:Wick}\rm
Instead of the renormalized 4NLS \eqref{NLS2}, 
one may  work with  the 
Wick renormalization
to study the same problem.
Disadvantage
for this approach is that  there is no equation for the limiting dynamics.
The limit $u$ of smooth approximating solutions
would formally ``satisfy''
 \begin{equation}
	i \dt u  =  \dx^4 u +  |u|^2 u - \infty \cdot  u . 
\label{NLS2a}
\end{equation}

\noi
This is in sharp contrast with the case of the renormalized 4NLS \eqref{NLS2}, 
where the renormalized nonlinearity has a well defined meaning
as a cubic operator, defined a priori on smooth functions.
 See \eqref{nonlin1} and \eqref{nonlin2}.
 Lastly, we point out that if the Gaussian measure $\mu_\al$ in~\eqref{gauss0}
 were invariant, then one could show that the
 renormalized 4NLS \eqref{NLS2}
 is equivalent to the Wick ordered 4NLS \eqref{NLS2a}
 in a suitable limiting sense, 
 provided that 
$\al > \frac 14$.
See Section~3 in~\cite{OS}.
Unfortunately, such invariance is true only for $\al = 0$.
\end{remark}

\subsection{Outline of the well-posedness argument}
When $\al > \frac 12$,  it follows from \eqref{gauss0a} that our random  initial data $u_0^\o$ defined 
in \eqref{IV1} belongs to $L^2(\T)$ almost surely.
Hence,  the aforementioned deterministic global well-posedness of \eqref{NLS2} in $L^2(\T)$
implies Theorem \ref{THM:LWP} in this case.
Therefore,  we focus on the case $0\leq \al \leq \frac 12$ in the following.

 When $0\leq \al \leq \frac 12$, 
the random   initial data $u_0^\o$ in \eqref{IV1} lies  strictly in negative Sobolev spaces almost surely.
In view of the failure of the local uniform continuity of the solution map  in these spaces (see \cite{CO, OTz1}),  it is non-trivial to construct solutions to \eqref{NLS2} in negative Sobolev spaces
since a straightforward contraction argument fails in this regime.
For $\al > \frac 16$, the random initial data $u_0^\o$ in \eqref{IV1} 
almost surely belongs to $H^s(\T)$ for some $s > -\frac 13$ and hence
the global well-posedness in \cite{OW}  based on a more robust energy method is  applicable
to conclude Theorem \ref{THM:LWP}.
In the following,  however, 
we present a uniform approach 
to construct local-in-time solutions in a probabilistic manner for $0 \leq \al \leq \frac 12$
by making use of randomness of the initial data $u_0^\o$ in \eqref{IV1}.

By writing \eqref{NLS2} in the Duhamel formulation, we have
\begin{align}
u(t) = S(t) u_0^\o - i \int_0^t S(t - t') \NN(u)(t') dt',
\label{NLS3}
\end{align}

\noi
where $S(t) = e^{-it\dx^4}$ denotes the linear propagator
and 
\begin{align}
\NN (u) = \bigg(|u|^2 -2 \fint |u|^2 dx\bigg) u.
\label{nonlin0}
\end{align}
Next, we make an important decomposition of the nonlinearity $\NN(u)$
into resonant and non-resonant parts. Namely,  define trilinear operators 
$\mathcal{N}_1$ and $\mathcal{N}_2$ by setting
\begin{align}
\label{nonlin1}
& \mathcal{N}_1(u_1, u_2, u_3) (x, t) 
: = \sum_{n \in \Z}
\sum_{\G(n)}
 \ft{u}_1(n_1, t)\cj{\ft{u}_2(n_2, t)}\ft{u}_3(n_3, t) e^{i(n_1-n_2+n_3)x}, \\
\label{nonlin2}
& \mathcal{N}_2(u_1, u_2, u_3) (x, t) : = - \sum_n \ft{u}_1(n, t)\cj{\ft{u}_2(n, t)}\ft{u}_3(n, t) e^{inx}, 
\end{align}

\noi
where $\G(n)$ denotes the hyperplane:
\begin{align}
\G(n) : = \big\{(n_1, n_2, n_3) \in \Z^3:\,  n = n_1 - n_2 + n_3 \text{ and }  n_1, n_3 \ne n\big\}.
\label{Gam0}
\end{align}

\noi
When all the arguments coincide, 
we simply  write $\NN_k(u) =\NN_k(u, u, u)$, $k = 1, 2$.
The term $\NN_1(u)$ denotes
the non-resonant part of the renormalized nonlinearity $\NN(u)$, 
while $\NN_2(u)$ denotes the resonant part.
Then, the renormalized  nonlinearity $\mathcal{N}(u) $ 
can be written as 
\begin{align*}
 \mathcal{N}(u) 
& = \mathcal{N}_1(u) + \mathcal{N}_2(u). 
\end{align*}

Let us first go over  the basic idea of the probabilistic local well-posedness, as developed for instance in  \cite{BO96, BT1, Th1, CO, NS}.  See also \cite{McKean}.
This argument is based on the following first order expansion:
\begin{align}
 u = z_1^\o + v,
 \label{exp1}
\end{align}

\noi
where $z_1^\o$ denotes the random linear solution defined by 
\begin{align}
z_1^\o(t) := S(t) u_0^\o.
\label{lin1}
\end{align}

\noi
By rewriting  \eqref{NLS3} as a  fixed point problem
for the residual term~$v: = u - z_1^\o$, we obtain
the following perturbed renormalized  4NLS:
\begin{align}
v(t) =  - i \int_0^t S(t - t') \NN(v + z_1^\o)(t') dt'.
\label{NLS4}
\end{align}

\noi
Then, the main aim is to   solve this fixed point problem
for $v$ in  $L^2(\T)$,\footnote
{Strictly speaking, we need to consider 
the fixed point problem \eqref{NLS4} in some appropriate function space $X_\dl \subset C([-\dl, \dl]; L^2(\T))$.
For simplicity, however, we only discuss the spatial regularity
and suppress its time dependence.
A similar comment applies in the following.
In particular, in discussing spatial regularity of a space-time distribution, 
we may suppress its time dependence.}
where the unperturbed equation \eqref{NLS2}
is deterministically well-posed by a simple contraction argument.
In particular, it is crucial to 
make use of 
probabilistic tools 
(for example, see Subsection~\ref{SUBSEC:prob})
and 
show 
 that the perturbation
$\NN(v + z_1^\o)-\NN(v )$
is {\it smoother} than the random linear solution $z_1^\o$
and lies in $L^2(\T)$ for each $t$.
When $\al > \frac 16$, this can be indeed achieved 
and we can show that for each small $\dl > 0$, there exists $\O_\dl \subset \O$
with $P(\O_\dl^c) < Ce^{-\frac{1}{\dl^c}}$
such that 
for each $\o \in \O_\dl$, there exists
a solution $u = z_1^\o+ v$ to the renormalized  4NLS \eqref{NLS2}  in the class:
\begin{align*}
z_1^\o + C([-\dl, \dl]; L^2(\T))
\subset  C([-\dl, \dl]; H^s(\T)),
\end{align*}

\noi
for $s < \al - \frac 12$.
The most singular contribution 
on the right-hand side of \eqref{NLS4} is given by 
\begin{align}
 z_3^\o(t) := - i \int_0^t S(t - t') \NN_2( z_1^\o)(t') dt'
  =  i t \sum_{n\in \Z} \frac{|g_n|^2 g_n}{\jb{n}^{3\al}}e^{i(nx - n^4 t)}
\label{lin2}
\end{align}

\noi
where $\NN_2$ is as in  \eqref{nonlin2},   denoting the resonant interaction.
This resonant cubic\footnote
{Namely, $z_3^\o$ in \eqref{lin2}
is trilinear in the random initial data.}
term is responsible
for the restriction $\al > \frac 16$.
It is easy to see that 
$z_3^\o(t)$ lies in   $H^s(\T) \setminus H^{3\al - \frac 12}(\T)$ 
almost surely 
for
\begin{align*}
s < 3\al - \frac 12.
\end{align*}

\noi
In particular, 
when $\alpha>\frac 16$, 
the $L^2$-deterministic well-posedness theory 
 (via a contraction argument) becomes available for solving 
 the perturbed equation \eqref{NLS4}. 
As mentioned above, 
the case  $\al > \frac 16$
is also covered by 
the deterministic well-posedness in~\cite{Kwak, OW} (based on a more robust energy method)
and thus our main goal in the following is to treat lower values of $\al$.

\begin{remark}\rm
This argument is basically the Da\,Prato-Debussche trick in the context of stochastic PDEs 
\cite{DP1, DP2}, where the random linear solution is replaced
by the  solution to a linear stochastic PDE. See \cite{HairerICM} for a concise discussion on 
the Da Prato-Debussche trick. 
It is worthwhile to point out that 
the paper \cite{McKean, BO96} by McKean and Bourgain 
 precede  \cite{DP1, DP2}.
\end{remark}


According to the  discussion above, 
the basic probabilistic argument based on the first order expansion \eqref{exp1}
does not work 
for our problem when $\al \leq \frac 16$ because  the second order term $z_3^\o$ does not belong to  $L^2(\T)$
almost surely if $\al \leq \frac 16$.
See also Case (b) in Subsection~4.2 of \cite{CO}.
This shows that we can not solve the fixed point problem \eqref{NLS4}
in $L^2(\T)$ when $\al \leq \frac 16$.

A natural next step would be to 
consider the following second order expansion:
\[u = z_1^\o + z_3^\o +  v\]

\noi
for a solution $u$ to \eqref{NLS2} 
and study the equation satisfied by the residual term $v := u - z_1^\o - z_3^\o $:
\begin{equation*}
	\begin{cases}
i \dt v  =  \dx^4 v +
\big[\NN(v + z_1^\o + z_3^\o) - \NN_2(z_1^\o)\big]\\
v|_{t= 0} = 0.
	\end{cases}
\end{equation*}

\noi
Namely, we consider the  following fixed point problem:
\begin{align}
v(t) = 
 - i \int_0^t S(t - t') \big[\NN(v + z_1^\o + z_3^\o) - \NN_2(z_1^\o)\big] (t')dt'.
\label{NLS6}
\end{align}

\noi
Note that the worst contribution $z_3^\o$ in the first step
coming from the resonant interaction $\NN_2(z_1^\o)$
is now eliminated.
We can then perform case-by-case nonlinear analysis
on $\NN_k(u_1, u_2, u_3)$, $k = 1, 2$, in the spirit of \cite{BO96, CO}, 
where each $u_j$ can be $z_1^\o$, $z_3^\o$, or the smoother unknown function $v$ 
except for the case $u_1 = u_2 =u_3 = z_1^\o$ with  $k = 2$. 
This allows us to 
show that 
the fixed point problem \eqref{NLS6} for the residual term $v$ is almost surely locally well-posed in $L^2(\T)$,
provided that $\al > \frac 1{10}$. 
Recalling  that $z_1^\o, z_3^\o \in C(\R; H^s(\T))$
for  $s $ satisfying \eqref{gauss0a}, we obtain 
a solution $u = z_1+z_3 + v$ to the renormalized 4NLS~\eqref{NLS2} 
  in the class:
\begin{align*}
z_1^\o + z_3^\o + C([-\dl, \dl]; L^2(\T))
\subset  C([-\dl, \dl]; H^s(\T))
\end{align*}

\noi
almost surely, 
for $s < \al - \frac 12$.

In this second step, the restriction $\al > \frac 1 {10}$ 
comes from the  following resonant quintic
term in \eqref{NLS6}: 
\begin{align}
 z_5^\o(t) : \! & = - i \sum_{\substack{j_1, j_2, j_3 \in 2\N - 1\\j_1 + j_2 + j_3 = 5}}
 \int_0^t S(t - t') \NN_2( z_{j_1}^\o, z_{j_2}^\o, z_{j_3}^\o)(t') dt'  \notag \\
&     =  - \frac{ t^2}{2} \sum_{n\in \Z} \frac{|g_n|^4 g_n}{\jb{n}^{5\al}}e^{i(nx - n^4 t)}.
\label{lin3}
\end{align}

\noi
Given $t \in \R$,  it is easy to see that
$z_5^\o(t)$ lies in   $H^s(\T) \setminus H^{5\al - \frac 12}(\T)$ 
almost surely 
for 
\begin{align*}
s < 5\al - \frac 12.
\end{align*}

\noi
In particular, $z_5^\o(t)$ does not lie in $L^2(\T)$
almost surely if $\al \leq \frac 1{10}$.

One can repeat this process in an obvious manner.
Namely,  consider the following third order expansion:
\[ u = z_1^\o + z_3^\o + z_5^\o +  v\]

\noi
for a solution $u$ to \eqref{NLS3} 
and study 
 the fixed point problem for $v = u - z_1^\o - z_3^\o - z_5^\o$.
From the discussion above, 
we see that the limitation comes from the resonant septic term, 
yielding the restriction of $\al > \frac 1 {14}$.

In general, in the  $k$th step,  
we could write a solution $u$ to \eqref{NLS3} as 
\begin{align}
u =  v + \sum_{j = 1} ^kz_{2j-1}^\o
\label{expk}
\end{align}

\noi
and consider the fixed point problem for $v = u - \sum_{j = 1} ^kz_{2j-1}^\o$.
Here, $z_{2j-1}$
denotes 
 the following resonant $(2j-1)$-linear term (in the random initial data):
\begin{align}
 z_{2j-1}^\o(t) : = - i \sum_{\substack{j_1, j_2, j_3 \in 2\N - 1\\j_1 + j_2 + j_3 = 2j-1}}
 \int_0^t S(t - t') \NN_2( z_{j_1}^\o, z_{j_2}^\o, z_{j_3}^\o)(t') dt'.
 \label{lin4}
\end{align}

\noi
Proceeding as before, it is easy to see that 
the limitation
in this $k$th step  comes from 
$z_{2k+1}^\o$
yielding the restriction of 
\begin{align}
\al > \frac 1 {2(2k+1)}
\label{al_k}
\end{align}

\noi
which is needed to guarantee that $z_{2k+1}^\o(t)$ belongs almost surely to $L^2(\T)$.

The restriction \eqref{al_k} shows that, 
in order to treat the $\al = 0$ case, we {\it at least} need an infinite iteration 
of this procedure.
Furthermore, 
the argument based on the $k$th order expansion \eqref{expk}
leads to the following equation 
for the 
 the residual term
$v = u - \sum_{j = 1} ^kz_{2j-1}^\o$:
\begin{equation*}
\begin{cases}
\displaystyle
 i \dt v =  \dx^4 v+  \NN\bigg(v +\sum_{j= 1}^{k} z_{2j-1}^\o\bigg) 
- 
 \sum_{\substack{j_1 + j_2 + j_3 \in \{3, 5, \dots, 2k-1\} 
\\j_1, j_2, j_3 \in \{1, 3, \dots, 2k-3\}}} 
\NN_2( z_{j_1}^\o,  z_{j_2}^\o, z_{j_3}^\o)\\
v|_{t = 0} = 0.
 \end{cases}
\end{equation*}

\noi
In particular, we need to carry out
 the following case-by-case nonlinear analysis on 
\begin{align*}
& \NN_\l  (u_1, u_2,  u_3), \quad  \l = 1, 2, 
\end{align*}

\noi
where each $u_i$,  $i = 1, 2, 3$, can be 
either the smoother unknown function
$v$ or $ z_j^\o$ for some $j \in \{1, 3, \dots, 2k-1\}$
 such that 
it is not of the form $\NN_2 (z_{j_1}, z_{j_2}, z_{j_3}) $ with 
$j_1 + j_2 + j_3  
\in  \{3, 5, \dots, 2k-1\}$.
In general, it could be a cumbersome task to 
carry out this case-by-case analysis
due to 
the increasing number of combinations.
In the next subsection, we will describe an approach
to overcome this issue.

\begin{remark}\rm
In \cite{BOP3}, the first author with B\'enyi and Pocovnicu
studied the cubic NLS on $\R^3$ with random initial data
based on a higher order expansion (of order $k$), analogous to~\eqref{expk}.
In order to avoid a combinatorial nightmare
in relevant case-by-case analysis
for high values of $k$,
the authors introduced a modified expansion of order $k$, which simplified the relevant analysis 
in a significant manner.
We point out that the analysis in \cite{BOP3} is significantly simpler than
that in the current paper, 
since (i) the random data considered in \cite{BOP3} 
are of positive regularities and (ii) the refinement of the bilinear Strichartz estimates
\cite{BO98, OzawaT} are available on the Euclidean space.
We also mention a recent work \cite{OPT}
on the probabilistic local well-posedness of the 
three-dimensional cubic nonlinear wave equation in negative Sobolev spaces, 
where the main analysis is based on the second order expansion.

\end{remark}

\subsection{The $\alpha > 0$ case}
\label{SUBSEC:a>0}

In this subsection, we describe an outline
of the proof of Theorem~\ref{THM:LWP}
for the $\al > 0$ case.
In the next subsection, we discuss additional ingredients required
to treat the $\al = 0 $ case.

In view of the restriction \eqref{al_k}, 
we need to iterate  indefinitely the procedure described above
 in order to treat  arbitrary $\al > 0$.
For this purpose, we define $z^\o$ by 
\begin{align}
z^\o = \sum_{j = 1} ^\infty z_{2j-1}^\o.
\label{expx}
\end{align}

\noi
Then, 
 from \eqref{lin1}, \eqref{lin2}, \eqref{lin3}, and \eqref{lin4},
 we see that $z^\o$ defined in \eqref{expx} is nothing but a power series expansion
 of a  solution to the following {\it resonant} 4NLS:
\begin{equation}
\begin{cases}
i \dt z^\o  =  \dx^4 z^\o 
+ \NN_2(z^\o)\\
z^\o|_{t= 0} = u_0^\o,
\end{cases}
\label{ZNLS1} 
\end{equation}

\noi
where $u_0^\o$ is the random initial data defined in \eqref{IV1}.
By letting $\z(t) = S(-t) z^\o(t)$, we see that 
$\ft \z_n(t) = \ft \z(n, t)$ satisfies the following ODE:
\begin{equation}
\begin{cases}
i \dt \ft \z_n  =   -|\ft \z_n|^2\ft \z_n\\
\ft \z_n|_{t= 0} = \frac{g_n}{\jb{n}^\al},
\end{cases}
\label{ZNLS2} 
\end{equation}

\noi
for each $n \in \Z$.
By the explicit formula of solutions to \eqref{ZNLS2}, 
we have 
\begin{align}
\ft \z_n(t) = e^{it |\ft \z_n(0)|^2} \ft \z_n(0).
\label{Zlin0}
\end{align}

\noi
Hence, we can express $z^\o$ as 
\begin{align}
z^\o(t) & = \sum_{n\in \Z} e^{i(nx - n^4 t)} e^{it \frac{|g_n|^2}{\jb{n}^{2\al}}}  \frac{g_n}{\jb{n}^{\al}}.
\label{Zlin1}
\end{align}

\noi
By expanding in a power series, we obtain
\begin{align}
z^\o(t) &  = \sum_{n\in \Z} 
e^{i(nx - n^4 t)}
\sum_{k = 0}^\infty \frac{(it)^k}{k!}  \frac{|g_n|^{2k}g_n}{\jb{n}^{(2k+1)\al}}\,.
\label{Zlin1a}
\end{align}

\noi
By comparing \eqref{z} and \eqref{Zlin1a} with \eqref{IV1}, we obtain
\[z^\o = Z(u_0^\o).\]

\noi
Note that,
unlike the random linear solution $z_1^\o$ in \eqref{lin1} and 
 other lower order terms $z^\o_{2j - 1}$ in~\eqref{lin4}, 
the random resonant solution 
$z^\o$ depends on arbitrarily high powers of Gaussian random variables
and hence it does {\it not} belong to Wiener chaoses of any finite order.
Nonetheless, 
the formula \eqref{Zlin1} shows that 
$z^\o$ has a particular simple structure, 
 allowing us to study its  regularity properties; 
see  Lemmas~\ref{LEM:Z} and ~\ref{LEM:Z1} below.
In carrying out analysis on the random resonant solution $z^\o$ 
involving the $X^{s, b}$-spaces, 
we instead need to make use of the series expansion  \eqref{Zlin1a}
and apply  Lemma~\ref{LEM:Z2} below for each $k$.

\begin{lemma}\label{LEM:Z}
Given $\al \in \R$, 
let $z^\o$ be as in \eqref{Zlin1}.
Then, $z^\o$ belongs
to $C(\R; H^s(\T))$
almost surely, provided that $s < \al - \frac 12$.
\end{lemma}

\begin{proof} 
Fix $\eps > 0$ sufficiently small such that 
\begin{align}
s + \eps < \al - \tfrac12.
\label{zintro1}
\end{align}

\noi
Lemma \ref{LEM:prob} below states that 
we have
\begin{align}
\sup_{n \in \Z} |g_n(\o)| \leq C(\o) \jb{n}^\eps
\label{zintro2}
\end{align}

\noi
for some almost surely finite constant $C(\o)> 0$.

For fixed $t \in \R$, let $\{t_j\}_{j = 1}^\infty$ be a sequence
converging to $t$.
Then, for each $n \in \Z$, 
it follows from \eqref{Zlin0} that 
$\ft{z^\o} (n, t_j)$ converges
to $\ft{z^\o} (n, t)$ almost surely  as $j \to \infty$.
Furthermore, from~\eqref{Zlin0} and~\eqref{zintro2}, we have
\[ 
\sup_{j \in \N} \jb{n}^s |\ft{z^\o} (n, t_j)|
+  \jb{n}^s  |\ft{z^\o} (n, t)| \leq 2 C(\o) \jb{n}^{s - \al + \eps},  
\]

\noi
where the right-hand side belongs to $\l^2(\Z)$ in view of \eqref{zintro1}.
Hence, the claim follows from the dominated convergence theorem.
\end{proof}

Now,  express  a solution $u$ to \eqref{NLS2} 
in the following {\it random-resonant\,/\,nonlinear decomposition}:
\begin{align}
u =   z^\o + v.
\label{expxx}
\end{align}

\noi
Then, the residual term  $v = u - z^\o$
satisfies
\begin{equation}
\begin{cases}
i \dt v  =  \dx^4 v +
\big[\NN(v + z^\o) - \NN_2(z^\o)\big]\\
v|_{t= 0} = 0.
\end{cases}
\label{ZNLS3} 
\end{equation}

\noi
By writing \eqref{ZNLS3} in the Duhamel formulation,   
we consider the  following fixed point problem:
\begin{align}
v(t) =  \G^\o v(t): =
 - i \int_0^t S(t - t') \big[\NN(v + z^\o ) - \NN_2(z^\o)\big](t')dt'.
\label{ZNLS4}
\end{align}

\noi
In this formulation, 
we successfully reduced the number of combinations;
we only need to study 
 $\NN_k(u_1, u_2, u_3)$, $k = 1, 2$, 
where each $u_j$ can be either the random resonant solution $z^\o$
 or the smoother unknown function $v$,  
except for the case $u_1 = u_2 =u_3 = z^\o$ with  $k = 2$. 
In Section \ref{SEC:LWP1}, 
 we perform the case-by-case nonlinear analysis
and 
show that the fixed point problem \eqref{ZNLS4} is almost surely locally well-posed in $L^2(\T)$
via the standard Fourier restriction norm method,
provided that $\al > 0$.

Lastly, Lemma \ref{LEM:Z} 
allows us to conclude that 
the solution $u = z^\o + v$ to the renormalized 4NLS \eqref{NLS2} lies  in the class:
\begin{align*}
z^\o + C([-\dl, \dl]; L^2(\T))
\subset  C([-\dl, \dl]; H^s(\T))
\end{align*}

\noi
almost surely.

\begin{remark}\label{REM:unique1}\rm
The probabilistic local well-posedness argument 
in \cite{BO96, BT1, Th1, CO} yields uniqueness of solutions
in a ball of radius $O(1)$ in a suitable (local-in-time) function  space  (such as the Strichartz spaces
or the $X^{s, b}$-spaces) centered at the random linear solution.
When $\al > 0$, the proof of Theorem \ref{THM:LWP}  yields uniqueness of solutions 
in the ball of radius $1$ in $X^{0, \frac 12+, \dl}$ centered at 
the random resonant solution $z^\o$.
\end{remark}

\begin{remark}\label{REM:infinite}\rm
(i) When $\al > 0$, 
the terms $z_{2j-1}^\o$ appearing in \eqref{expk} get smoother 
as $j$ increases
and hence only a finite number of expansion is needed.
Nonetheless, the random-resonant\,/\,nonlinear decomposition  \eqref{expxx} allows us to avoid
 a number of combinations in the relevant case-by-case analysis 
when $k \gg1 $.
When $\al = 0$, 
the terms $z_{2j-1}^\o$  in \eqref{expx} do {\it not} get smoother
and hence the infinite order expansion in \eqref{expx}  is necessary
in this case.

\smallskip

\noi
(ii) 
Let $\al > 0$.
In this case, the random-resonant\,/\,nonlinear decomposition 
\eqref{expxx}
with~\eqref{expx} allows us to write the solution $u$ as
\begin{align}
 u = z_1^\o + z_3^\o + \dots + z_{2k+1}^\o + v
 \label{expy}
\end{align}

\noi
for some $v \in C([-\dl, \dl]; L^2(\T))$, where $k$ is the smallest non-negative integer
such that~\eqref{al_k} holds.
The expansion \eqref{expy} provides a finer regularity 
description\footnote
{This regularity description can also be understood as the ``local'' (in space) description 
of the solution since the singular components of the solution become dominant 
 in small scales.}
 of the solution $u$  than the random-linear\,/\,nonlinear decomposition \eqref{exp1}.
As mentioned above, the terms in~\eqref{expx} do not get smoother
when $\al = 0$.
In this case, the solution $u$ can be written as
\[ u= z^\o + v\]

\noi
for some $v \in C([-\dl, \dl]; L^2(\T))$.
Namely, the dominant part of the dynamics in small scales
is indeed given by the random resonant solution $z^\o$ defined in \eqref{Zlin1}.

\end{remark}

\subsection{The $\al = 0$ case}
Next, let us discuss the $\al = 0$ case.
Namely, we consider the white noise initial data  \eqref{grand_publique}.
Unfortunately, the argument described above breaks down in this case.
As we see in Section \ref{SEC:LWP1}, 
 the worst interaction  comes from 
 the following {\it resonant nonlinear} terms
on the right-hand side of \eqref{ZNLS3}:
\begin{align*}
 \NN_2(v, z^\o,z^\o)+ 
\NN_2(z^\o,z^\o, v)
=  - 2 \F^{-1} \big[ |g_n|^2 \ft v(n)\big]
 \end{align*}

\noi
and
\begin{align*}
 \NN_2(z^\o, v, z^\o)
  = - \F^{-1} \Big[ e^{-2 i n^4 t} e^{2 i t |g_n|^2 } g_n^2  \, \cj{\ft v(n)}\Big].
\end{align*}

\noi
In order to weaken the effect of these terms, we introduce 
the following {\it random} gauge transform:
\begin{align}
\label{Ga}
\J^\o (u) (x, t) = \sum_{n\in \Z}  e^{inx-it|g_n(\o)|^2} \ft{u}(n, t) .
\end{align}

\noi
When $\alpha =0$, 
the solution $z^\o$ to the resonant 4NLS \eqref{ZNLS1} reads as 
\begin{align}
z^\o(x, t) = \sum_{n\in\Z} e^{i(nx -n^4t)}e^{it|g_n|^2} \ft{u_0^\o} (n).
\label{Zlin2}
\end{align}

\noi
The random gauge transform $\J^\o$ in \eqref{Ga} 
allows us to  filter out the random phase oscillations  appearing in~\eqref{Zlin2}.
This gauge transform is clearly invertible and leaves the $H^s$-norm invariant. 
If $u$ is a solution to the renormalized 4NLS \eqref{NLS2}, 
then
the gauged function  $w: = \J^\o(u) $ satisfies the following random equation:
\begin{equation}
\label{NLS2-3} 
\begin{cases}
i \dt w = \partial_x^4 w + \NN_1^\o(w) + \NN_{2}^\o(w) \\
w|_{t= 0} = u_0^\o.
\end{cases} 
\end{equation}

\noi
Here,  the first nonlinearity $\NN_1^\o(w)$ is defined by
\begin{equation}\label{NN1O}
\NN_1^\o(w) (x, t) : = \sum_{n\in \Z}  e^{inx} \sum_{\G(n)} 
e^{it\Psi^\o(\bar n)} \ft w (n_1, t) \cj{\ft w (n_2, t)} \ft w( n_3, t),
\end{equation}

\noi
where $\G(n)$ is as in \eqref{Gam0}
and $\Psi^\o(\bar n)$ denotes the random phase function:
\begin{align}
\label{Psi}
\Psi^\o (\bar n) : =  \Psi^\o(n_1,n_2,n_3, n) & = 
|g_{n_1}(\o)|^2 - |g_{n_2}(\o)|^2 + |g_{n_3}(\o)|^2 - |g_{n}(\o)|^2.
\end{align}

\noi
The second nonlinearity $\NN_{2}^\o(w) $ is defined by 
\begin{equation}\label{NN2O}
\NN_{2}^\o(w) (x, t) : = - \sum_{n\in \Z}  e^{inx}  \big[|\ft w (n, t) |^2 - |g_n(\o)|^2\big] \ft w (n, t)\,.
\end{equation}
As we can see, \eqref{NN1O} and \eqref{NN2O} are random versions of \eqref{nonlin1} and \eqref{nonlin2}. The main advantage of working with this gauged version of 
the renormalized 4NLS \eqref{NLS2} lies in the weaker resonant nonlinearity $\big[|\ft w(n)|^2 - |g_n(\o)|^2\big] \ft w (n)$,  which would be eliminated 
if  $\ft w (n)  = g_n$.  This observation turns out to be crucial in our later analysis.

The Duhamel formulation for the gauged solution $w$ is given by
\begin{align}
\label{NLS3-1}
w(t) = S(t)u_0^\o - i \int_0^t S(t-t') \big[\NN_1^\o(w) + \NN_{2}^\o(w)\big](t') dt'.
\end{align}

\noi
Now by setting $z_1^\o = S(t)u_0^\o$, 
we see that the residual term
\[
v = w - z_1^\o, 
\]

\noi
satisfies the following Duhamel formulation:
\begin{align}
\label{Duhamel}
v(t) = - i \int_0^t S(t-t') \big[\NN_1^\o(v+z_1^\o) + \NN_{2}^\o(v+z_1^\o)\big](t') dt'.
\end{align}

\noi
A naive approach would be to try to solve
the  fixed point problem  \eqref{Duhamel} by a contraction argument (namely, 
by the Picard iteration scheme) for  $v$ in $L^2(\T)$,
exploiting randomness.
It turns out, however, that this naive approach via a contraction argument does {\it not} work for our problem.
 In the following, 
by partially iterating the Duhamel formulation, 
we prove convergence  in $L^2(\T)$ of 
approximating smooth solutions
and  construct a solution to~\eqref{Duhamel} and hence to~\eqref{NLS2-3}. 
See Section \ref{SEC:LWP2} for more details. 
We establish 
 the crucial nonlinear estimates (Propositions~\ref{PROP:I1} and~\ref{PROP:I2})
by reducing them 
to boundedness properties
of certain random multilinear functionals
of the white noise, 
whose tail estimates are proved in Appendix \ref{SEC:A}.

\begin{remark}\rm
 As it will become clear from the analysis below, 
there is  room to extend our analysis to 
the fractional NLS with dispersion weaker than the fourth order dispersion.  
However, this would not introduce any new qualitative phenomenon as compared to the case of the fourth order dispersion and hence we only consider the fourth order NLS in this paper.
We also point out that the case of the standard NLS
(with the second order dispersion) 
is out of reach at this point.
See the introduction in \cite{FOW} for a discussion on the criticality of this problem
(in the context of the stochastic NLS with additive space-time white noise forcing).

\end{remark}

\begin{remark}\rm

(i)
In the deterministic setting, 
Takaoka-Tsutsumi \cite{TakT} implicitly used a gauge transform 
analogous to \eqref{Zlin2} 
in the low regularity study of the modified KdV equation
to weak the resonant interaction.
This  led them to work in the modified $X^{s, b}$-spaces.
See also \cite{NTT}.
In our case, the gauge transform $\J^\o$ is random 
and hence it leads to 
 the random $X^{s, b}$-spaces.
See Subsection \ref{SUBSEC:A.1}.
We also point out
the work \cite{OTTz}
on  the use of a gauge transform in the probabilistic context.

\smallskip

\noi
(ii)
In order to construct the dynamics for the $\al =  0$ case, 
we partially iterate the Duhamel formulation (of the gauged equation)
and establish convergence property of smooth approximating solutions.
See Section \ref{SEC:LWP2}.
This strategy is close in spirit to the work 
\cite{OH6, R}.
In the context of stochastic PDEs, 
such iteration of a Duhamel formulation appears
in the dispersive setting \cite{OH5, GKO}
and in the parabolic setting 
\cite{Hairer, CC, MW1}.
We also mention 
\cite{BO97, BB2, BB3, Bring} on the probabilistic construction
of solutions 
by establishing convergence  of smooth  solutions.
In particular, the recent approach by Bourgain-Bulut~\cite{BB2, BB3} relying
on the invariance of the truncated Gibbs measures even in the construction
of local solutions
works well for a power-type nonlinearity with positive regularity
but is not suitable to our problem at hand.
See \cite{BOP4} for a survey on this method.

\end{remark}

\subsection{Organization of the paper}

In Section~\ref{SEC:notations}, we introduce the basic notations and list some basic deterministic and probabilistic lemmas. 
In Section~\ref{SEC:LWP1},  
we present the proof of Theorem \ref{THM:LWP} for $\al > 0$.
The remaining part of the paper is devoted to handle the $\al = 0$ case.
In Section~\ref{SEC:LWP2},  we prove  Theorem \ref{THM:LWP},
by assuming two key nonlinear estimates (Propositions~\ref{PROP:I1} and~\ref{PROP:I2}).
 In Section~\ref{SEC:GWP}, 
we prove Theorem~\ref{THM:GWP} and then Theorem~\ref{THM:soft_version}.
 We present the proofs of 
Propositions~\ref{PROP:I1} and~\ref{PROP:I2}
in Sections~\ref{SEC:Non1}  and~\ref{SEC:Non2}.
Appendix~\ref{SEC:A}
contains
the proofs of some probabilistic lemmas.

\section{Notations and preliminaries}
\label{SEC:notations}

As  in the usual low regularity analysis of dispersive PDEs,
 an important ingredient will be the Fourier restriction norm method introduced in \cite{BO93}.
Given $s, b \in \R$, define  $X^{s, b}(\mathbb{T} \times \mathbb{R})$
as a completion of the test functions under the following norm:
\begin{equation} \label{Xsb}
\| u\|_{X^{s, b}(\mathbb{T} \times \mathbb{R})} = \|\jb{n}^s \jb{\tau +n^4}^b \ft{u}(n, \tau)\|_{\l^2_n L^2_\tau } , 
\end{equation}

\noi
\noi 
where $\jb{\, \cdot\, } = (1 + |\cdot|^2)^{\frac{1}{2}}$. Recall that $X^{s, b}$ embeds into 
$C(\R;  H^s(\T))$ for $b  >\frac{1}{2}$. 
Given a time interval $I = [a, b]$,  
we define the local-in-time version 
$X^{s, b}_I = X^{s, b}([a, b])$
 by setting
\begin{equation} \label{Xsb2}
 \|u\|_{X^{s, b }_I} = \inf \big\{ \|v \|_{X^{s, b}(\T \times \mathbb{R})}: {v|_{I} = u}\big\}.
\end{equation}

\noi
Note that $X^{s, b}_I$ is a Banach space.
When $I = [-\dl, \dl]$, 
we simply set $X^{s, b, \dl } = X^{s, b}_I$.
The local-in-time versions of other function spaces are defined analogously.

For simplicity, we often drop $2\pi$ in dealing with the Fourier transforms.  If a function $f$ is random, we may use the superscript $f^\o$ to show the dependence on $\o \in \O$.

Let $\eta \in C^\infty_c(\mathbb{R})$ be a smooth non-negative cutoff function supported on $[-2, 2]$ with $\eta \equiv 1$ on $[-1, 1]$ and set
\begin{align}
\eta_{_\dl}(t) =\eta(\dl^{-1}t)
\label{eta1}
\end{align}

\noi
for $\dl > 0$.
We also denote by   $\chi = \chi_{[-1, 1]}$   the characteristic function of the interval $[-1, 1]$
 and let $\chi_{_\dl}(t) =\chi(\dl^{-1}t) = \chi_{[-\dl, \dl]}(t)$.

Let $\Z_{\geq 0} := \Z \, \cap\,  [0, \infty)$.
Given a dyadic number $N \in 2^{\Z_{\geq 0}}$, 
 let $\P_N$
 be the (non-homogeneous) Littlewood-Paley projector
 onto the (spatial) frequencies $\{n \in \Z: |n|\sim N\}$
 such that 
 \[ f = \sum_{\substack{N\geq 1\\
 \text{dyadic}}}^\infty \P_N f.\]

\noi
Given a non-negative integer $N \in \Z_{\geq 0}$, 
we also define the Dirichlet projector $\pi_N$ 
onto the frequencies $\{|n|\leq N\}$ by setting
\begin{align}
\pi_N f(x)=
\sum_{ |n| \leq N}  \ft f(n)    e^{in x}.
\label{pi1}
\end{align}

\noi
Moreover, we set
\begin{align}
\pi_N^\perp  = \Id - \pi_N. 
\label{pi2}
\end{align}

\noi
By convention, we also set $\pi_{-1}^\perp = \Id$.

We use $c,$ $ C$ to denote various constants, usually depending only on $\al$ and $s$. If a constant depends on other quantities, we will make it explicit. 
For two quantities $A$ and $B$,  we use $A\lesssim B$ to denote an estimate of the form $A\leq CB$, where $C$ is a universal constant,  independent of particular realization of $A$ or $B$. 
Similarly, we use $A\sim B$ to denote $A\lesssim B$ and $B\lesssim A$ . 
The notation $A\ll B$ means $A \leq cB$ for some sufficiently small constant $c$. 
We also use the notation $a+$ (and $a-$) to denote $a + \eps$ (and $a - \eps$, respectively)
for arbitrarily small  $\eps >0$ (this notation is often used 
when there is an implicit constant  which diverges in the limit $\eps\to 0$).

\subsection{Deterministic tools}
Define the phase function $\Phi(\bar n)$ by 
\begin{align}
 	\Phi(\bar n) = \Phi(n_1, n_2, n_3, n) = n_1^4 - n_2^4 + n_3^4 - n^4.
\label{phi1}
\end{align}

\noi
Then, the phase function $\Phi(\bar n)$ admits the following factorization. See \cite{OTz1} for the proof.
\begin{lemma}\label{LEM:phase}
Let $n = n_1 - n_2 + n_3$.
Then, we have
\begin{align*}
\Phi(\bar n) =  (n_1 - n_2)(n_1-n) 
\big( n_1^2 +n_2^2 +n_3^2 +n^2 + 2(n_1 +n_3)^2\big).
\end{align*}
\end{lemma}

Recall that by restricting the $X^{s, b}$-spaces onto a small time interval $[-\dl, \dl]$,  
we can gain a small power of $\dl$ (at a slight loss in the modulation). 
\begin{lemma}
\label{LEM:timedecay} 
Let $s \in\R$ and $b < \frac{1}{2}$. Then, there exists $C = C(b) > 0$ such that  
\begin{align*}
\|\eta_{_\dl} (t) \cdot u\|_{X^{s, b}}
+ 
 \|\chi_{_\dl} (t) \cdot u\|_{X^{s, b}} &  \leq C \dl^{\frac{1}{2}-b-} \|u\|_{X^{s, \frac{1}{2}-}}.
\end{align*}
\end{lemma}

The proof of Lemma \ref{LEM:timedecay} is based on the following scaling property:
$\ft \eta_{_\dl}(\tau) = \dl \ft{\eta}(\dl\tau)$,
yielding
\begin{equation}
 \|\ft \eta_{_\dl}\|_{L^q_\tau} \sim \dl^\frac{q-1}{q} \|\ft{\eta}\|_{L^q_\tau} 
\lesssim \dl^\frac{q-1}{q}\, ,
\label{decay}
\end{equation}

\noi
for $q \geq 1$. 
See  \cite{CO} for details. 

Next, we collect the basic linear estimates (see  \cite{GTV}). 
\begin{lemma}\label{LEM:linear}
Let $s\in \R$.

\noi
\textup{(i)} Given $b \in \R$, 
 there exists $C = C(b)>0$ such that 
\begin{align*}
\|S(t) u_0\|_{X^{s,b,\dl}} \le C  \|u_0\|_{H^s}
\end{align*}

\noi
for any $0 < \dl \leq 1$.

\smallskip

\noi
\textup{(ii)} 
Given $b  >\frac 12$,  there exists $C = C(b)>0$ such that 
\begin{align*}
\bigg\| \int_0^t S(t-t') F(x,t') dt' \bigg\|_{X^{s,b,\dl}} \les  \|F\|_{X^{s,b-1,\dl}}
\end{align*}
 
\noi
for any $\dl > 0$.

\end{lemma}
The following periodic $L^4$-Strichartz estimate from  \cite{OTz1} also plays an important role:
\begin{equation}\label{L4} 
\|u\|_{L^4_{x, t}} \lesssim \|u\|_{X^{0, \frac{5}{16}}}. 
\end{equation}
Interpolating \eqref{L4} with $\|u\|_{L^2_{x, t}} = \|u\|_{X^{0, 0}}$, we have 
\begin{equation}\label{L3}
 \|u\|_{L^{3+}_{x, t}} \lesssim \|u\|_{X^{0, \frac{5}{24}+}}
 \qquad  \text{and} 
 \qquad  \|u\|_{L^{2+}_{x, t}} \lesssim \|u\|_{X^{0, 0+}}. 
\end{equation}

We also recall the following lemma on convolutions.
See \cite{GTV} for a proof.
\begin{lemma}\label{LEM:GTV}
Let $\al > \be \geq 0$ with $\al + \be > 1$.
Then, there exists $C>0$ such that 
\[ \int_\R \frac{1}{\jb{x - y}^\al \jb{y}^\be} dy 
\leq \frac{C}{\jb{x}^\g}
\]

\noi
for any $x \in \R$, 
where $\g$ is given by 
\[\g = \begin{cases}
\al + \be - 1, & \text{if } \al < 1, \\
 \be -\eps, & \text{if } \al = 1, \\
\be, & \text{if } \al >1
\end{cases}
\]

\noi
for any small $\eps > 0$.	
\end{lemma}

Lastly, we state two lemmas related to boundedness
properties of products in Sobolev spaces. 

\begin{lemma}\label{LEM:algebra}
Let $\eps > 0$.  Then, there exists $C =C(\eps) > 0$ such that 
\begin{align*}
\|fg\|_{H^{\frac12 -\eps}(\R)} 
\le C \|f\|_{H^{\frac12 +\eps}(\R)}
 \|g\|_{H^{\frac12 -\frac{\eps}2}(\R)}.
\end{align*}

\end{lemma}

Lemma \ref{LEM:algebra} 
easily follows from standard analysis
with Littlewood-Paley decompositions and Bernstein's inequality.
We omit details.

\begin{lemma}\label{LEM:bound}
Let $0 \leq b < \frac{1}{2}$.
Then, we have 
\[ \| \ind_{[0, T]} \cdot f \|_{H^b(\R)}  \les \| f\|_{H^b(\R)}, \]

\noi
uniformly in $T\geq 0$.
\end{lemma}

See \cite{DD1} for a classical  proof via an interpolation argument.
By Plancherel's identity, 
Lemma \ref{LEM:bound} also follows from 
the boundedness of the Hilbert transform (on the Fourier side)
with an $A_2$-weight $\jb{\tau}^{2b}$, 
$0 \leq b < \frac 12$.
See \cite{Gra}.

\subsection{Probabilistic estimates} \label{SUBSEC:prob}

Next, we state several probabilistic lemmas related to Gaussian random variables.
See also Appendix~\ref{SEC:A} 
for further lemmas.
In the following, $\{g_n\}_{n\in\mathbb{Z}}$ denotes a family of independent standard complex-valued Gaussian random variables
on a probability space $(\Omega, \mathcal{F}, P)$.

We first start by a well known fact (see for example
\cite{OHDIE, CO}).

\begin{lemma}\label{LEM:prob}
Let $\eps>0$. 
Then, there exist $c, C>0$ such that 
\begin{equation*}
P\Big( \sup_{n\in \Z} |g_n(\o)| >  K \jb{n}^\eps \Big)
< C e^{-cK^2}
\end{equation*}

\noi
for any $K>0$.
In particular, 
given $\be > 0$, 
by choosing $K = \dl^{-\frac\be2}$,  
we have
\begin{equation*}
P\Big( \sup_{n \in \Z}  |g_n(\o)| >  \dl^{-\frac{\beta}{2}} \jb{n}^\eps \Big)
< C e^{-\frac{1}{\dl^c}}
\end{equation*}
for any $ \dl > 0$.	
\end{lemma}

Next, we recall the Wiener chaos estimates.
Let $\{ g_n \}_{n \in \N}$ be a sequence of independent standard Gaussian random variables defined on a probability space $(\O, \F, P)$, where $\mathcal{F}$ is the $\s$-algebra generated by this sequence. 
Given $k \in \Z_{\geq0}$, 
we define the homogeneous Wiener chaoses $\mathcal{H}_k$ 
to be the closure (under $L^2(\O)$) of the span of  Fourier-Hermite polynomials $\prod_{n = 1}^\infty H_{k_n} (g_n)$, 
where
$H_j$ is the Hermite polynomial of degree $j$ and $k = \sum_{n = 1}^\infty k_n$.\footnote{This implies
that $k_n = 0$ except for finitely many $n$'s.}
Then, we have the following Ito-Wiener decomposition:
\begin{equation*}
L^2(\Omega, \F, P) = \bigoplus_{k = 0}^\infty \mathcal{H}_k.
\end{equation*}

\noi
See Theorem 1.1.1 in \cite{Nu}.
We also set
\begin{align}
 \H_{\leq k} = \bigoplus_{j = 0}^k \H_j
\label{Wiener}
 \end{align}

\noi
 for $k \in \N$.
For example, the random linear solution $z_1^\o$ defined in \eqref{lin1}
belongs to $\H_1$ (for each fixed $t \in \R$),
while $z_3^\o$ in \eqref{lin2} belongs to $\H_{\leq 3}$.
As pointed out above, 
the random resonant solution $z^\o$ defined in \eqref{Zlin1}
does {\it not} belong to $\H_{\leq k}$ for any finite $k \in \N$.


In this setting, 
we have the following Wiener chaos estimate
\cite[Theorem~I.22]{Simon}.
See also \cite[Proposition~2.4]{TTz}.

\begin{lemma}\label{LEM:hyp}
Let $k \in \N$.
Then, we have
\begin{equation*}
\|X \|_{L^p(\O)} \leq (p-1)^\frac{k}{2} \|X\|_{L^2(\O)}
 \end{equation*}
 
 \noi
 for any finite $p \geq 2$
 and any $X \in \H_{\leq k}$.

\end{lemma}

We also recall the following  lemma,
which is a consequence of Chebyshev's inequality. 
See, for example, Lemma 4.5 in \cite{TzBO}
and  the proof of Lemma 3 in \cite{BOP1}.\footnote{This corresponds to Lemma 2.3 in the arXiv version.}

\begin{lemma} \label{LEM:tail}
Let $k \geq 1$.
Suppose that there exists $C_0>0$ such that a random variable $X$ satisfies
$\|X\|_{L^p(\O)} \leq C_0 p^{\frac{k}2}$
for any finite $p \geq 2$.
Then,  there exist $c, C > 0$ such that
\begin{align*}
P\big( | X| > \lambda\big) 
\le C 
e^{-c  \, C_0^{-\frac{2}{k}}\lambda^{\frac2k}}
\end{align*}

\noi
 for any $\lambda >0$.
\end{lemma}

\smallskip

In  probabilistic well-posedness theory,  a probabilistic improvement of Strichartz estimates for random linear solutions plays an important role. The following lemma states that a similar estimate also holds for the random resonant solution $z^\o$ defined in \eqref{Zlin1}.

\begin{lemma} \label{LEM:Z1} 
Given $\al \geq 0$, let $z^\o$ be the solution to the resonant 4NLS \eqref{ZNLS1} given by~\eqref{Zlin1}.
Then, given $p \geq 2$ and  $\eps > 0$, there exist  $c, C > 0$ such that
\begin{align}
 P \Big( \|\P_N z^\o\|_{L^p_{x, t}(\mathbb{T}\times [-\dl, \dl])} >N^{\frac{1}{2} - \al+\eps} \Big)
 <C e^{-\frac{N^{2\eps}}{\dl^c} }
\label{Z1-1}
 \end{align}

\noi	
for any $ \dl  >0$ and dyadic $N \geq 1$. 
\end{lemma}

One way to prove Lemma \ref{LEM:Z1} would be  to directly apply the Wiener chaos estimate  (Lemma \ref{LEM:hyp}) to  the  $(2k+1)$-fold products of Gaussian random variables in 
the series expansion \eqref{Zlin1a}. 
See Lemma \ref{LEM:Z2} for such a direct approach. In the particular case of Lemma~\ref{LEM:Z1}, 
we can  give a shorter proof by exploiting 
the invariance of  a complex-valued mean-zero Gaussian random variable
under the transformation:  $g \mapsto e^{it |g|^2 } g$;
 see Lemma~4.2 in~\cite{OTz1}.
This allows us to  avoid
higher order products of Gaussian random variables.
\begin{proof}[Proof of Lemma~\ref{LEM:Z1}]
Given $n \in \Z$ and $(x, t) \in\T\times  \R$, define $h_n(x, t)$ by 
\[h_n(x, t)  := e^{i(nx - n^4 t)} e^{it \frac{|g_n|^2}{\jb{n}^{2\al}}}  \frac{g_n}{\jb{n}^{\al}}.\]

\noi
Then, it follows from the rotational invariance of complex-valued Gaussian random variables
and Lemma 4.2 in \cite{OTz1}
that $h_n(x, t) \sim \NN_\C(0,  \jb{n}^{-2\al})$
for each {\it fixed} $(x, t) \in \T\times \R$.

By Minkowski's integral inequality and  Lemma \ref{LEM:hyp}, 
we have
\begin{align*}
\bigg(\mathbb{E} \Big[\|\P_N z^\o\|_{L^p_{x, t}(\mathbb{T} \times [-\dl, \dl])}^r\Big]\bigg)^\frac{1}{r} 
&\leq \bigg\|\Big\|\sum_{|n| \sim N} h_n(x, t)\Big\|_{L^r(\Omega)}\bigg\|_{L^p_{x, \dl}} \notag \\
& \les \sqrt r  \,\bigg\|\Big\|\sum_{|n| \sim N} h_n(x, t)\Big\|_{L^2(\Omega)}\bigg\|_{L^p_{x, \dl}} \notag \\
&\les \sqrt{r}\, \dl^\frac{1}{p} N^{\frac{1}{2} - \al}
\end{align*}

\noi 
for any $r \geq p$. 
Then, the desired estimate~\eqref{Z1-1}
follows from Lemma~\ref{LEM:tail}.
\end{proof}

Finally, we conclude this section by stating a crucial  lemma 
 in studying powers
of  the random resonant solution $z^\o$
in the multilinear $X^{s, b}$-analysis.
This lemma
also plays an important role in 
establishing boundedness properties
of certain random multilinear functionals
of the white noise
(see Lemma \ref{LEM:NRSum} below),
which is a key ingredient for the proof of Theorem~\ref{THM:LWP}
when $\al = 0$.
We present the proof of this lemma in Appendix~\ref{SEC:A}.

\begin{lemma} \label{LEM:Z2} 
Fix a non-empty set  $\A \subset \{ 1, 2, 3\}$ and  $k, k_j \in \Z_{\geq 0}$, $j \in \A$, 
such that 
\begin{align}
k = \sum_{j \in \A} k_j.
\label{Z2-1b}
\end{align}

\noi
Given   a \textup{(}deterministic\textup{)} sequence $\big\{c^{\bar k}_{n_1, n_2, n_3} \big\}_{n_1, n_2, n_3 \in \Z}$
with $\bar k = \{ k_j \}_{j \in \A}$, 
define   a sequence $\{\Si_{n}\}_{n\in \Z}$ by
setting
\begin{align}
\Si_{n} 
= \Si_{n}(\bar k) 
 = 
\frac{1}{\prod_{j \in \A} k_j! }\sum_{(n_1, n_2, n_3) \in \G(n)}
 c^{\bar k}_{n_1, n_2, n_3}
\prod_{j \in \A} |g_{n_j}|^{2k_j}g^*_{n_j}
\label{Z2-1}
\end{align}

\noi
for $n \in \Z$, 
where $\G(n)$ is as in \eqref{Gam0} and $g^*_{n_j}$ is defined by 
\begin{align}
g^*_{n_j}= 
\begin{cases}
g_{n_j}, & \text{when }j = 1 \text{ or }3,\\
 \cj {g_{n_j} }, &  \text{when }j = 2.
\end{cases}
\label{Z2-1a}
\end{align}

\noi
Then, there exists $C> 0$, independent of $k$ and $ k_j \in \Z_{\geq 0}$, $j \in \A$, 
such that 
\begin{equation}
 \|\Si_{n} \|_{L^p(\O)} \leq C^k (p-1)^{k + \frac{|\A|}{2}}
\bigg(  \sum_{(n_1, n_2, n_3) \in \G(n)}
 |c^{\bar k}_{n_1, n_2, n_3}|^2 
\bigg)^\frac{1}{2}
\label{Z2-2}
 \end{equation}

\noi
 for all $p \geq 2$ and $n \in \Z$.
\end{lemma}

\section{Local theory, Part 1: $0 < \al \leq \frac 12$} \label{SEC:LWP1}

In this section, we present the proof of Theorem \ref{THM:LWP}
when $0 < \al \leq \frac 12$.
In particular, we show that the Cauchy problem
\eqref{ZNLS3} for $v$ 
is almost surely locally well-posed.
More precisely, we show that for each small $\dl> 0$, 
 there exists $\Omega_\dl$ with $P(\O_\dl^c) < Ce^{-\frac{1}{\dl^c}}$ 
such that, for each $\o \in \Omega_\dl$, 
the map  $\G^\o$ defined in \eqref{ZNLS4} is 
a contraction on $B(1)$,  
where $B(1)$ denotes the ball of radius 1 in $X^{0, \frac{1}{2}+, \dl}$
centered at the origin.

Given $v$ on $\T\times [-\dl, \dl]$, 
let $\wt{v}$ be an extension of $v$ onto $\T \times \R$.
By the non-homogeneous linear estimate (Lemma~\ref{LEM:linear}),  we have
\begin{align*}
\bigg\| \int_0^t S(t - t') \fN^\o(v) (t') d t'\bigg\|_{X^{0, \frac{1}{2}+, \dl}}
& \leq \bigg\|\eta_{_\dl}(t) \int_0^t S(t - t') \fN^\o(\wt{v}) (t') d t'\bigg\|_{X^{0, \frac{1}{2}+}} \notag \\
& \lesssim \| \fN^\o(\wt{v}) \|_{X^{0, -\frac{1}{2}+}}, 
\end{align*}

\noi 
where  $\eta_{_\dl}$ is a smooth cutoff on $[-2\dl, 2\dl]$ 
as in \eqref{eta1}
and
\begin{align}
\fN^\o (v) : = 
\chi_{_\dl} \cdot \big( \NN(v + \wt z^\o) - \NN_2(\wt z^\o)\big)
\label{D2}
\end{align}

\noi
with  an extension $\wt z^\o$  of 
the truncated random linear solution $\chi_{_\dl}\cdot z^\o$  
from $[-\dl, \dl]$ to $\R$. 
Then, 
our main goal is to prove 
that there exists $\O_\dl \subset \O$ and $\theta > 0$  with $P(\Omega^c_\dl) < Ce^{-\frac{1}{\dl^c}}$ such that 
\begin{equation}
\| \fN^\o(\wt{v}) \|_{X^{0, -\frac{1}{2}+}} 
\lesssim \dl^\theta
\Big( 1+ \|\wt v\|_{X^{0, \frac 12 + }}\Big)^3
\label{tri1}  
\end{equation}

\noi 
for all $\o \in \Omega_\dl$ 
and for  any  extension $\wt{v}$ of $v$.
By the definition \eqref{Xsb2} of the local-in-time norm, 
we then conclude from \eqref{D2} and \eqref{tri1} that 
\begin{equation*}
\bigg\| \int_0^t S(t - t') \fN^\o(v) (t') d t'\bigg\|_{X^{0, \frac{1}{2}+, \dl}}
\lesssim \dl^\theta
\Big( 1+ \| v\|_{X^{0, \frac 12 + ,\dl}}\Big)^3.
\end{equation*}

\noi
 By the trilinear structure of the nonlinearity,  a similar estimate holds for the difference $\G^\o v_1 - \G^\o v_2$,  allowing us to conclude that  $\G^\o$ is a contraction on $B(1) \subset X^{0, \frac 12+, \dl}$ for $\o \in \O_\dl$. Note that  the claim \eqref{class1} follows from  the embedding $X^{0, \frac{1}{2}+, \dl} \subset C([-\dl, \dl];L^2(\T))$
and Lemma \ref{LEM:Z}.

In view of \eqref{D2},  in order to prove \eqref{tri1},  we need to  carry  out case-by-case analysis on  
\begin{align}
\|\chi_{_\dl} \cdot \mathcal{N}_k(u_1, u_2, u_3)\|_{X^{s, -\frac{1}{2}+}},
\qquad  k = 1, 2, 
\label{tri2}
\end{align}

\noi
where $u_j$ is taken to be either of type
\begin{itemize}
\item[(I)] rough random resonant part: 
 \[\ u_j  = 
\wt z^\o\text{, where $\wt z^\o$ is {\it some} extension of }
\chi_{_\dl}\cdot  z^\o,\]

\noi
where $z^\o$ denotes the random resonant solution defined in \eqref{Zlin1}, 

\vspace{1mm}

\item[$(\II)$]	 smoother  `deterministic' nonlinear part: 
 \[ \ u_j = \wt{v}_j \text{, where  $\wt{v}_j$ is {\it any} extension of $v$,}
  \]
\end{itemize}

\noi
{\it except} for $u_1 = u_2 = u_3 = \wt z^\o$ when $k = 2$
(thanks to the subtraction of $\NN_2(\wt z^\o)$ in \eqref{D2}).

In the following, 
we take 
$\wt z^\o = \eta_{_\dl} z^\o$.
It follows from~\eqref{Zlin1} that 
\begin{align}
\F( \eta_{_\dl} z^\o)(n, \tau) = \ft \eta_{_\dl}\Big(\tau + n^4 - \tfrac{|g_n|^2}{\jb{n}^{2\al}}\Big)
\cdot \frac{g_n}{\jb{n}^{\al}}.
\label{Zlin3}
\end{align}

\noi
Thanks to the sharp cutoff function in~\eqref{tri2}, 
we may take 
\begin{align}
u_j = \chi_{_\dl}\cdot \wt v_j 
\label{tri3}
\end{align}

\noi
in \eqref{tri2} 
 when $u_j$ is of type $(\II)$.
We use the expressions $u_j(\I)$ (and $u_j(\II)$, respectively)
to mean that  $u_j$ is of type~(I) (and of type~$(\II)$, respectively) in the following.
We point out that the most intricate case appears
when all $u_j$'s are of type (I)
in estimating the non-resonant contribution.
In this case, a simple application of the Wiener chaos estimate (Lemma~\ref{LEM:hyp})
is no longer applicable
and we need to carefully estimate
the contribution from the sum of the products
of the $(2k_j+1)$-linear term, $k_j \in \N_0$,  $j  = 1, 2, 3$,  in \eqref{Zlin1a}, 
using Lemma~\ref{LEM:Z2}.
See Case (D) in Subsection~\ref{SUBSEC:LWP1b}.

\subsection{Resonant part $\mathcal{N}_2$}  \label{SUBSEC:LWP1a}
In this subsection, we estimate the resonant part of the nonlinear estimate \eqref{tri1}.
In particular, we prove 
\begin{equation}
\label{trilinear2} 
\| \chi_\dl \cdot \mathcal{N}_2(u_1, u_2, u_3)\|_{X^{0, -\frac{1}{2}+}} \lesssim \dl^\theta 
\prod_{j \in \mathcal{I}} \|\wt v_j\|_{X^{0, \frac 12 + }}
\end{equation}

\noi for some $ \theta > 0$, 
outside an exceptional set of probability  $< Ce^{-\frac{1}{\dl^c}}$, where 
$\mathcal{N}_2$ is the resonant part of the nonlinearity defined in \eqref{nonlin2}, 
$u_j$ is either of type (I) or $(\II)$, {\it except} for the case 
when all $u_j$'s  are of type (I), 
and the index set $\mathcal{I}$ is defined by 
\begin{align}
\mathcal{I} = \big\{j \in \{1, 2, 3\}: 
\text{$u_j$ is of type $(\II)$}\big\}.
\label{trilinear2a}
\end{align}

We have 
\begin{equation}
\text{LHS of } \eqref{trilinear2} 
= \bigg\| \frac{1}{\jb{\tau + n^4}^{\frac{1}{2}-}} \intt_{\tau = \tau_1 - \tau_2 + \tau_3} \ft{u}_1(n, \tau_1)\cj{\ft{u}_2(n, \tau_2)}\ft{u}_3(n, \tau_3) d\tau_1 d\tau_2 \bigg\|_{\l^2_n L^2_\tau} . 
\label{easy1} 
\end{equation}

\medskip

\noi 
$\bullet$ {\bf Case (a):} $u_j$ of type $(\II)$, $j = 1, 2, 3$.
\\
\indent
By H\"older's inequality with $p$ large ($\frac{1}{2} =\frac{1}{2+} + \frac{1}{p}$), we have 
\begin{align*}
\eqref{easy1} 
& \les \sup_n \|\jb{\tau +n^4 }^{-\frac{1}{2}+}\|_{L^{2+}_\tau} 
\bigg\| \intt_{\tau = \tau_1 - \tau_2 + \tau_3} \ft{u}_1(n, \tau_1)\cj{\ft{u}_2(n, \tau_2)}\ft{u}_3(n, \tau_3)
 d\tau_1 d\tau_2 \bigg\|_{\l^2_{n} L^p_\tau}  \,.
\intertext{By Young's and H\"older's inequalities,   $\l^2_n \subset \l^6_n$, 
and Lemma \ref{LEM:timedecay} with \eqref{tri3},} 
& \les \prod_{j = 1}^3    \| \ft u_j (n, \tau) \|_{\l^6_n L^{\frac{3}{2}-}_\tau} 
 \les  \prod_{j = 1}^3 \|  \jb{\tau + n^4}^{\frac{1}{6}+} \ft u_j (n, \tau) \|_{\l^6_n L^2_\tau} 
 \leq  \prod_{j = 1}^3 \| u_j \|_{X^{0, \frac{1}{6}+}}\\
& 
\les \dl^{1-} \prod_{j = 1}^3 \| \wt v_j \|_{X^{0, \frac{1}{2}+}}.
\end{align*}

\medskip

\noi 
$\bullet$ {\bf Case (b):} Exactly one $u_j$ of type (I). Say $u_1(\I)$, $u_2(\II)$, and $u_3(\II)$.
\\
\indent
By H\"older's inequality 
(with $p\gg1$ as before), 
\eqref{Zlin3}, and a change of variables, we have
\begin{align*}
\eqref{easy1} 
& \les \sup_n \|\jb{\tau + n^4}^{-\frac{1}{2}+}\|_{L^{2+}_\tau} \\
& 
\hphantom{X}
\times \bigg\| \jb{n}^{-\al} |g_n| \intt_{\tau = \tau_1 - \tau_2 + \tau_3}
 \ft \eta_{_\dl}\Big(\tau_1 + n^4 - \tfrac{|g_n|^2}{\jb{n}^{2\al}}\Big)
\cj{\ft u_2(n, \tau_2)}\ft u_3(n, \tau_3) d\tau_1 d\tau_2 \bigg\|_{\l^2_{n} L^p_\tau}\\
& \les 
\big(\sup_n  \jb{n}^{-\al} |g_n|\big)
\bigg\|
 \intt_{\tau = \zeta_1 - \tau_2 + \tau_3 - C(n, \o)}
 \ft \eta_{_\dl} (\zeta_1) 
\cj{\ft u_2(n, \tau_2)}\ft u_3(n, \tau_3) d\zeta_1 d\tau_2 \bigg\|_{\l^2_{n} L^p_\tau}, 
\end{align*}

\noi
where $C(n, \o)$ is defined by
\begin{align}
 C(n, \o) := n^4 -   \tfrac{|g_n|^2}{\jb{n}^{2\al}}.
\label{CN}
\end{align}

\noi
Note that for fixed $n \in \Z$ and $\o \in \O$,  
 $C(n, \o)$ is a fixed number.
 Hence, we can apply Young's inequality (in $\tau, \zeta_1, \tau_2$, and $\tau_3$),
 Lemma~\ref{LEM:prob} with $\be = 0+$, \eqref{decay}, 
 and Lemma \ref{LEM:timedecay}  with \eqref{tri3} as above and obtain 
\begin{align*}
\eqref{easy1} 
& \les \dl^{\frac{1}{2}-} \big(\sup_n \jb{n}^{-\al} |g_n|\big)
 \prod_{j = 2}^3 \|  \ft u_j (n, \tau)\|_{\l^4_n L^\frac{4}{3}_\tau} \\
& \lesssim \dl^{\frac{1}{2}-}
\prod_{j = 2}^3 \| \jb{\tau+ n^4}^{\frac{1}{4}+} \ft u_j (n, \tau) \|_{\l^4_n L^2_\tau} 
\leq \dl^{\frac{1}{2}-}
\prod_{j = 2}^3 \|u_j \|_{X^{0 , \frac{1}{4}+}}\\
& \les \dl^{1-}
\prod_{j = 2}^3 \|\wt v_j \|_{X^{0 , \frac{1}{4}+}}
\end{align*}

\noi 
for any $\al > 0$,  outside an exceptional set of probability $< C e^{-\frac{1}{\dl^c}}$.

\medskip

\noi 
$\bullet$ {\bf Case (c):} Exactly two $u_j$'s of type (I). 
\\
\indent
First, consider the case $u_1(\I)$, $u_2(\I)$, and $u_3(\II)$.
Proceeding as before
with $p\gg1$
and  a change of variables,  we have
\begin{align*}
\eqref{easy1} 
& \les
\bigg\| \jb{n}^{-2\al} |g_n|^2 \intt_{\tau = \tau_1 - \tau_2 + \tau_3} 
\ft \eta_{_\dl}\Big(\tau_1 + n^4 - \tfrac{|g_n|^2}{\jb{n}^{2\al}}\Big)\\
& \hphantom{XXXXXXXXXXX}
\times \cj{\ft \eta_{_\dl}\Big(\tau_2 + n^4 - \tfrac{|g_n|^2}{\jb{n}^{2\al}}\Big)}
\ft u_3(n, \tau_3) d\tau_1 d\tau_2 \bigg\|_{\l^2_{n} L^p_\tau}\\
& \leq
\big(\sup_n \jb{n}^{-2\al} |g_n|^2 \big)
 \bigg\|
\intt_{\tau = \zeta_1 - \zeta_2 + \tau_3} 
\ft \eta_{_\dl}(\zeta_1)
\cj{\ft \eta_{_\dl}(\zeta_2)}
 \ft u_3(n, \tau_3) d\zeta_1 d\zeta_2 \bigg\|_{\l^2_{n} L^p_\tau}\\
\intertext{By Lemma \ref{LEM:prob},  \eqref{decay}, 
and Lemma \ref{LEM:timedecay} with \eqref{tri3}, 
}
& \lesssim \dl^{\frac{1}{2}-} \big(\sup_n \jb{n}^{-2\al} |g_n|^2\big)
 \| \ft u_3 (n, \tau) \|_{\l^2_n L^2_\tau} 
\lesssim \dl^{\frac{1}{2}-} \|u_3 \|_{X^{0, 0}}\\
& \lesssim \dl^{1-} \|\wt{v}_3 \|_{X^{0, \frac 12 +}}
\end{align*}

\noi for $\al > 0$,  outside an exceptional set of probability  $< C e^{-\frac{1}{\dl^c}}$.

Next, consider the case $u_1(\I)$, $u_2(\II)$, and $u_3(\I)$. 
Proceeding in a similar manner (with $p\gg 1$ and a change of variables with $C(n,\o)$ as in \eqref{CN}), we have 
\begin{align*}
\eqref{easy1} 
& \les 
\bigg\| \jb{n}^{-2\al} |g_n|^2 \intt_{\tau = \tau_1 - \tau_2 + \tau_3} 
\ft \eta_{_\dl}\Big(\tau_1 + n^4 - \tfrac{|g_n|^2}{\jb{n}^{2\al}}\Big)\\
& \hphantom{XXXXXXXXXXX}
\times 
\cj{ \ft u_2(n, \tau_2) }
\ft \eta_{_\dl}\Big(\tau_3 + n^4 - \tfrac{|g_n|^2}{\jb{n}^{2\al}}\Big)
d\tau_1 d\tau_2 \bigg\|_{\l^2_{n} L^p_\tau}\\
& \les 
\big(\sup_n \jb{n}^{-2\al} |g_n|^2\big)
\bigg\|  \intt_{\tau = \zeta_1 - \tau_2 + \zeta_3 - 2C(n, \o)} 
\ft \eta_{_\dl}(\zeta_1)
 \cj{ \ft u_2(n, \tau_2) }
 \ft \eta_{_\dl}(\zeta_3) 
d\zeta_1 d\zeta_3 \bigg\|_{\l^2_{n} L^p_\tau}\\
& \lesssim \dl^{\frac{1}{2}-} \big(\sup_n \jb{n}^{-2\al} |g_n|^2\big)
 \| \ft u_2 (n, \tau) \|_{\l^2_n L^2_\tau} 
\lesssim \dl^{\frac{1}{2}-} \|u_2 \|_{X^{0, 0}}\\
& \lesssim \dl^{1-} \|\wt{v}_2 \|_{X^{0, \frac 12 + }}
\end{align*}

\noi 
for $\al > 0$,  outside an exceptional set of probability  $< C e^{-\frac{1}{\dl^c}}$.

\subsection{Non-resonant part $\mathcal{N}_1 $}  \label{SUBSEC:LWP1b}
In this subsection, we evaluate the non-resonant part of the nonlinearity $\fN^\o(v)$.
In particular, we prove
\begin{equation}
\label{trilinear3} 
\| \chi_{_\dl}\cdot \mathcal{N}_1(u_1, u_2, u_3) \|_{X^{0, -\frac{1}{2}+}} \lesssim \dl^\theta 
\prod_{j \in \mathcal{I}} \|\wt v_j\|_{X^{0, \frac 12 + }}
\end{equation}

\noi for some $ \theta > 0$, 
outside an exceptional set of probability $< C e^{-\frac{1}{\dl^c}}$, 
where $\mathcal{N}_1$ is the non-resonant part of the nonlinearity defined  in \eqref{nonlin1},  $u_j$ is either of type (I) or $(\II)$,
and the index set $\mathcal I$ is as in \eqref{trilinear2a}.
Set
\[ \s := \jb{\tau + n^4}\qquad \text{and}\qquad \s_j := \jb{\tau_j +n_j^4}, \quad j = 1, 2, 3,\]

\noi
and
\begin{align}
\sigma_{\max} : = \max(\sigma, \sigma_1,\sigma_2,\sigma_3)
\qquad \text{and}  \qquad 
n_{\max} : = \max\big( |n|, |n_1|, |n_2|, |n_3|\big) + 1.
\label{Nmax}
\end{align}

\noi
Given dyadic numbers $ N, N_1,N_2,N_3 \geq 1$, we also set
\[N_{\max} : = \max( N, N_1,N_2,N_3).\]

\noi
By duality, we can estimate the left-hand side of \eqref{trilinear3} by 
\begin{equation}
\label{duality1} 
\sup_{\| w\|_{X^{0, \frac{1}{2}-} }\leq 1} \bigg|\int_{-\dl}^\dl \int_\T \NN_1( u_1, u_2 , u_3) \cdot w \, dx dt \bigg|.
\end{equation}

\noi 
Without loss of generality, we may assume that $w = \chi_{_\dl} \cdot w$.

\medskip

\noi $\bullet$ {\bf Case (A):} 
$u_j$ of type ($\II$), $j = 1, 2, 3$.
\\
\indent
By H\"older's inequality,  \eqref{L4}, and 
Lemma \ref{LEM:timedecay} with \eqref{tri3}, we have 
\begin{equation*}
\eqref{duality1} 
\lesssim 
\prod_{j = 1}^3 \| u_j \|_{L^4_{x, t}} \| w\|_{L^4_{x, t}}
\les 
\dl^{\frac{3}{4}-}  \prod_{j = 1}^3 \| \wt v_j\|_{X^{0, \frac{1}{2}+}}
\| w\|_{X^{0, \frac{1}{2}-}}.
\end{equation*}

\medskip

\noi $\bullet$ {\bf Case (B):} 
Exactly one $u_j$ of type (I). Say $u_1(\I)$, $u_2(\II)$, and $u_3(\II)$.
\\
\indent
First suppose that $\max (\s_2, \s_3, \s) \sim \s_\text{max}$.
Then, it follows from Lemma \ref{LEM:phase} that 
\begin{align}
\max (\s_2,  \s_3, \s)^{\frac{7}{24}-} \sim \s_\text{max}^{\frac{7}{24}-}  
\ges N_\text{max}^{\frac{7}{12}-}.
\label{A1}
\end{align}

\noi
By $L^p_{x, t} L^{3+}_{x, t} L^{3+}_{x, t} L^{3+}_{x,t}$-H\"older's inequality 
with $p$ large,  
 \eqref{L3}, Lemma \ref{LEM:Z1},  
  Lemma \ref{LEM:timedecay}, and \eqref{A1},  we have 
\begin{align}
\eqref{duality1} 
& \les  \sum_{\substack{ N, N_1, N_2, N_3\\\text{dyadic}}}
\| \P_{N_1} u_1\|_{L^p_{x, t}} 
\| \P_{N_2} u_2\|_{X^{0,  \frac{5}{24}+}} \| \P_{N_3}u_3\|_{X^{0, \frac{5}{24}+}} 
\|\P_N  w\|_{X^{0,\frac{5}{24}+}}  \notag \\
& \lesssim 
\sum_{\substack{ N, N_1, N_2, N_3\\\text{dyadic}}}
N_1^{\frac{1}{2}-\al+} \|  \P_{N_2}  u_2\|_{X^{0, \frac{5}{24}+}}
 \| \P_{N_3} u_3 \|_{X^{0, \frac{5}{24}+}} 
\| \P_N w\|_{X^{0, \frac{5}{24}+}}  \notag \\
& \lesssim 
\dl^{\frac{7}{12}-}
\sum_{\substack{ N, N_1, N_2, N_3\\\text{dyadic}}}
N_{\max}^{-\frac{1}{12}+} \|  \P_{N_2} \wt  v_2\|_{X^{0, \frac{1}{2}+}} \| \P_{N_3} \wt v_3 \|_{X^{0, \frac{1}{2}+}} 
\| \P_N w\|_{X^{0, \frac{1}{2}-}} \notag \\
& \lesssim 
\dl^{\frac{7}{12}-}\prod_{j = 2}^3 \| \wt v_j\|_{X^{0, \frac{1}{2}+}}
\label{A2}
\end{align}

\noi
for $\al \geq  0$,  
outside an exceptional set of probability 
\[< \sum_{\substack{N_1\geq 1\\\text{dyadic}}} Ce^{-\frac{N_1^{\eps}}{\dl^c}}
\les e^{-\frac{1}{\dl^c}}.\]

Next,  suppose that $\max (\s_2, \s_3, \s) \ll \s_\text{max}$, namely $\s_1 \sim \s_\text{max}$.
We first consider the case
$\dl^\be \gg N_{\max}^{-2 + 2\eps}$
for some small $\be , \eps > 0$.
It follows from  Lemmas \ref{LEM:phase}  and  \ref{LEM:prob}
that there exists a set $\O_{\be, \eps} \subset \O$
with $P(\O_{\be, \eps}^c) < Ce^{-\frac{1}{\dl^c}}$
such that 
\begin{align*}
 \frac{|g_{n_1}|^2}{\jb{n_1}^{2\al}}
\les \dl^{-\be} \jb{n_1}^\eps \ll N_{\max}^{2 - \eps}  \ll \s_{\max}, 
\end{align*}

\noi
on $\O_{\be, \eps}$, uniformly in $n_1 \in \Z$, 
as long as $\al \geq 0$. 
Hence, 
we have 
\begin{align}
\Big| \ft \eta_{_\dl}\Big(\tau_1 + n_1^4 - \tfrac{|g_{n_1}|^2}{\jb{n_1}^{2\al}}\Big)\Big|
\les \frac{1}{\s_1} \les \frac{1}{N_{\max}^2 |(n-n_1) (n-n_3)|}
\label{A3}
\end{align}

\noi 
on $\O_{\be, \eps}$.
Then, 
by H\"older's inequality 
(with $p\gg1$ as in Case (a)),  \eqref{Zlin3}, 
 \eqref{A3},  Young's inequality,  and Lemma \ref{LEM:prob}
 (with $\be \ll 1$), 
 the contribution 
 to~\eqref{duality1} in this case is bounded by 
\begin{align*}
& \les 
  \sum_{\substack{ N, N_1, N_2, N_3,\text{dyadic}\\\dl^\be \gg N_{\max}^{-2 + 2\eps}}}
\bigg\| \sum_{\substack{(n_1, n_2, n_3) \in \G(n)\\ 
|n|\sim N, |n_j| \sim N_j}}
\frac{ |g_{n_1}| }{ \jb{n_1}^{\al}}
\frac{1}{\big\{N_{\max}^2 |(n-n_1) (n-n_3)|\big\}^{\frac{1}{2}+\eps}}
\\
& \hphantom{X}
\times \intt_{\tau = \tau_1 - \tau_2 + \tau_3}
 \Big|\ft \eta_{_\dl}\Big(\tau_1 + n_1^4 - \tfrac{|g_{n_1}|^2}{\jb{n_1}^{2\al}}\Big)\Big|^{\frac{1}{2}-\eps}
|\ft{\P_{N_2} u}_2(n_2, \tau_2)||\ft{\P_{N_3} u}_3(n_3, \tau_3)| d\tau_1 d\tau_2 \bigg\|_{\l^2_{n} L^p_\tau}\\
& \les 
\big(\sup_{n_1}  \jb{n_1}^{-\al - \eps} |g_{n_1}|\big)
\sum_{\substack{ N, N_1, N_2, N_3,\text{dyadic}\\\dl^\be \gg N_{\max}^{-2 + 2\eps}}}
\bigg\| \sum_{\substack{(n_1, n_2, n_3) \in \G(n)\\
|n|\sim N, |n_j| \sim N_j}} 
\frac{1}{\big\{N_{\max}^2 |(n-n_1) (n-n_3)|\big\}^{\frac{1}{2}+\frac{\eps}{2}}}\\
& \hphantom{X}
\times \big\| |\ft \eta_{_\dl}|^{\frac{1}{2}-\eps}\big\|_{L^\frac{p}{3}_{\tau}}
\prod_{j = 2}^3 \|\ft{\P_{N_j} u}_j(n_j, \tau_j)\|_{L^\frac{p}{p-1}_{\tau_j}}
\bigg\|_{\l^2_{n}}\\
& \les 
\dl^{\frac{1}{2} - \eps  - \frac{3}{p} -\frac{\be}{2}}  
\sum_{\substack{ N, N_1, N_2, N_3,\text{dyadic}\\\dl^\be \gg N_{\max}^{-2 + 2\eps}}}
N_{\max}^{0-} 
\prod_{j = 2}^3 \|\ft{\P_{N_j} u}_j(n_j, \tau_j)\|_{\l^2_{n_j} L^\frac{p}{p-1}_{\tau_j}}\\
& \les 
\dl^{\frac{1}{2} - \eps - \frac{3}{p} -\frac{\be}{2}}  
\prod_{j = 2}^3 \|\wt v_j \|_{X^{0 , \frac{1}{2}+}}
\end{align*}

\noi
for $\al \geq 0$, 
\noi 
outside an exceptional set of 
probability $< Ce^{-\frac{1}{\dl^c}}$.

Lastly, we consider the case $\dl^\be \les N_{\max}^{-2 + 2\eps}$.
Proceeding as in \eqref{A2},  
we bound  the contribution of this case to \eqref{duality1} by 
\begin{align*}
& \lesssim 
\dl^{\frac{21}{24}-}
\sum_{\substack{ N, N_1, N_2, N_3\\\text{dyadic}}}
N_1^{\frac{1}{2}-\al+}\|  \P_{N_2}  \wt v_2\|_{X^{0, \frac{1}{2}+}} \| \P_{N_3} \wt v _3 \|_{X^{0, \frac{1}{2}+}} 
\| \P_N w\|_{X^{0, \frac{1}{2}-}} \notag \\
& \lesssim 
\dl^{\frac{21}{24} - \be -}
\sum_{\substack{ N, N_1, N_2, N_3\\\text{dyadic}}}
N_{\max}^{-\frac{3}{2}+}\|  \P_{N_2}  \wt v_2\|_{X^{0, \frac{1}{2}+}} \| \P_{N_3} \wt v _3 \|_{X^{0, \frac{1}{2}+}} 
 \notag \\
& \lesssim 
\dl^{\frac{21}{24} - \be -}\prod_{j = 2}^3 \|\wt v_j \|_{X^{0 , \frac{1}{2}+}}
\end{align*}

\noi
for $\al \geq 0$,  
outside an exceptional set of probability 
$<C e^{-\frac{1}{\dl^c}}$.

\medskip

\noi $\bullet$ {\bf Case (C):} 
Exactly two $u_j$'s of type (I). 
Say  $u_1(\I)$, $u_2(\I)$, and $u_3(\II)$.

First,  suppose that $\max (\s_3, \s) \sim \s_\text{max}$.
Then, it follows from Lemma \ref{LEM:phase} that 
\begin{align}
\max ( \s_3, \s)^{\frac{1}{2}-} \sim \s_\text{max}^{\frac{1}{2}-}  
\ges N_\text{max}^{1-}.
\label{A4}
\end{align}

\noi
Suppose that $ \s \sim \s_\text{max}$.
Then, by $L^p_{x, t} L^{p}_{x, t} L^{2+}_{x, t} L^{2}_{x,t}$-H\"older's inequality 
with $p$ large,  
 \eqref{L3}, Lemma \ref{LEM:Z1},  
  Lemma \ref{LEM:timedecay}, and \eqref{A4},  we have 
\begin{align*}
\eqref{duality1} 
& \les  \sum_{\substack{ N, N_1, N_2, N_3\\\text{dyadic}}}
\| \P_{N_1} u_1\|_{L^p_{x, t}} 
\| \P_{N_2} u_2\|_{L^p_{x, t}} 
 \| \P_{N_3}u_3\|_{X^{0, 0+}} 
\|\P_N  w\|_{X^{0,0}}  \notag \\
& \lesssim 
\sum_{\substack{ N, N_1, N_2, N_3\\\text{dyadic}}}
N_1^{\frac{1}{2}-\al+}N_2^{\frac{1}{2}-\al+} 
 \| \P_{N_3}u_3\|_{X^{0, 0+}} 
\|\P_N  w\|_{X^{0,0}}  \notag \\
& \lesssim 
\dl^{\frac 12 -  }
\sum_{\substack{ N, N_1, N_2, N_3\\\text{dyadic}}}
N_{\max}^{-2\al +}  \| \P_{N_3} \wt v_3 \|_{X^{0, \frac{1}{2}+}} 
\| \P_N w\|_{X^{0, \frac{1}{2}-}} \notag \\
& \lesssim 
\dl^{\frac{1}{2}-}
\| \wt v_3 \|_{X^{0, \frac{1}{2}+}} 
\end{align*}

\noi
for $\al > 0$,  
outside an exceptional set of probability 
\[ < \sum_{\substack{N_1\geq 1\\\text{dyadic}}} Ce^{-\frac{N_1^{\eps}}{\dl^c}}
+ \sum_{\substack{N_2\geq 1\\\text{dyadic}}}Ce^{-\frac{N_2^{\eps}}{\dl^c}}
\les e^{-\frac{1}{\dl^c}}. \]

\noi
A similar argument holds when \noi
 $ \s_3 \sim \s_\text{max}$.

Next,  suppose that $\max ( \s_3, \s) \ll \s_\text{max}$, namely $\max(\s_1, \s_2) \sim \s_\text{max}$.
Without loss of generality, suppose that $\s_1 \sim \s_\text{max}$.
We first consider the case $\dl^\be \gg N_{\max}^{-2 + 2\eps}$ 
for some small $\be, \eps > 0$.
Proceeding as in Case (B) above, 
 the contribution  to \eqref{duality1}  is bounded by
\begin{align*}
& \les 
  \sum_{\substack{ N, N_1, N_2, N_3,\text{dyadic}\\\dl^\be \gg N_{\max}^{-2 + 2\eps}}}
\bigg\| \sum_{\substack{(n_1, n_2, n_3) \in \G(n)\\
|n|\sim N, |n_1| \sim N_1}} 
\bigg(\prod_{j = 1}^2  \frac{|g_{n_j}|}{\jb{n_j}^{\al}}\bigg) 
\frac{1}{\big\{N_{\max}^2 |(n-n_1) (n-n_3)|\big\}^{\frac{1}{2}+\eps}}.
\\
& \hphantom{X} 
\times \intt_{\tau = \tau_1 - \tau_2 + \tau_3}
 \Big|\ft \eta_{_\dl}\Big(\tau_1 + n_1^4 - \tfrac{|g_{n_1}|^2}{\jb{n_1}^{2\al}}\Big)\Big|^{\frac{1}{2}-\eps}
  \Big|\ft \eta_{_\dl}\Big(\tau_2 + n_2^4 - \tfrac{|g_{n_2}|^2}{\jb{n_2}^{2\al}}\Big)\Big|
|\ft{\P_{N_3} u}_3(n_3, \tau_3)| d\tau_1 d\tau_2 \bigg\|_{\l^2_{n} L^p_\tau}\\
& \les 
\bigg(\prod_{j = 1}^2 \sup_{n_j}  \jb{n_j}^{-\al - \frac{\eps}{2}} |g_{n_j}|\bigg)
\sum_{\substack{ N, N_1, N_2, N_3,\text{dyadic}\\\dl^\be \gg N_{\max}^{-2 + 2\eps}}}
\bigg\| \sum_{\substack{(n_1, n_2, n_3) \in \G(n)\\
|n|\sim N, |n_1| \sim N_1}} 
\frac{1}{\big\{N_{\max}^2 |(n-n_1) (n-n_3)|\big\}^{\frac{1}{2}+\frac{\eps}{2}}}\\
& \hphantom{X}
\times \big\| |\ft \eta_{_\dl}|^{\frac{1}{2}-\eps}\big\|_{L^\frac{p}{2}_{\tau}}
\| \ft \eta_{_\dl}\|_{L^1_{\tau}}
 \|\ft{\P_{N_3} u}_3(n_3, \tau_3)\|_{L^\frac{p}{p-1}_{\tau_3}}
\bigg\|_{\l^2_{n}}\\
& \les 
\dl^{\frac{1}{2} - \eps  - \frac{2}{p} -\be}  
\sum_{\substack{ N, N_1, N_2, N_3,\text{dyadic}\\\dl^\be \gg N_{\max}^{-2 + 2\eps}}}
N_{\max}^{0-} 
 \|\ft{\P_{N_3} u}_3(n_3, \tau_3)\|_{\l^2_n L^\frac{p}{p-1}_{\tau_3}}\\
& \les 
\dl^{\frac{1}{2} - \eps  - \frac{2}{p} -\be}  
\|\wt v_3 \|_{X^{0 , \frac{1}{2}}}
\end{align*}

\noi
for $\al \geq 0$, 
outside an exceptional set of 
probability $< Ce^{-\frac{1}{\dl^c}}$.

Lastly, we consider the case $\dl^\be \les N_{\max}^{-2 + 2\eps}$.
Proceeding as in \eqref{A2}
but with  $L^p_{x, t} L^{p}_{x, t} L^{2+}_{x, t} L^{2}_{x,t}$-H\"older's inequality, 
 the contribution of this case to \eqref{duality1} 
\begin{align*}
& \lesssim 
\sum_{\substack{ N, N_1, N_2, N_3\\\text{dyadic}}}
N_{\max}^{1-2\al+}
 \| \P_{N_3}u_3\|_{X^{0, 0+}} 
\|\P_N  w\|_{X^{0,0}}  \notag \\
& \lesssim 
\dl^{1  - \be -  }
\sum_{\substack{ N, N_1, N_2, N_3\\\text{dyadic}}}
N_{\max}^{-1- 2\al +}  \| \P_{N_3} \wt v_3 \|_{X^{0, \frac{1}{2}+}} 
\| \P_N w\|_{X^{0, \frac{1}{2}-}} \notag \\
& \lesssim 
\dl^{1 - \be - }
\|  \wt v_3 \|_{X^{0, \frac{1}{2}+}} 
\end{align*}

\noi
for $\al \geq 0$, 
outside an exceptional set of 
probability $< Ce^{-\frac{1}{\dl^c}}$.

\medskip

\noi
$\bullet$ {\bf Case (D):}
$u_j$ of type (I), $j = 1, 2, 3$.

Fix small $\dl > 0$ (to be chosen later).
From \eqref{Zlin1a} and \eqref{eta1}, we have 
\begin{align*}
\F(\eta_{_\dl} z^\o)(n, \tau)
 & =\dl\, \, \sum_{k = 0}^\infty
 \frac{(-\dl)^k  }{k!}
 \, (\dd^k\ft \eta)  (\dl(\tau + n^4) )
  \frac{|g_n|^{2k}g_n}{\jb{n}^{(2k+1)\al}}\,.
\end{align*}
\noi
Then, we have 
\begin{align}
 \| \NN_1  ( \eta_{_\dl} z^\o) \|_{X^{0, -\frac 12+}}
&  = \bigg\| \frac{1}{\jb{\tau + n^4}^{\frac 12 -}}
\sum_{k = 0}^\infty
\sum_{\substack{k_1, k_2, k_3 = 0\\ k = k_1 + k_2 + k_3}}^\infty
\frac{(-\dl)^k }{k_1! k_2!k_3!} \notag \\
& \hphantom{XX}
\times \sum_{\substack{(n_1, n_2, n_3) \in \G(n)}}
c^{k_1, k_2, k_3}_{n_1, n_2, n_3} (\tau, \dl) 
\prod_{j = 1}^3 |g_{n_j}|^{2k_j}g^*_{n_j}
 \bigg\|_{\l^2_nL^2_\tau}, 
\label{X2}
\end{align}

\noi
where
$g^*_{n_j}$ is as in \eqref{Z2-1a}
and
$c^{k_1, k_2,k_3}_{n_1, n_2, n_3} (\tau, \dl) $ is defined by 
\begin{align*}
c^{k_1, k_2, k_3}_{n_1, n_2, n_3} (\tau, \dl) 
= 
\dl^3\,\,
 \intt_{\tau  = \tau_1 - \tau_2 + \tau_3}
\prod_{j = 1}^3 \frac{  (\dd^{k_j}\ft \eta_j)  (\dl(\tau_j + n_j^4) ) }{\jb{n_j}^{(2k_j + 1)\al}}
d\tau_1 d \tau_2
\end{align*}
	
\noi
 with the convention that 
$\ft \eta_j = \ft \eta$ when $j = 1$ or $3$ 
and 
$\ft \eta_j = \cj {\ft \eta}$ when $j = 2$.
Then, by Minkowski's integral inequality and Lemma~\ref{LEM:Z2}, there exists $C> 0$ such that 
\begin{align}
\big\|     \|   \NN_1  &   (  \eta_{_\dl} z^\o) \|_{X^{0, -\frac 12+}}\big\|_{L^p(\O)} \notag \\
&  \leq
p^\frac{3}{2} 
\sum_{k = 0}^\infty
\sum_{\substack{k_1, k_2, k_3 = 0\\ k = k_1 + k_2 + k_3}}^\infty
(Cp\dl)^{k}
 \notag \\
& \hphantom{X}\times 
\Bigg(\int_{\R} \sum_{n\in \Z} \sum_{\substack{(n_1, n_2, n_3) \in \G(n)}}
 \frac{1}{\jb{\tau + n^4}^{1-}}
|c^{k_1, k_2, k_3}_{n_1, n_2, n_3} (\tau, \dl)|^2 
d\tau \Bigg)^\frac{1}{2}   
\label{X4}
\end{align}

\noi
for any $p \geq 2$.
In the following, we estimate \eqref{X4} with
\begin{align}
 p = \dl^{-\theta} \gg 1
\label{X6}
\end{align}

\noi
for some sufficiently small $\theta > 0$.
Note that, from 
Lemma \ref{LEM:phase} and $n \ne n_1, n_3$, we have 
\begin{align}
\s_\text{max} \ges n_\text{max}^2 |(n - n_1) (n - n_3)| \geq n_\text{max}^2. 
\label{X5}
\end{align}

\medskip

\noi
$\circ$ Subcase (D.1):
$\s \sim  \s_\text{max}$.
\quad 
First, note that,
in view of $\supp \eta \subset [-2, 2]$,  
 we have
\begin{align}	
|\F^{-1}(\dd^{k_j}\ft \eta)(t)|  = |(-it)^{k_j}\eta (t)| \leq C^{k_j} \eta(t)\,.
\label{Xa1}
\end{align}

\noi
Then, 
by a change of variables:
$\zeta = \dl \tau + n_1^4 - n_2^4 + n_3^4$
and  $\zeta_j = \dl (\tau_j + n_j^4)$, $j = 1, 2, 3$, 
 Plancherel's identity, H\"older's inequality (in $t$) 
with \eqref{Xa1},  and 
$k = k_1 + k_2 + k_3$, we have 
\begin{align}
\| c^{k_1, k_2, k_3}_{n_1, n_2, n_3} (\tau, \dl) \|_{L^2_\tau}
& = \dl^\frac{1}{2}\Bigg\| \intt_{\zeta  = \zeta_1 - \zeta_2 + \zeta_3}
\prod_{j = 1}^3 \frac{  \dd^{k_j}\ft \eta_j  (\zeta_j ) }{\jb{n_j}^{(2k_j + 1)\al}}
d\zeta_1 d \zeta_2\Bigg\|_{L^2_\zeta}\notag\\
& = \dl^\frac{1}{2}\Bigg\|\prod_{j = 1}^3 \frac{ \F^{-1} (\dd^{k_j}\ft \eta_j ) }{\jb{n_j}^{(2k_j + 1)\al}}
\Bigg\|_{L^2_t}\notag\\
& \leq C^k \dl^\frac{1}{2}
\prod_{j = 1}^3 \frac{1 }{\jb{n_j}^{(2k_j + 1)\al}}.
\label{Xa2}
\end{align}

\noi

From  \eqref{X4},  \eqref{X5} and \eqref{Xa2}, we bound 
the contribution to 
$$
\big\|    \|   \NN_1   (  \eta_{_\dl} z^\o)   \|_{X^{0, -\frac 12+}}\big\|_{L^p(\O)} 
$$ 
in this case  by
\begin{align}
&  
p^\frac{3}{2} \dl^\frac{1}{2} 
\sum_{k_1, k_2, k_3 = 0}^\infty
(Cp\dl)^{k_1}(Cp\dl)^{k_2}(Cp\dl)^{k_3} \notag \\
& \hphantom{XX}
\times\Bigg(\sum_{n\in \Z} \sum_{\substack{(n_1, n_2, n_3) \in \G(n)}}
 \frac{1}{\big\{n_\text{max}^2 (n- n_1)(n - n_3)\big\}^{1-}}
\prod_{j = 1}^3 \frac{1 }{\jb{n_j}^{(2k_j + 1)\al}} \Bigg)^\frac{1}{2}   \notag 
\intertext{By choosing small $\dl = \dl(C) >0$ such that $Cp \dl  = C \dl^{1- \theta}  < 1$, }
&  \les
p^\frac{3}{2} \dl ^\frac{1}{2}
\label{Xa3}
\end{align}

\noi
for $\al  \geq 0 $.

\medskip

\noi
$\circ$ Subcase (D.2):
$\s \ll  \s_\text{max}$.
\quad 
Assume that $\s_1 \sim \s_\text{max}$.
A similar argument holds when 
$\s_2 \sim \s_\text{max}$ or $\s_3 \sim \s_\text{max}$.

From \eqref{X2}, H\"older's inequality with $q$ large $\big(\frac 12 = \frac{1}{2+} + \frac{1}{q}\big)$, 
Minkowski's integral inequality, and Lemma~\ref{LEM:Z2}, 
we have 
\begin{align}
\big\|    \|   \NN_1   ( & \eta_{_\dl}   z^\o) \|_{X^{0, -\frac 12+}}\big\|_{L^p(\O)} 
\notag \\
&   \les
 p^\frac{3}{2} \sum_{k = 0}^\infty 
\sum_{\substack{k_1, k_2, k_3 = 0\\ k = k_1 + k_2 + k_3}}^\infty
(Cp\dl)^k
\Bigg( \sum_{n\in \Z} \sum_{\substack{n = n_1 - n_2 + n_3\\n_2 \ne n_1, n_3}}
\|c^{k_1, k_2, k_3}_{n_1, n_2, n_3} (\tau, \dl)\|_{L^q_\tau}^2 
 \Bigg)^\frac{1}{2}   
\label{Xb1}
\end{align}

\noi
for any $p \geq q$.
By integration by parts, we have
\begin{align*}
| \dd^{k_1}\ft \eta  (\tau )| = 
\bigg| \frac{1}{|\tau|^{\beta}} \int \frac{d^\be}{dt^\be}\big( t^{k_1}\eta (t)\big) e^{it \tau} dt\bigg|
\end{align*}

\noi
for $\tau \ne 0$.
In particular, with $\be = 1$, 
we have 
\begin{align}
\| \dd^{k_1}\ft \eta_1  (\tau )\|_{L^\frac{2q}{q+2}_\tau(|\tau| \ges K)}
\les \frac {C^{k_1}}{K^{1 - \frac{q+2}{2q}}}.
\label{Xb2}
\end{align}

\noi
By a change of variables (as in \eqref{Xa2}) 
and Young's inequality,
\eqref{Xb2} with $K \sim \dl \s_1$, \eqref{X5}, and  \eqref{Xa1},  
we can bound  the contribution  to
$\|c^{k_1, k_2, k_3}_{n_1, n_2, n_3} (\tau, \dl)\|_{L^q_\tau} $ in this case by 
\begin{align}
&  \dl^{1-\frac{1}{q}}
\bigg(\prod_{j = 1}^3 \frac{1 }{\jb{n_j}^{(2k_j + 1)\al}}\bigg)
\| \dd^{k_1}\ft \eta_1  (\tau )\|_{L^\frac{2q}{q+2}_\tau(|\tau| \ges K)}\notag
\big\| \F^{-1} (\dd^{k_2}\ft \eta_2 )\F^{-1} (\dd^{k_3}\ft \eta_3 )\big\|_{L^2_t}
\notag\\
& \leq C^k \dl^\frac{1}{2}
\prod_{j = 1}^3 \frac{1 }{\jb{n_j}^{(2k_j + 1)\al}}
\frac{1}{\big\{ n_\text{max}^2 (n - n_1) (n - n_3)\big\}^{1 - \frac{q+2}{2q}}}.
\label{Xb3}
\end{align}

\noi
Hence,  by choosing $q \gg 1$  and proceeding as in \eqref{Xa3},  
we conclude  from \eqref{Xb1} and \eqref{Xb3} that the contribution to
$\big\|    \|   \NN_1   (  \eta_{_\dl} z^\o)   \|_{X^{0, -\frac 12+}}\big\|_{L^p(\O)} $ 
in this case is also bounded by
\begin{align}
  \les
p^\frac{3}{2} \dl ^\frac{1}{2}.
\label{Xb4}
\end{align}

Finally, by Chebyshev's inequality with \eqref{Xa3} and \eqref{Xb4}, we have 
\begin{align*}
P \Big(
\| \NN_1   (  \eta_{_\dl} z^\o)   \|_{X^{0, -\frac 12+}}
 > \ld \Big) 
 \leq C^p \ld^{-p} p^{\frac{3}{2}p}  \dl^\frac{p}{2}
\end{align*}
	
\noi 
for any $\ld > 0$.
Letting $\ld = C p^2 \dl^\frac{1}{2} $ 
and $ p = \dl^{-\theta}$ as in \eqref{X6},  we have
\[ P \Big(
\| \NN_1   (  \eta_{_\dl} z^\o)  \|_{X^{0, -\frac 12+}}
 > C \dl^{\frac 12 -2 \theta} \Big) 
\leq e^{-p \ln \sqrt{p}} \leq e^{-\frac{1}{\dl^{c}}}\]

\noi
for all $\al \geq 0$.
In other words, 
we have
\[ 
\| \NN_1   (  \eta_{_\dl} z^\o)  \|_{X^{0, -\frac 12+}}
 \leq C \dl^{\frac 12 -} \]
 
\noi
for $\al \geq 0$, outside an exceptional set of probability
$\lesssim  e^{-\frac{1}{\dl^{c}}}$.

\medskip

This completes the proof of the nonlinear estimate~\eqref{tri1}
and hence the proof of Theorem~\ref{THM:LWP} for $0 < \al \leq \frac 12$.

\section{Local theory, Part 2: $\al = 0$} \label{SEC:LWP2}

The remaining part of this paper is devoted to the $\al = 0$ case.
Namely, we consider the white noise initial data.
In this section, we
  present the proof of almost sure local well-posedness (Theorem~\ref{THM:LWP})
  by establishing convergence of smooth approximating solutions.
The key ingredients are Propositions~\ref{PROP:I1} and \ref{PROP:I2}, whose proofs will be presented 
in Sections~\ref{SEC:Non1} and~\ref{SEC:Non2}, respectively. 

\subsection{Partially iterated Duhamel formulation}
\label{SUBSEC:LWPx}
In Section \ref{SEC:1}, 
we introduced the random gauge transform $\J^\o$ in \eqref{Ga}
and converted the renormalized 4NLS \eqref{NLS2}
into the random equation \eqref{NLS2-3}
for $w = \J^\o(u)$.
In the following, we study the Duhamel formulation \eqref{NLS3-1}
for this random equation.
Define 
\begin{align}
\In_1(w_1,w_2,w_3) (t)
& := - i \int_0^t S(t-t')  \NN_1^\o (w_1,w_2,w_3) (t') dt' ,\notag \\
\label{I2}
\In_2(w)(t)&  := - i \int_0^t S(t-t')  \NN_2^\o (w) (t') dt' ,
\end{align}

\noi
where  $\NN_1^\o(w_1,w_2,w_3)$ is defined by 
\begin{align*}
\NN_1^\o (w_1,w_2,w_3) 
(x, t)
: = \sum_{n\in \Z}  e^{inx} \sum_{\G(n)} 
e^{it\Psi^\o(\bar n)} \ft w_1 (n_1, t) \cj{\ft w_2 (n_2, t)} \ft w_3( n_3, t)
\end{align*}

\noi
with  the random phase function  $\Psi^\o$ defined in \eqref{Psi}
and  $\NN_{2}^\o(w)$ is as in  \eqref{NN2O}.
By setting
 $\In_1(w) : = \In_1(w,w,w)$, 
we define $\In(w):= \In_1(w) + \In_2(w)$.
Then, we can write the Duhamel formulation \eqref{NLS3-1} for $w = \J^\o(u)$ as 
\begin{align}
\label{Duhamel2}
w = S(t) u_0^\o + \In(w),
\end{align}

If we were to apply the strategy for the $\al > 0$ case discussed in Section \ref{SEC:LWP1}, 
then by noting that $\J^\o(z^\o) = S(t) u_0^\o$, 
we would   write $v = w - S(t) u_0^\o$
and try to solve the fixed point problem for $v$:
\begin{align}
v = \In_1(v+S(t)u_0^\o) +  \In_2(v+S(t)u_0^\o)
\label{GaNLS}
\end{align}

\noi
by a contraction argument.
As mentioned in Section \ref{SEC:1}, however, 
we are not able to solve the fixed point problem~\eqref{GaNLS}
by a contraction argument.
In the following, we reformulate the equation 
by {\it assuming} that $w$ is a solution to~\eqref{Duhamel2}
and study the reformulated problem.
Recalling that $\ft w(n,0) = g_n$ and that $w$ satisfies the equation~\eqref{NLS2-3}, 
we formally have
\begin{align}
|\ft w(n,t)|^2 - |g_n|^2 & = \int_0^t \frac{d}{dt} |w(n,t')|^2 dt'\notag\\
& = -2\Re i \int_0^t \sum_{\G(n)} e^{it\Psi^\o(\bar n)} \ft{w}(n_1,t') \cj{\ft w (n_2,t')} \ft w (n_3,t') \cj{\ft w (n,t')} dt'\notag\\
& =: \EE_n(w,w,w,w)(t). \label{En}
\end{align}

\noi
In view of \eqref{NN2O}, \eqref{I2},  and \eqref{En},   we then have
\begin{align*}
\In_2 (w) = i \int_0^t S(t-t') \sum_{n\in \Z} e^{inx} \EE_n(w,w,w,w)(t') \ft w (n,t')  dt'
\end{align*}

\noi
for a solution $w$ to \eqref{NLS2-3}.
We denote by $\wt \In_2(w)$ the quintilinear operator  $\wt \In^\o_2(w,w,w,w,w)$
given by 
\begin{align*}
\wt \In_2^\o (w_1,w_2,w_3,w_4,w_5) (x, t)
:= \int_0^t S(t-t') \sum_{n\in \Z} e^{inx} \EE_n(w_1,w_2,w_3,w_4)(t') \ft w_5 (n,t')  dt'.
\end{align*}

\noi
Then,  for a solution $w$ to \eqref{NLS2-3}, 
the equality 
 \begin{align}
 \label{ExpandI2}
 \In_2 (w) = \wt \In_2 (w)
 \end{align}
formally holds. 
As a result, 
we can rewrite \eqref{Duhamel2} as  the following partially iterated Duhamel
formulation  with cubic and quintic nonlinearities:
\begin{align}
w = S(t) u_0^\o + \In_1(w) + \wt \In_2(w).
\label{Duhamel3}
\end{align}

\noi
We then obtain the following fixed point problem
for   $v = w - S(t) u_0^\o$:
 \begin{align}
v = \In_1(v+S(t)u_0^\o) + \wt \In_2(v+S(t)u_0^\o).
\label{GaNLS2}
\end{align}

\noi
It turns out that 
the quintic term $\wt \In_2(v+S(t)u_0^\o)$ 
has a better regularity property
than the original cubic resonant nonlinearity $ \In_2(v+S(t)u_0^\o)$, 
which enables us to solve the fixed point problem~\eqref{GaNLS2} for $v$
by a contraction argument.
See Remark \ref{REM:uniq2} below.
Note, however, that in deriving the equation~\eqref{GaNLS2}, 
we used the a priori equality \eqref{ExpandI2},
 which only holds for a solution $w = S(t) u_0^\o + v$ to \eqref{Duhamel2}.

In order to overcome this issue, 
we use an approximation method to construct 
a solution to \eqref{NLS2}. 
To be more precise,  
we construct a local solution $u$ to \eqref{NLS2} 
as a limit of a sequence $\{u^N\}_{N \in \N}$ of smooth solutions 
with smooth initial data $u_{0, N}^\o$.
For simplicity of the presentation, 
we only consider 
the following frequency-truncated data: 
 \begin{align*}
 u_{0, N}^\o : = \pi_{N} u_0^\o = \sum_{|n|\le N} g_n(\o)e^{inx}
 \end{align*} 
 
 \noi
 in the following.
Here,  $\pi_N$ is the Dirichlet frequency projection 
onto the frequencies $\{|n| \leq N \}$ defined in \eqref{pi1}. 
See Remark~\ref{REM:uniq2}\,(ii)
for the case of smooth initial data given by  mollification as in \eqref{smooth1}.

Letting 
\begin{align}
\label{gN}
g_n^N := \ind_{|n|\leq N} \cdot g_n
= 
\begin{cases}
g_n, &\text{ if } |n| \le N,\\
0, &\text{ if } |n| > N, 
\end{cases}
\end{align}

\noi
we have
\[
u_{0, N}^\o(x) =  \sum_{n\in \Z} g_n^N(\o)e^{inx}.
\]

\noi
Define a truncated version of the random phase function $\Psi^\o$ in \eqref{Psi}
by setting 
\begin{align}
\label{PsiN}
\Psi_N^\o : =  |g_{n_1}^N(\o)|^2 - |g_{n_2}^N(\o)|^2 + |g_{n_3}^N(\o)|^2 - |g_{n}^N(\o)|^2.
\end{align}

\noi
We also set  $\Psi_\infty^\o = \Psi^\o$.

Let  $N \in \N$. 
Then, we have  $u_{0, N}^\o \in C^\infty(\T)$ almost surely.
Hence, by Proposition 1.1 in~\cite{OTz1}, 
 there exists a unique global-in-time solution $u^N$ to \eqref{NLS2} with 
 $u^N|_{t = 0} = u_{0, N}^\o$. 
 Furthermore, by  introducing the truncated random gauge transform:
\begin{align}
\label{GaN}
w^N(x, t)= \J_N^\o(u^N) : = \sum_{n\in\Z} e^{inx - it|g_n^N(\o)|^2} \ft{u^N}(n, t)
\end{align}

\noi
with $g_n^N$ in  \eqref{gN}, 
we see that  $w^N$ satisfies a modified version of the random equation \eqref{NLS2-3}: 
\begin{equation}
\label{NLS2-3'} 
\begin{cases}
i \dt w^N =  \partial_x^4 w^N + \NN_{1,N}^\o(w^N) + \NN_{2,N}^\o(w^N) \\
w|_{t= 0} = u_{0, N}^\o,
\end{cases} 
\end{equation}

\noi
where $\NN_{1,N}^\o(w) = \NN_{1,N}^\o(w, w, w)$
and $\NN_{2,N}^\o(w) $ are defined by 
\begin{align}
\NN_{1,N}^\o(w_1, w_2, w_3) (x, t)
&  : =  \sum_{n\in \Z}  e^{inx} \sum_{\G(n)} e^{it\Psi_N^\o(\bar n)} \ft w_1 (n_1, t) \cj{ \ft w_2 (n_2, t)} \ft w_3 (n_3, t), \label{I0N}\\
\NN_{2,N}^\o(w) (x, t)
& : = - \sum_{n\in \Z}  e^{inx}  \big[|\ft w(n, t)|^2 - |g^N_n(\o)|^2\big] \ft w( n, t).
\notag
\end{align}

\noi
By writing \eqref{NLS2-3'} in the Duhamel formulation, we have
\begin{align}
\label{DuhamelN}
w^N = S(t) u_{0, N}^\o + \In_{1,N}^\o(w^N) + \In_{2,N}^\o(w^N),
\end{align}

\noi
where $\In_{1,N}^\o(w) : = \In_{1,N}^\o(w,w,w)$ and $\In_{2,N}^\o(w) $ are defined by
\begin{align}
\label{I1N}
\In_{1,N}^\o(w_1,w_2,w_3) &:= -i \int_0^t S(t-t')  \NN_{1,N}^\o (w_1,w_2,w_3) (t') dt' ,\\
\label{I2N}
\In_{2,N}^\o(w) & := - i \int_0^t S(t-t')  \NN_{2,N}^\o (w) (t') dt'.
\end{align}

\noi
Noting that
$w^N$ is almost surely a  smooth solution to \eqref{NLS2-3'} with 
the truncated random initial data $u_{0, N}^\o$, we have
\begin{align}
|\ft {w^N}(n,t)|^2 - |g_n^N|^2 & = \int_0^t \frac{d}{dt} |\ft{w^N}(n,t')|^2 dt'\notag\\
& = -2\Re i \int_0^t \sum_{\G(n)} e^{it'\Psi_N^\o(\bar n)} \ft{w^N}(n_1,t') \cj{\ft {w^N}(n_2,t')} \ft {w^N} (n_3,t') 
\cj{\ft {w^N}(n,t')} dt'\notag\\
& =: \EE_n^N(w^N,w^N,w^N,w^N)(t). \label{EnN}
\end{align}

\noi
This motivates us to define  a truncated version of $\wt \In_2$ by 
\begin{align}
\label{MultiI2N}
\wt \In_{2,N}^\o  (w_1,w_2, & w_3,w_4,w_5) (x, t)\notag \\
& : = \int_0^t S(t-t') \sum_{n\in \Z} e^{inx} \EE_n^N (w_1,w_2,w_3,w_4)(t') \ft w_5 (n,t')  dt'.
\end{align}

\noi
We also set $\wt \In_{2,N}^\o(w^N)= \wt \In_{2,N}^\o(w^N,w^N,w^N,w^N,w^N)$.
Then, we can rewrite 
\eqref{DuhamelN} as the following partially iterated Duhamel formulation:
\begin{align}
\label{DuhamelNexpan}
w^N = S(t) u_{0, N}^\o + \In_{1,N}^\o(w^N) + \wt \In_{2,N}^\o(w^N).
\end{align}

\noi
Note that while $\In_{2,N}^\o(w)$ in \eqref{I2N}
corresponds to  the resonant part of the nonlinearity,
only  the non-resonant contribution survives in \eqref{EnN}
after substituting the equation,
thus yielding a non-resonant structure 
in the quintic term $\wt \In_{2,N}^\o(w^N)$.

In order to prove Theorem \ref{THM:LWP},
we need to show that $\{w^N\}_{N \in \N} $ converges in some function space
 and that the limit $w = \lim_{N\rightarrow \infty} w^N$ is a distributional solution to \eqref{NLS2-3}. 
We now state the crucial nonlinear estimates in our analysis.
Recall from  \eqref{pi2}
that given $N \in \Z_{\geq -1} = \Z\cap [-1, \infty)$, 
 $\pi_N^\perp$
denotes
 the frequency projection operator onto the (spatial) frequencies $\{|n| >N\}$
with the understanding that 
$\pi^\perp_{-1} = \Id.$

 \begin{proposition}
 \label{PROP:I1}
 Let  $0< \beta, \g \ll 1$ 
 and   $b > \frac 12$ be  sufficiently close to $\frac 12$.
Then, there exist $c, \ta>0$
and small $\dl_0 > 0$  
with the following property.
For each $0 < \dl < \dl_0$, 
there exists $\O_\dl \subset \O$ 
with 
 $P(\Omega_\dl^c) < e^{-\frac{1}{\dl^c}}$
   such that for each $\o \in \Omega_\dl$, we have
 \begin{align}
 \label{eD1}
 \|\In_{1,N}^\o(w_1,w_2,w_3)\|_{X^{0,
 b,\dl}} \le C\dl^{\theta} \prod_{j=1}^3 
 \Big(\jb{N_j}^{-\beta} + \|w_j - S(t) \pi_{N_j}^\perp (u_0^\o)\|_{X^{-\g,b,\dl}}\Big),
 \end{align}

 \noi
uniformly in     $N_j \in \Z_{\geq -1}$, $j = 1, 2, 3$,
and $N \geq N_0 (\o, \dl)$ for some $N_0(\o, \dl) \in \N$.
Here, we allow $N = \infty$ as well.

 \end{proposition}
\begin{proposition}
\label{PROP:I2}
 Let  $0< \beta, \g \ll 1$ 
 and   $b > \frac 12$ be  sufficiently close to $\frac 12$.
Then, there exist $c, \ta>0$
and small $\dl_0 > 0$  
with the following property.
For each $0 < \dl < \dl_0$, 
there exists $\O_\dl \subset \O$ 
with 
 $P(\Omega_\dl^c) < e^{-\frac{1}{\dl^c}}$
   such that for each $\o \in \Omega_\dl$, we have
\begin{align}
\label{eD2}
& \| \wt \In_{2,N}^\o (w_1,w_2,w_3,w_4,w_5) \|_{X^{0,b,\dl}} \notag\\
& \hphantom{XXXXX}
\le C\dl^\ta \prod_{j=1}^5
\Big(  \jb{N_j}^{-\beta} + \|w_j - S(t)\pi_{N_j}^\perp (u_0^\o)\|_{X^{-\g,b,\dl}} \Big),
\end{align}

 \noi
uniformly in     $N_j \in \Z_{\geq -1}$, $j = 1, \dots, 5$,
and $N \geq N_0 (\o, \dl)$ for some $N_0(\o, \dl) \in \N$.
Here, we allow $N = \infty$ as well.
\end{proposition}

We remark that both estimates \eqref{eD1} and \eqref{eD2} 
exhibit some smoothing effect. The main reason is that both nonlinearities 
$\In_{1,N}^\o(w_1,w_2,w_3)$ and $ \wt \In_{2,N}^\o (w_1,w_2,w_3,w_4,w_5)$ possess non-resonant structures. 
In the next subsection, we present the proof of Theorem~\ref{THM:LWP}
by assuming  Propositions~\ref{PROP:I1} and \ref{PROP:I2}.
We present the proofs of these propositions in Sections~\ref{SEC:Non1} and~\ref{SEC:Non2}.
By careful analysis, 
we  reduce these nonlinear estimates 
to boundedness properties
of certain random multilinear functionals
of the white noise.

\begin{remark}\label{mod_scat_det}
\rm 
In deriving  $\EE_n(w, w, w, w)$ in \eqref{En}, 
we made  use of a key cancellation:
\begin{equation}\label{cancel1}
\Re 
\Big( i 
\FF\big( \NN_2^\o (w) \big) (n)\cj{\ft w (n)}
\Big)=0,
\end{equation}

\noi
i.e.~the resonant part of the nonlinearity disappears in   \eqref{En}. Interestingly, a similar cancellation is used in the context of the modified scattering analysis of the one-dimensional cubic nonlinear 
Schr\"odinger equation on the real line: 
\begin{equation}\label{modif_scat}
i\partial_t u=\partial_x^2 u+|u|^2u
\end{equation}
with localized initial data. 
More precisely, if we set  $v(t)=e^{it\dx^2}u(t)$, 
then by a stationary phase argument, 
\eqref{modif_scat} can be rewritten as
\begin{equation}\label{1d}
\partial_{t} \ft v (\xi, t)=cit^{-1}|\ft v (\xi, t)|^2\ft v (\xi, t)+R(\xi, t),\qquad \xi\in\R,
\end{equation}

\noi
where $c$ is a real constant and $\ft v$ denotes the Fourier transform of $v$ on the real line.  
The trilinear  remainder  term $R(\xi, t)$ decays (in a suitable functional framework) faster than $t^{-1}$
and therefore the principal part of the nonlinearity 
for analyzing long-time behavior  is given by 
$cit^{-1}|\ft v (\xi, t)|^2\ft v (\xi, t)$,  which is the analogue of the resonant part of the nonlinearity 
$ \NN_2^\o (w) $ in our problem. 
Note that the key cancellation in the context of
\eqref{1d} is
\begin{equation}\label{cancel2}
\Re
\Big( i 
t^{-1}
|\ft v (\xi, t)|^2\ft v (\xi, t)
\cj{\ft v(\xi, t)}
\Big)=0.
\end{equation}

\noi
The cancellation \eqref{cancel2}  appears 
in computing $\partial_{t}|\ft v (\xi, t)|^2$,
 which is the analogue of the computation  \eqref{En} in the context of \eqref{1d}.
We point out strong similarity between 
\eqref{cancel1}
and   \eqref{cancel2}.
\end{remark}

\subsection{Proof of Theorem \ref{THM:LWP}: the  $\alpha =0$ case} 
In this subsection, we present  the proof of Theorem \ref{THM:LWP}
for $\al = 0$.
More precisely, 
by applying Propositions~\ref{PROP:I1} and~\ref{PROP:I2}
to  the iterated Duhamel formulation~\eqref{DuhamelNexpan} 
we prove that,  
for each $0< \dl\ll 1$, 
there exists $\Omega_\dl\subset \Omega$ with $P(\Omega_\dl^c) \le e^{-\frac{1}{\dl^c}}$ such that for $\o\in \Omega_\dl$, the following statements hold:

\begin{itemize}

\item[(i)] The sequence $\{w^N - S(t) u_{0, N}^\o \}_{N \in \N}$ is Cauchy in $X^{0,\frac12+,\dl}$.

\smallskip

\item[(ii)] The limit $w$ of $w^N$ 
satisfies the equation  \eqref{NLS2-3} in the  distributional sense  with 
the white noise initial data $u_0^\o$.

\smallskip

\item[(iii)] The solution $w$ is unique in the class:
$S(t) u_0^\o + B_1$,
where $B_1$ denotes the ball of radius 1 in 
$X^{0,\frac12+,\dl}$ centered at the origin.

\end{itemize}

Given  $0< \beta, \g \ll 1$ 
 and   $b > \frac 12$  sufficiently close to $\frac 12$, 
 apply Propositions~\ref{PROP:I1} and~\ref{PROP:I2}
and construct a set $\O_\dl \subset \O$ 
with 
 $P(\Omega_\dl^c) < e^{-\frac{1}{\dl^c}}$
for each $0 < \dl \ll 1$
such that the conclusions of both Propositions~\ref{PROP:I1} and~\ref{PROP:I2} hold.
In the following, 
we fix $\o   \in \O_\dl$ 
and hence the parameter $N_{0}(\o,\dl)$ in Propositions \ref{PROP:I1} and \ref{PROP:I2} is a {\it fixed} number.
In what follows, unless otherwise stated, the number $N$ and $M$ are always 
assumed to be greater than $N_0(\o,\dl)$.

\smallskip

\noi
(i) 
By setting $v^N = w^N - S(t) u_{0, N}^\o$, 
it follows from  \eqref{DuhamelN} and \eqref{DuhamelNexpan}
that  $v^N$ satisfies 
\begin{align}
\label{vN}
v^N & = \In_{1,N}^\o(v^N + S(t) u_{0, N}^\o) + \In_{2,N}^\o(v^N+S(t) u_{0, N}^\o)\notag\\
& = \In_{1,N}^\o(v^N + S(t) u_{0, N}^\o) + \wt \In^\o_{2,N}(v^N+S(t) u_{0, N}^\o),
\end{align}

\noi
where $\In_{1,N}^\o$, $\In_{2,N}^\o$ and $\wt \In^\o_{2,N}$ 
are as  in \eqref{I1N}, \eqref{I2N} and \eqref{MultiI2N}. 
Note that the second equality holds
since $w^N$  is a classical solution to \eqref{NLS2-3'}.

We first claim that 
\begin{align}
\label{vNbound}
\|v^N \|_{X^{0,\frac12+,\dl}} \le 1
\end{align} 

\noi  
by choosing $\dl > 0$ sufficiently small.
Indeed, by 
applying 
\eqref{eD1} and \eqref{eD2}
in Propositions~\ref{PROP:I1} and~\ref{PROP:I2}
(with $N_j = -1$, i.e. $\pi_{N_j}^\perp = \Id$)
to \eqref{vN}, we have 
\begin{align}
\label{T1-3}
\|v^N\|_{X^{0,\frac12+,\dl}} 
& \les \dl^{\theta} (1+\|v^N\|_{X^{-\g,\frac12+,\dl}})^3 + \dl^{\theta}  (1+\|v^N\|_{X^{-\g,\frac12+,\dl}})^5\notag\\
& \le \dl^{\theta} (1+\|v^N\|_{X^{0,\frac12+,\dl}})^3 + \dl^{\theta}  (1+\|v^N\|_{X^{0,\frac12+,\dl}})^5.
\end{align}
Then by choosing $\dl > 0$  sufficiently small,
the bound  \eqref{vNbound} follows from \eqref{T1-3} and a standard continuity argument.

Next, we show that the sequence $\{v^N\}_{N\in \N}$ is a Cauchy sequence in $X^{0,\frac12+,\dl}$. 
 By possibly restricting to  smaller $\dl >0$, 
 we prove 
\begin{align}
\label{diff}
\|v^M -v^N\|_{X^{0,\frac12+,\dl}} \les N^{-\min(\be, \g)}
\end{align}

\noi
for any  $\o \in \Omega_\dl$ and $M\geq N \geq N_0(\o,\dl)$.
The bound  \eqref{diff}
shows that  $v^N$ converge in $X^{0,\frac12+,\dl}$ for each $\o \in \Omega_\dl$
and  thus $w^N = v^N + S(t) u_{0, N}^\o$ converge 
to $w = v + S(t) u_0^\o$
in $C([-\dl, \dl]; H^s(\T))$, $s < -\frac 12$.

We now prove  \eqref{diff}. 
From \eqref{vN}, we have
\begin{align}
\label{diff1}
\|v^M -v^N\|_{X^{0,\frac12+,\dl}} 
& \le \| \In_{1,M}^\o(v^M + S(t) u_{0, M}^\o) 
- \In^\o_{1,N}(v^N + S(t) u^\o_{0,N})\|_{X^{0,\frac12+,\dl}} \notag\\
& \hphantom{X}
+ \| \wt \In^\o_{2,M}(v^M + S(t) u_{0, M}^\o) 
- \wt \In^\o_{2,N}(v^N + S(t) u^\o_{0,N})\|_{X^{0,\frac12+,\dl}}.
\end{align}

\noi
We first estimate the first term on the right-hand side of \eqref{diff1}.
From \eqref{I0N} and \eqref{I1N}
with  $w^N = v^N + S(t) u_{0, N}^\o$, we have
\begin{align}
  \| & \In_{1,M}^\o(w^M) - \In^\o_{1,N}(w^N)\|_{X^{0,\frac12+,\dl}} \notag \\
 & \le  \| \In_{1,M}^\o(w^M) - \In^\o_{1,N}(w^M)\|_{X^{0,\frac12+,\dl}} 
 +  \| \In_{1,N}^\o(w^M) - \In_{1,N}^\o(w^N)\|_{X^{0,\frac12+,\dl}} \notag\\
 & \le \| \In_{1,M}^\o(w^M) - \In^\o_{1,N}(w^M)\|_{X^{0,\frac12+,\dl}} 
 +  \| \In_{1,N}^\o(w^M - w^N, w^M,w^M) \|_{X^{0,\frac12+,\dl}} \notag\\
 & \hphantom{X}
  + \| \In_{1,N}^\o(w^N,w^M - w^N,w^M) \|_{X^{0,\frac12+,\dl}} 
 +\| \In_{1,N}^\o(w^N,w^N,w^M - w^N) \|_{X^{0,\frac12+,\dl}} .
 \label{T1-4}
\end{align}

\noi
In the following, 
we only treat the first two terms since the other two terms can be treated in a similar manner.
Using the trilinear structure of $\In_{1,L}^\o$ for $L\in \{M,N\}$,
we have
\begin{align*}
 \In_{1,L}^\o( w^M) 
 & = \In_{1,L}^\o( \pi_{ \frac N3}^\perp w^M,w^M,w^M) 
 +\In_{1,L}^\o( \pi_{ \frac N3}w^M, \pi_{\frac N3}^\perp w^M,w^M)
 \\
 & \hphantom{X}
 + \In_{1,L}^\o(\pi_{\frac N3}w^M, \pi_{\frac N3}w^M, \pi_{\frac N3}^\perp w^M) 
+ \In_{1,L}^\o(\pi_{\frac N3}w^M, \pi_{\frac N3}w^M, \pi_{\frac N3}w^M).
\end{align*}

\noi
The key point is to observe that it follows directly from the definitions
\eqref{I0N} and \eqref{I1N} with \eqref{PsiN} that for $M\geq N$
$$
\In_{1,M}^\o(\pi_{\frac N3}w^M, \pi_{\frac N3}w^M, \pi_{\frac N3}w^M) -
 \In_{1,N}^\o(\pi_{\frac N3}w^M, \pi_{\frac N3}w^M, \pi_{\frac N3}w^M)=0.
 $$
 
 \noi
 Therefore,  in order to control 
 $$
 \| \In_{1,M}^\o(w^M) - \In^\o_{1,N}(w^M)\|_{X^{0,\frac12+,\dl}},
 $$ 
 
 \noi
 we only need to bound
\begin{align*}
&\| \In_{1,L}^\o(  \pi_{\frac N3}^\perp w^M,w^M,w^M) \|_{X^{0,\frac12+,\dl}}, \\
&\| \In_{1,L}^\o(  \pi_{\frac N3} w^M,  \pi_{\frac N3}^\perp w^M, w^M) \|_{X^{0,\frac12+,\dl}}, \\
&\| \In_{1,L}^\o( \pi_{\frac N3} w^M,  \pi_{\frac N3} w^M,  \pi_{\frac N3}^\perp w^M)\|_{X^{0,\frac12+,\dl}}
\end{align*}

\noi
for $L = M$ and $N$.
We only consider the first one
since  the others can be treated similarly.  From Proposition~\ref{PROP:I1}
and \eqref{vNbound}, we have
\begin{align*}
\| \In_{1,L}^\o( \pi_{\frac N3}^\perp w^M,w^M,w^M) \|_{X^{0,\frac12+,\dl}} 
& \les 
\dl^{\theta}  \big( N^{-\beta}+ \|\pi_{\frac N3}^\perp v^M \|_{X^{-\g,\frac12+,\dl}}  \big)
\big ( 1+ \| v^M \|_{X^{-\g,\frac12+,\dl}}  \big)^2 \notag\\
& \les 
\dl^{\theta}  \big( N^{-\beta}+ N^{-\g}\| v^M \|_{X^{0,\frac12+,\dl}}  \big)\\ 
& \les \dl^{\theta} N^{- \min(\beta,\g)}, 
\end{align*}

\noi
where we used the fact that $w^N = v^N + S(t) u_{0, N}^\o$.
Therefore,  we obtain
\begin{align*}
 \| \In_{1,M}^\o(w^M) - \In^\o_{1, N}(w^M)\|_{X^{0,\frac12+,\dl}} 
 \les \dl^{\theta} N^{- \min(\beta,\g)}.
\end{align*}

Next,  we proceed with estimating the second term 
on the right-hand side of  \eqref{T1-4}:
\begin{align}
\| \In_{1,N}^\o & (w^M - w^N, w^M,w^M) \|_{X^{0,\frac12+,\dl}} \notag\\
& \leq  \| \In_{1,N}^\o(v^M-v^N + S(t) \pi_{N}^\perp u_{0}^\o, v_N + S(t)u_{0, N}^\o, v_N + S(t)u_{0, N}^\o) \|_{X^{0,\frac12+,\dl}}\notag \\
&
\hphantom{X}
+  \| \In_{1,N}^\o( S(t) \pi_{M}^\perp u_{0}^\o, v_N + S(t)u_{0, N}^\o, v_N + S(t)u_{0, N}^\o) \|_{X^{0,\frac12+,\dl}}.
\label{T1-6}
\end{align}

\noi
By applying  Proposition \ref{PROP:I1} to \eqref{T1-6} 
with $N_1 = N$ or $M$ 
and $N_2 =N_3 = -1$, we obtain
\begin{align}
\label{T1-7}
\| \In_{1,N}^\o & (w^M - w^N, w^M,w^M) \|_{X^{0,\frac12+,\dl}} \notag\\
& \les  \dl^{\theta} \Big( N^{-\beta} + \| v^M-v^N \|_{X^{-\g,\frac12+,\dl}} \Big) 
\Big(1+\| v^M \|_{X^{-\g,\frac12+,\dl}}\Big)^2\notag\\
& \les \dl^{\theta} \Big( N^{-\beta} + \| v^M-v^N \|_{X^{-\g,\frac12+,\dl}} \Big). 
\end{align}

\noi
Similarly, 
we can estimate the second term on the right-hand side of \eqref{diff1}
by applying Proposition \ref{PROP:I2} and obtain
\begin{align}
\label{T1-7''}
 \| \wt \In^\o_{2,M}(  w^M) & - \wt \In^\o_{2,N}(w^N)\|_{X^{0,\frac12+,\dl}} \notag \\ 
 &  \les \dl^{\theta} \Big( N^{-\min(\beta,\g)}   +  \| v^M-v^N \|_{X^{0,\frac12+,\dl}}\Big).
\end{align}

\noi
Putting  \eqref{diff1},  \eqref{T1-7},  and \eqref{T1-7''} together, we obtain
\begin{align*}
\|v^M -v^N\|_{X^{0,\frac12+,\dl}} \le  C\dl^{\theta}  N^{- \min(\beta,\g)}  + C\dl^{\theta} \| v^N-v^M \|_{X^{0,\frac12+,\dl}}  .
\end{align*}

\noi
Therefore, by choosing $\dl > 0$ sufficiently small, 
we obtain 
 \eqref{diff}.

\smallskip

\noi
(ii)
Next, we show that the limit
 $w = v + S(t) u_0^\o$
 satisfies the Duhamel formulation \eqref{Duhamel2}:
\begin{align}
\label{diseq}
w = S(t) u_0^\o + \In_1(w) + \In_2(w),
\end{align}

\noi
in the distributional sense, locally in time.  
We first recall the following definition of the Fourier-Lebesgue spaces
$\FL^{s, p}(\T)$.
Given $s \in \R$ and $1\leq p \leq \infty$, 
define
 the Fourier-Lebesgue space $\F L^{s, p}(\T)$  by the norm:
\begin{align*}
\|f\|_{\FL^{s,p}(\T)}:=\|\jb{n}^{s}\ft f(n)\|_{\l^{p}_{n}(\Z)}.
\end{align*}

\noi
Then, it is easy to see that the white noise 
$u_0^\o$ in \eqref{IV1} (with $\al = 0$)
almost surely belongs to $\FL^{s, p}(\T)$ if and only if $sp < -1$ 
with the understanding that $s < 0$ when $p = \infty$.

Given $0 < \dl \ll 1$, 
let  $\o\in \Omega_\dl$.
Then, it follows from Lemma \ref {LEM:prob}\footnote
{Note that Lemma \ref {LEM:prob} appears in the proof of Propositions
\ref{PROP:I1} and \ref{PROP:I2} (see also Lemma \ref{LEM:NRSum1})
and thus we may assume that the conclusion of Lemma \ref {LEM:prob}
holds on the set $\O_\dl$ constructed in Step (i).}
that  the truncated random linear solution $S(t) u_{0, N}^\o$ converges
to $ S(t) u_0^\o$ in 
$C([-\dl, \dl];   \FL^{-\eps,\infty}(\T))$
for any $\eps > 0$. 
The residual part  $v^N$ converges to $v$ in  $X^{0,\frac12+,\dl}$, 
and hence in $C([-\dl, \dl]; L^2(\T))$.
Putting together, we see that $w^N$ converges to 
$w$ 
in $C([-\dl, \dl];   \FL^{-\eps,\infty}(\T))$.
Hence, from the definitions \eqref{I2} and \eqref{I2N}
of $\In_2$ and $\In_{2, N}$, 
we conclude that 
$\In_{2,N} (w^N)$ converges to $\In_2(w)$
in $C([-\dl, \dl];   \FL^{-3\eps,\infty}(\T))$.
On the other hand, 
from \eqref{T1-4}, 
we see that 
that  $\In_{1,N} (w^N)$ converges to $\In_1(w) $ in 
$X^{0,\frac12+,\dl}$. 
Together with the convergence of $w^N$ to $w$, 
we have shown that 
 each term in the truncated Duhamel formulation \eqref{DuhamelN}
converges to 
 the corresponding term in \eqref{diseq}.
Recalling that $w^N$
satisfies \eqref{DuhamelN}, 
we conclude that $w$ is a solution to the Duhamel formulation \eqref{diseq}
in the distributional sense.

In Step (i), we already showed that $w$ satisfies the iterated formulation \eqref{ExpandI2}.
Thus, as a byproduct, we have verified that
\begin{align*}
\In_2 (w) = \wt \In_2 (w),
\end{align*}

\noi
for the solution $w$ constructed in Step (i).

\smallskip

\noi
(iii) 
Lastly, we turn to the uniqueness issue.
Given $0 < \dl \ll 1$, 
fix  $\o\in \Omega_\dl$.
Let $w = S(t) u_0^\o + v$ be the solution to 
 \eqref{Duhamel2} with the white noise initial data $u_0^\o$ 
 constructed in Steps~(i) and~(ii).
 Suppose that there exists another solution $\wt w$
 to \eqref{Duhamel2}
 of the form
 $\wt w = S(t) u_0^\o + \wt v$
 for some $\wt v \in B_1 \subset  X^{0,\frac12+,\dl}$.
Since such $\wt w$ is also a solution to  \eqref{NLS2-3}, 
by repeating the argument in Subsection \ref{SUBSEC:LWPx}, 
we see that $\wt w$ satisfies
the iterated Duhamel formulation \eqref{Duhamel3}:
\begin{align*}
\wt w = S(t) u_0^\o + \In_1(\wt w) + \wt \In_2(\wt w).
\end{align*}

 \noi
 Then, by repeating the argument in Step (i) 
 with Propositions \ref{PROP:I1} and \ref{PROP:I2}, 
we obtain 
\begin{align*}
\|v -\wt v\|_{X^{0, \frac12+,\dl}} \le C \dl^{\theta} \|v -\wt v\|_{X^{0, \frac12+,\dl}}  \le \frac12 \|v -\wt v\|_{X^{0, \frac12+,\dl}}
\end{align*}
for $\dl>0$ sufficiently small,
yielding  $v = \wt v$ in $X^{0, \frac12+,\dl}$.
This proves uniqueness in the  class $S(t)u_0^\o + B_1$.

This completes the proof of Theorem \ref{THM:LWP} when $\al = 0$.

\begin{remark} \label{REM:uniq2}\rm
(i) By a continuity argument, 
we can easily upgrade the uniqueness of $w$ in $S(t) u_0^\o + B_1$
to  uniqueness of $w$ in the class
\[ S(t) u_0^\o + X^{0, \frac 12 + , \dl}\]

\noi
See Remark 2.9 in \cite{CGKO}.
By inverting the random gauge transform $\J^\o$ in \eqref{Ga}, 
we then obtain uniqueness of $u$
in the class
\[ Z(u_0^\o) + X^{0, \frac 12 + , \dl}_{-, \o}\]

\noi
where $Z$ is as in \eqref{z} and 
$X^{0, \frac 12 + , \dl}_{-, \o}$ is the local-in-time version
of the {\it random} Fourier restriction norm space $X^{0, \frac 12 + }_{-, \o}$ defined in \eqref{rxsb1}.

\smallskip

\noi
(ii)
Let $u_{0, m}^\o = u_0^\o * \rho_m$ be the regularization of the white noise $u_0^\o$ 
by mollification via a mollification kernel $\rho_m$ in \eqref{smooth1}.
Denote by $w_m$ the solution to 
the gauged equation \eqref{NLS2-3} with $w_m|_{t = 0} = u_{0, m}^\o$.
Then, 
by proceeding as above,\footnote
{Here, our assumption that the symbol $\ft \rho_m \equiv 1$
on $[-c_0m, c_0m]$ for some $c_0 > 0$, independent of $m \in \N$
provides a simplification of the argument as compared to a general mollification kernel.}
 one can easily establish convergence of $w_m$
to $\wt w$ in the class $S(t)u_0^\o + B_1$, satisfying \eqref{Duhamel2}. 
Then, by the uniqueness proved in Step (iii) above, 
we conclude that $w = \wt w$.
This proves independence of the mollification kernel.

\end{remark}

\section{Global well-posedness
and invariance of the white noise measure}
\label{SEC:GWP}

In this section, we extend the local solutions constructed in Theorem~\ref{THM:LWP} 
to global solutions and prove invariance of the white noise measure \eqref{gauss0} with $\alpha = 0$ under the flow of the renormalized 4NLS \eqref{NLS2}. 
The main ingredient is Bourgain's invariant measure argument~\cite{BO94, BO96}.

\subsection{Invariance of the white noise measure under the  truncated 4NLS }
In this section, we will denote the white noise measure by $\mu$.
For fixed $\eps >0$, $\mu$ is a measure on $H^{-\frac12-\eps}(\T)$,
  defined as the pushforward of $P$ under the map from $(\Omega,\F,P)$ to $H^{-\frac12-\eps}(\T)$  
  (equipped with the Borel $\s$-algebra)
given by
\[
\o \longmapsto  u_0^\o=\sum_{n\in \Z}g_{n}(\o)e^{inx}.
\]

\noi
Given $N \in \N$, 
we also define
the finite-dimensional white noise measure $\mu_N$ on 
$
E_N = \text{span}\big\{ e^{inx}, |n|\le N\big\}
$
as the pushforward of $P$ under the map from $(\Omega,\F,P)$ to $E_N$ given by 
$\o \mapsto \pi_N u_0^\o$, 
where $\pi_N$ is the Dirichlet projector onto the frequencies $\{|n|\leq N\}$
defined in~\eqref{pi1}.

Consider the frequency-truncated version of the renormalized 4NLS \eqref{NLS2}:
\begin{align}
\label{NLS2N}
\begin{cases}
i\partial_t u^N =  \partial_x^4 u^N + \pi_N (\NN(u^N))\\
u^N(x,0) = \pi_N u_0(x) \in E_N, 
\end{cases}
\end{align}

\noi
where $\NN(u)$ denotes the renormalized nonlinearity in \eqref{nonlin0}.
It is easy to see that the solution $u^N$ to \eqref{NLS2N} 
exists  globally in time.  Let $\wt \Theta_N (t)$ denote the flow map for \eqref{NLS2N}.
By  the Liouville theorem, 
we see that the truncated white noise measure $\mu_N$ is invariant under  $\wt \Theta_N (t)$.
Following \cite{BTT}, 
we also consider the extension of \eqref{NLS2N} to infinite dimensions, 
where the higher modes evolve according to linear dynamics:
\begin{align}
\label{NLS2N2}
\begin{cases}
i \dt u^N =  \dx^4 u^N + \pi_N (\NN(\pi_N u^N))\\
u^N(x,0) =  u_0(x) \in H^{-\frac12-\eps}(\T).
\end{cases}
\end{align}

\noi
Let $\Theta_N (t)$ denote the flow map for \eqref{NLS2N2}. Then, we have 
\[
\Theta^N(t) = \wt \Theta^N(t) \pi_N + S(t) \pi_N^\perp, 
\]

\noi
where $\pi_N^\perp = \Id - \pi_N$.
Denoting by  $E_N^{\perp}$  the orthogonal complement of $E_N$ in $H^{-\frac12-\eps}(\T)$, 
let  $\mu_N^{\perp}$ be  the white noise measure on $E_N^\perp$ 
(i.e.~the image measure under the map: $\o \mapsto \pi_N^\perp u_0^\o$).
Note that $\mu_N^\perp$ is invariant along the linear flow on $E_N^\perp$ (this is a consequence of the invariance of complex-valued Gaussians under rotations). 
Therefore, by writing
\[ d\mu = d\mu_N \otimes d\mu_N^{\perp}, \] 

\noi
we conclude  the following invariance of $\mu$ under $\Theta_N(t)$.
\begin{lemma}
\label{LEM:Appinvar}
For each $t\in \R$, the white noise measure $\mu$ is invariant under the flow map $ \Theta^N (t)$ on $H^{-\frac12-\eps}(\T)$.
\end{lemma}

\subsection{Almost sure global well-posedness}
By using the invariance of the  white noise measure
for \eqref{NLS2N2}
(Lemma \ref{LEM:Appinvar}) and a PDE approximation argument, 
we have the following lemma,
guaranteeing long time existence with large probability 
for the renormalized 4NLS~\eqref{NLS2}.

\begin{lemma}
\label{LEM:AlmostGWP}
There exist small 
$0 < \eps < \eps_1 \ll1$
and $\be > 0$ 
such that given any small  $\kk > 0$ and $T>0$, 
there exists  a measurable set $\Si_{\kk,T} \subset H^{-\frac12-\eps}(\T)$
 such that 
 \textup{(i)}
 $\mu(\Si_{\kk ,T}^c) < \kk$
 and \textup{(ii)}
  for any $u_0 \in \Si_{\kk, T}$, there exists a \textup{(}unique\textup{)} solution 
$$
u\in Z(u_0) + C([-T,T];L^{2}(\T)) \subset C([-T,T];H^{-\frac12-\eps}(\T))
$$ 

\noi
to the renormalized 4NLS \eqref{NLS2} with $u|_{t = 0} = u_0$, 
where $Z$ is defined in \eqref{z}.
Furthermore, given any large $N  \gg 1$, 
 we have
\begin{align*}
\Big\| u(t) - \Theta^N (t) (u_0) \Big\|_{C([-T,T]: H^{ - \frac 12 - \eps_1} (\T))} \les C(\kk , T) N^{-\beta},
\end{align*}

\noi
where $\Theta^N(t)$ denotes the flow map for 
 \eqref{NLS2N2}.

\end{lemma}

For the uniqueness statement, see Remark \ref{REM:uniq2}\,(i).

\begin{proof}
Once we have almost sure local well-posedness
(Theorem~\ref{THM:LWP}), the proof of Lemma~\ref{LEM:AlmostGWP} is 
by now  standard. 
In the following, we only sketch key parts of the argument
and refer to \cite{BO94, BO96, BT2, R0, R}
for further details.

Given  a solution $u^N$ to  \eqref{NLS2N2}, 
we define  $w^N = \J_N^\o(u^N)$ as in the proof of Theorem~\ref{THM:LWP},
where $\J_N^\o$ denotes the truncated random gauge transform in \eqref{GaN}.
 Namely,  
we have
$$
w^N(x, t)= \sum_{n\in\Z} e^{inx - it|g_n^N(\o)|^2} \ft{u^N}(n, t),
$$

\noi
where  $g_n^N$ is as in  \eqref{gN}.
The key observation is that convergence properties of $w^N$ 
in a Fourier lattice\footnote
{Namely, in a space
where a norm depends only on the sizes of the Fourier coefficients.
For example,  $H^s(\T)$ and $\FL^{s, p}(\T)$.}
 can be directly converted to convergence properties of $u^N$. 
For $M > N \geq 1$, write
\begin{align*}
 w^M - w^N 
 & = \big(\pi_M w^M -   \pi_N w^N\big)
+ \pi_M^\perp w^M - \pi_N^\perp w^N.
\end{align*}

\noi
The convergence of 
$ (\pi_M w^M - S(t) u_{0, M}^\o)-   (\pi_N w^N -  S(t) u_{0, N}^\o)$
can be shown  exactly as in the proof of Theorem~\ref{THM:LWP}, 
locally in time, i.e.~in  $X^{0, \frac{1}{2}+, \dl} \subset C([-\dl, \dl];  L^2(\T))$,
which yields 
convergence of 
$ \pi_M w^M -   \pi_N w^N$
in $ C([-\dl, \dl];  H^{-\frac{1}{2}-\eps}(\T))$.
On the other hand, the second and third terms
decay like $N^{-\be}$ for some $\be > 0$
thanks to the high frequency projections.
The remaining part of the argument leading to the proof of Lemma~\ref{LEM:AlmostGWP} is contained in
\cite{BO94,BO96,BT2, R0, R}.
In particular, see the proof of Proposition~3.5 in \cite{R0} for  details
in a setting analogous to our work.
\end{proof}


Once we have Lemma \ref{LEM:AlmostGWP}, 
the desired almost sure global well-posedness follows
from the Borel-Cantelli lemma.
Given $\kk > 0$, 
let $T_j = 2^j$ and $\kk_j = \frac\kk{2^j}$, $j \in \N$.
By applying Lemma~\ref{LEM:AlmostGWP}, 
construct a set $ \Si_{\kk_j, T_j}$
and set
\begin{align}
 \Sigma_\kk : = \bigcap_{j=1}^{\infty} \Si_{\kk_j, T_j}.
\label{Sik}
 \end{align}

\noi
Then, we have $\mu(\Si_\kk^c) < \kk$
and 
  for any $u_0 \in \Si_{\kk}$, 
  there exists a unique global-in-time solution to 
 the renormalized 4NLS \eqref{NLS2} with $u|_{t = 0} = u_0$.
Finally, set
\[ \Si : = \bigcup_{n = 1}^\infty \Si_{\frac 1n }.\]

\noi
Then,  we have $\mu(\Si^c) = 0$
and 
  for any $u_0 \in \Si$, 
  there exists a unique global-in-time solution to 
 the renormalized 4NLS \eqref{NLS2} with $u|_{t = 0} = u_0$.
 This proves almost sure global well-posedness.

\subsection{Invariance of the white noise measure}


Let $\Theta(t)$ be the flow map for the renormalized 4NLS  \eqref{NLS2} 
defined on the set $\Si$ of full probability constructed above.
Our goal here is to show
that 
\begin{align}\label{Invariance}
\intt_{\Sigma} F\big(\Theta(t) (u)\big) d\mu(u) = \intt_\Sigma F(u) d\mu (u)
\end{align}

\noi
for any $F \in L^1(H^{-\frac12 -\eps}(\T),d\mu)$
and any $t\in \R$. 
By a density argument, it suffices to prove~\eqref{Invariance}
for continuous and bounded $F$.

Fix $t \in \R$. By  Lemma~\ref{LEM:Appinvar}, we have 
\begin{align}
\label{Inv1}
\intt_{\Sigma} F\big(\Theta^N(t) (u)\big) d\mu(u) = \intt_{\Sigma} F ( u) d\mu(u).
\end{align}

\noi
Fix small $\dl>0$.  
The boundedness of $F$ implies that for any sufficiently small $\kk >0$, we have
\begin{align}\label{Inv2}
  \Bigg| \intt_{\Si_\kk^c} F\big(\Theta(t) (u)\big) d\mu(u)\Bigg| + 
  \Bigg| \intt_{\Si_\kk^c} F\big(\Theta^N(t) (u)\big) d\mu(u)  \Bigg| < \dl, 
\end{align}

\noi
where $\Si_\kk$ is as in \eqref{Sik}.
Fix one such $\kk > 0$.
Then, by 
 Lemma \ref{LEM:AlmostGWP}, 
we have
\[
 \| \Theta(t) (u) - \Theta^N(t) (u)\|_{H^{-\frac12 -\eps}} \le C(\kk, t) N^{-\beta}
\]

\noi
for  any $u\in \Sigma_\kk$
and sufficiently large $N \gg1$.
Hence,  by continuity of $F$, we have
\begin{align}
\label{Inv3}
  \Bigg| \intt_{\Sigma_\kk} F\big(\Theta(t) (u)\big) d\mu(u) - \intt_{\Sigma_\kk} F\big(\Theta^N(t) (u)\big) d\mu(u)  \Bigg| < \dl,
\end{align}

\noi
for any sufficiently large
 $N \gg1$.
 Combining \eqref{Inv1}, \eqref{Inv2},  and \eqref{Inv3}
 and taking $\dl \to 0$, 
 we obtain \eqref{Invariance}.

\subsection{Proof of Theorem~\ref{THM:soft_version}}
The proof of Theorem~\ref{THM:soft_version} follows from the arguments presented in the proofs of Theorems~\ref{THM:LWP} and~\ref{THM:GWP}. 

\section{Nonlinear estimate I: non-resonant part}\label{SEC:Non1}

In this section, we present the proof of Proposition~\ref{PROP:I1}. 

\subsection{Probabilistic estimates}\label{SUBSEC:Non1a}

We begin by presenting several probabilistic estimates
that will be used to prove Proposition \ref{PROP:I1}.
The proofs of these lemmas are presented in Appendix~\ref{SEC:A}.

We first recall some notations.
Let $\eta \in C^\infty_c(\mathbb{R})$ be a smooth non-negative cutoff function supported on $[-2, 2]$ with $\eta \equiv 1$ on $[-1, 1]$.
Recall from \eqref{Gam0},  \eqref{phi1}, and \eqref{PsiN}  that 
\begin{align} 
\G(n) & = \{(n_1, n_2, n_3) \in \Z^3:\,  n = n_1 - n_2 + n_3 \text{ and }  n_1, n_3 \ne n\}, 
 \notag\\
 \Phi(\bar n) &= \Phi(n_1, n_2, n_3, n) = n_1^4 - n_2^4 + n_3^4 - n^4, 
\notag 
\\
\Psi^\o_N(\bar n) & = |g_{n_1}^N(\o)|^2 - |g_{n_2}^N(\o)|^2 + |g_{n_3}^N(\o)|^2 - |g_{n}^N(\o)|^2. \label{Psi2}
\end{align}

\noi
where $g_n^N$ is as in \eqref{gN}.
Given $s, b \in \R$ and $\dl > 0$, the following random functionals
  $S^{s,b, \dl}_{j,N}$,  $ j = 1, 2, 3$, play an important role
  in the proof of Proposition \ref{PROP:I1}
  (and also
    in the proof of Proposition \ref{PROP:I2} presented in Section \ref{SEC:Non2}):
 \begin{align}
 S^{s,b, \dl}_{1,N}(f) 
&  = \Bigg\| 
  \sum_{\substack{n_1 \in \Z \\(n_1,n_2,n_3)\in \G(n)}} \ft f(n_1)   
\frac{  \ft \eta_{_\dl} (\tau + \Phi(\bar n) - | g_{n_1}^N |^2)}{  \jb{n_2}^s \jb{n_3}^s   \jb{n}^{2s}\jb{\tau}^{b}} \Bigg\|_{\l^2_{n,n_2,n_3}L^2_\tau}
\label{A1N}\\
\intertext{(observe that there is at most one term in the $n_1$ summation),}
 S^{s,b, \dl}_{2,N}(f_1,f_2) 
&  =  
\Bigg\| 
\sum_{\substack{n_1, n_2 \in \Z\\ (n_1,n_2,n_3)\in \G(n)}}  \ft f_1 (n_1)  \cj{\ft f_2 (n_2) }
\notag\\
& \hphantom{XXXX}
\times 
 \frac{ \ft \eta_{_\dl} (\tau + \Phi(\bar n)
-  |{g}^N_{n_1}|^2 + |{g}^N_{n_2}|^2 )}{   \jb{n_3}^s  \jb{n}^{2s} \jb{\tau}^{b}}  \Bigg\|_{\l^2_{n,n_3}L^2_\tau } , 
\label{A2N}\\
 S^{s,b, \dl}_{3,N}(f_1,f_2,f_3) 
&  =   \Bigg\| 
  \sum_{\G(n)}  \ft f_1 (n_1) \cj{\ft f_2 (n_2) }\ft f_3 (n_3)  
\notag\\
& \hphantom{XXXX}
\times
\frac{\ft \eta_{_\dl}(\tau + \Phi(\bar n)
-  |{g}^N_{n_1}|^2 + |{g}^N_{n_2}|^2
-  |{g}^N_{n_3}|^2)
}{ \jb{n}^{2s} \jb{\tau}^{b}}  \Bigg\|_{\l^2_n L^2_{\tau}}.
\label{A3N}
\end{align}

\noi
In the following, we will take $f_1,f_2,f_3$  as the white noise 
\begin{align}
f_1=f_2=f_3 =u_0^\o =\sum_{n\in\Z} g_n(\o) e^{inx},
\label{white00}
\end{align} 

\noi
or its frequency truncated version (projected onto high frequencies)
\begin{align*}
\pi_{N_j}^\perp (u_0^\o) = \sum_{|n| > N_j} g_n(\o) e^{inx}.
\end{align*}

\noi
For simplicity of notations, we
set\footnote
{Strictly speaking, we should denote the dependence
of $S^{s,b, \dl}_{j,N}(\o)$ on the parameters $N_1, N_2$, and $N_3$.
For simplicity of the presentation, however, we suppress such dependence
unless it plays an important role.}
\begin{align}
S^{s,b, \dl}_{1,N}(\o) & := S^{s,b, \dl}_{1,N}(\pi_{N_1}^\perp(u_0^\o)), \label{S1}\\
S^{s,b, \dl}_{2,N}(\o) & := S^{s,b, \dl}_{2,N}(\pi_{N_1}^\perp(u_0^\o), \pi_{N_2}^\perp(u_0^\o)), 
\label{S2}\\
S^{s,b, \dl}_{3,N}(\o) & := S^{s,b, \dl}_{3,N}(\pi_{N_1}^\perp(u_0^\o), \pi_{N_2}^\perp(u_0^\o), \pi_{N_3}^\perp(u_0^\o) ), 
\label{S3}
\end{align}

\noi
for $N_1, N_2, N_3 \in \Z_{\geq -1}$
(recall
our convention: 
$\pi^\perp_{-1} = \Id$).
With the notations defined above, 
we have the following tail estimates
for these random functionals.

\begin{lemma}
\label{LEM:NRSum}
Let $s < 0$,  $b < \frac 12$, and $\be > 0$
such that $s$ and $\be$ are sufficiently close to $0$
and  $b$ is sufficiently close to $\frac 12$.
Then, there exist $c, \kk>0$
and small $\dl_0 > 0$  such that the following statements holds.

\smallskip
\noi
\textup{(i)} We have 
\begin{align*}
&P\bigg( \bigg\{\o\in \Omega: 
\sup_{N \in \N} \sup_{N_1 \in \Z_{\geq-1}} 
\jb{N_1}^\be |S^{s,b, \dl}_{1,N}(\o)|>   \dl^\kk \bigg\} \bigg) <  e^{-\frac{1}{\dl^c}}
\end{align*}

\noi
for any $0 < \dl < \dl_0$.

\smallskip
\noi
\textup{(ii)} Let $k = 2, 3$.
Given $0 < \dl < \dl_0$, 
define the sets $\A_k$ by 
\begin{align*}
\A_k := \bigg\{ \o\in \Omega: 
\ & \text{there exists } N_0 = N_0(\o, \dl) \in \N
\text{ such that }\\
& \sup_{N \ge N_0} \sup_{\substack{N_j \in \Z_{\geq-1}\\ j = 1, \cdots, k}} 
\bigg(\prod_{j = 1}^k \jb{N_j}^\be \bigg)
|S^{s,b, \dl}_{k,N}(\o)|
\leq  \dl^\kk \bigg\}.
\end{align*}

\noi
Then, we have 
\begin{align*}
&P( \A_k^c ) < e^{-\frac{1}{\dl^c}}
\end{align*}

\noi
for any $0 < \dl < \dl_0$.

\end{lemma}

Given $N \in \N \cup\{\infty\}$, 
 we introduce a random version $X^{s,b}_+(\o,N)$ of the $X^{s,b}$-space:
\begin{align*}
\|u\|_{X^{s,b}_+(\o,N)} = \|\jb{n}^s \jb{\tau + n^4 + |g_n^N (\o)|^2}^b  \ft u (n,\tau) \|_{\l^2_n L^2_{\tau}} 
\end{align*}

\noi
with the understanding that $g_n^\infty = g_n$.
By slightly losing spatial regularity, 
we can control the 
 random $X^{s, b}$-norm by the standard $X^{\s, b}$-norm 
 (with $\s > s$) {\it uniformly} in $u \in X^{\s, b}$.

\begin{lemma}
\label{LEM:RXsb}
Let $\s > s$ and $b > 0$.
Then,
for each $K > 0$, there exists a set $\O_K\subset \O$ with $P(\O_K^c) < C e^{-cK^\frac{1}{b}}$
such that 
\begin{align*}
  \sup_{N \in \N \cup\{\infty\}}\|  u\|_{X^{s,b}_+ (\o,N)} \les  (1+K) \|u\|_{X^{\s,b}}
\end{align*}

\noi
In particular, by choosing $K = \dl^{-\eps}$ for some small $\eps > 0$, 
there exists a set $\O_\dl \subset \O$
with $P(\O_K^c) < C e^{-\frac{1}{\dl^c}}$
such that 
\begin{align*}
  \sup_{N \in \N \cup\{\infty\}}\|  u\|_{X^{s,b}_+ (\o,N)} \les  \dl^{-\eps} \|u\|_{X^{\s,b}}
\end{align*}

\noi
uniformly in $u \in X^{\s, b}$, 
for any $0 < \dl \ll 1$.

\end{lemma}

For the proofs of Lemmas~\ref{LEM:NRSum} and~\ref{LEM:RXsb}, see Appendix \ref{SEC:A}. 
In the next subsection, 
we prove Proposition \ref{PROP:I1}, 
assuming these lemmas.

\subsection{Proof of Proposition \ref{PROP:I1}}

For $j = 1, 2, 3$, let $z_j =  S(t)\pi_{N_j}^\perp( u_0^\o)$
and set $v_j =w_j - z_j$.
Then, by 
 the linear estimate (Lemma \ref{LEM:linear}),  
it suffices to construct $\O_\dl \subset \O$
with 
 $P(\Omega_\dl^c) < e^{-\frac{1}{\dl^c}}$
   such that for each $\o \in \Omega_\dl$, we have, 
   for some  $s < 0$ sufficiently close to $0$, 
\begin{align}
\label{F1}
\|\NN_{1,N}^\o(v_1+z_1, v_2+z_2,v_3+z_3)\|_{X^{0,-\frac12+,\dl}}  
\le C\dl^{\theta} \prod_{j=1}^3 \Big( \jb{N_j}^{-\beta}  + \|v_j \|_{X^{\frac s2,\frac12+,\dl}} \Big)
\end{align}

 \noi
uniformly in     $N_j \in \Z_{\geq -1}$, $j = 1, 2, 3$,
and $N \geq N_0 (\o, \dl)$ for some $N_0(\o, \dl) \in \N$.
By the definition \eqref{Xsb2} of the local-in-time space, 
the estimate \eqref{F1} follows once we prove 
\begin{align}
\label{F2}
\|\eta_{_\dl} (t) \cdot \NN_{1,N}^\o(\wt v_1+z_1, \wt v_2+z_2,\wt v_3+z_3)\|_{X^{0,-\frac12+}}  
\le C\dl^{\theta} \prod_{j=1}^3 \Big( \jb{N_j}^{-\beta}  + \|\wt v_j \|_{X^{\frac s2,\frac12+}} \Big)
\end{align}

\noi
for any extension
$\wt v_j$ of $v_j$ (restricted to the time interval $[-\dl, \dl]$)
onto $\R$, $j = 1, 2, 3$.
For simplicity of notations, 
we denote the extension $\wt v_j$ by $v_j$ 
in the following.

By duality, we have 
\begin{align}
\text{LHS  of } \eqref{F2} 
& =  \sup_{\|a\|_{X^{0,\frac12-}} \le 1}  
\bigg|\int_{\T\times \R}  \eta_{_\dl}(t)\cdot  \NN^\o_{1,N}(v_1+z_1,v_2+z_2,v_3+z_3) \cj{a(x,t)} dx dt \bigg|, 
\label{expandN}
\end{align}

\noi
where $\eta_{_\dl}$ is as in \eqref{eta1}.
By  \eqref{I0N} and expanding the product, 
we write the double integral in~\eqref{expandN} as\footnote{Here and in the following, 
we suppress the time dependence.}
\begin{align*}
\int_\R \eta_{_\dl}  (t) \sum_{n } &  \sum_{\G(n)} 
 e^{it\Psi^\o_N(\bar n)}\Big[ \ft v_1 (n_1) \cj{\ft v_2 (n_2)} \ft v_3 (n_3) \cj{\ft a (n)} \notag\\
&\hphantom{X} +  \ind_{|n_1|> N_1}(n_1) e^{-itn_1^4}g_{n_1} \cj{\ft v_2 (n_2)} \ft v_3 (n_3) \cj{\ft a (n)} + \text{ similar terms}  \notag \\
&\hphantom{X}
+  \bigg(\prod_{j=1}^2 \ind_{|n_j|>N_j}\bigg)    e^{-it(n_1^4 - n_2^4)}g_{n_1}  \cj{g_{n_2}} \ft v_3 (n_3) \cj{\ft a (n)} + \text{ similar terms}   \\
&\hphantom{X}
+ \bigg( \prod_{j=1}^3 \ind_{|n_j|>N_j}   \bigg)  e^{-it(n_1^4 - n_2^4 + n_3^4)}g_{n_1}  \cj{g_{n_2}} g_{n_3} \cj{\ft a (n)} \Big] dt  \notag \\
& =:  \I + \II + \III + \IV, \notag
\end{align*}

\noi
where the term $\I$ consists of the term with all three factors
given by $v_j$'s, 
$\II$ consists of the terms 
with  one factor of $z_j$ and two factors of $v_j$'s, 
$\III$ consists of the terms 
with  two factors of $z_j$'s and one factor of $v_j$, 
and 
$\IV$ consists of the term with all three factors
given by $z_j$'s.

\smallskip

\noi
$\bullet$ {\bf Estimate on $\I$.} 
Define
\begin{align}
\label{IR1}
b_n^{(j)}  & = e^{itn^4+i t |g_n^N|^2} \jb{n}^s \ft v_j (n)
\qquad \text{and}
\qquad a_n = e^{itn^4 +i t |g_n^N|^2} \jb{n}^{2s} \ft a(n), 
\end{align}

\noi
essentially representing the Fourier transforms
of the ungauged interaction representations of $v_j$ and~$a$.
Then, we have
\begin{align*}
\I & =  \int_\R \eta_{_\dl}(t) \sum_n \sum_{\G(n)} e^{it\Psi^\o_N (\bar n)} \ft v_1 (n_1) \cj{\ft v_2 (n_2)} \ft v_3 (n_3) \cj{\ft a (n)} \,dt \notag\\
& =  \sum_n  \sum_{\G(n)} \frac{1}{\jb{n_1}^s \jb{n_2}^s \jb{n_3}^s \jb{n}^{2s}}  
\int_\R \big( \eta_{_\dl}(t) e^{ - it \Phi(\bar n)} \big) 
\, b^{(1)}_{n_1}\cj{b^{(2)}_{n_2} }{ b^{(3)}_{n_3}} \cj{ a_n} \, dt .
\end{align*}

\noi
By Parseval's identity  in the $t$ variable,  we have
\begin{align*}
\I & = \sum_n \sum_{\G(n)}  \frac{1}{\jb{n_1}^s \jb{n_2}^s \jb{n_3}^s \jb{n}^{2s}}  
\int_\R   \ft \eta_{_\dl}(\tau + \Phi(\bar n)) 
\F(b^{(1)}_{n_1}\cj{b^{(2)}_{n_2}} b^{(3)}_{n_3} \cj{a_n})(-\tau)  d\tau.
\end{align*}

\noi
By Cauchy-Schwarz inequality, we have
\begin{align}\label{I1}
\I & \les  
\bigg( \sum_{n} \sum_{\G(n)}   \frac{1}{\jb{n_1}^{2s} \jb{n_2}^{2s} \jb{n_3}^{2s}  \jb{n}^{4s}}
   \bigg\|  \frac{ \ft \eta_{_\dl}(\tau + \Phi(\bar n) )}{\jb{\tau}^{\frac12 -}} \bigg\|_{L^2_\tau}^2  \bigg)^{\frac12}
 \notag\\
 &  \hphantom{X}
 \times 
 \bigg(\sum_{n} \sum_{\G(n)} \Big\|  \jb{\tau}^{\frac12 -} \FF(b^{(1)}_{n_1}\cj{b^{(2)}_{n_2}} b^{(3)}_{n_3} \cj{a_n})(\tau) \Big\|_{L^2_\tau}^2 \bigg)^{\frac12}.
\end{align}

By Lemma \ref{LEM:GTV}
with  \eqref{eta1},  we have
\begin{align}
\label{int-eta}
\bigg\|  \frac{\ft \eta_{_\dl}( \tau - \Phi(\bar n) )}{\jb{\tau}^{\frac12 -\eps}} \bigg\|_{L^2_\tau} 
\lesssim
\bigg( \int \frac{\dl^2}{\jb{\tau}^{1-2\eps}\dl \jb{\tau - \Phi(\bar n)}}d\tau \bigg)^{\frac 12}
\les  \frac{\dl^{\frac12}}{\jb{\Phi(\bar n)}^{\frac{1}{2} - 2\eps}}
\end{align}

\noi
for any small  $\eps >0$.
Then, 
 by \eqref{int-eta} and 
Lemma \ref{LEM:phase} , we can bound 
 the first factor of~\eqref{I1} by 
\begin{align}
 \Bigg( \sum_{n} \sum_{\G(n)}  \frac{1}{\jb{n_1}^{2s} \jb{n_2}^{2s} \jb{n_3}^{2s} \jb{n}^{4s}}  \frac{\dl}{\jb{\Phi(\bar n)}^{1 - }}   \Bigg)^{\frac12} \
  \les  \dl^{\frac12} ,\label{SumPhi}
\end{align}

\noi
provided that $s < 0$ is sufficiently close to $0$.
Next, 
we consider the second factor of \eqref{I1}.
By Lemma \ref{LEM:algebra},  we have 
\begin{align}
\sum_{n} \sum_{\G(n)}  & \Big\|  \jb{\tau}^{\frac12 -} \FF(b^{(1)}_{n_1}\cj{b^{(2)}_{n_2}} b^{(3)}_{n_3} \cj{a_n})(\tau) \Big\|_{L^2_\tau}^2 
=  \sum_{n} \sum_{\G(n)} \Big\| b^{(1)}_{n_1}\cj{b^{(2)}_{n_2}} b^{(3)}_{n_3} \cj{a_n} \Big\|_{H^{\frac12-}}^2 \notag \\
& \les \sum_{n} \sum_{n_1,n_2,n_3} \big\|  b^{(1)}_{n_1} \big\|_{H^{\frac12+}}^2 \big\| b^{(2)}_{n_2} \big\|_{H^{\frac12+}}^2 \big\| b^{(3)}_{n_3} \big\|_{H^{\frac12+}}^2 \big\| a_n \big\|_{H^{\frac12-}}^2 
\notag \\
&  =  \bigg(\sum_{n}  \big\| a_n \big\|_{H^{\frac12-}}^2 \bigg)
 \prod_{j=1}^3 \bigg(\sum_{n_j} \big\| b_{n_j}^{(j)} \big\|_{H^{\frac12+}}^2  \bigg) .
\label{FFTT}
 \end{align}

 \noi
By \eqref{IR1}, 
Plancherel's identity, 
and Lemma \ref{LEM:RXsb}, 
we have that 
\begin{align}
\sum_{n} \big\|  b^{(j)}_{n} \big\|_{H^{\frac12+}}^2 
& =  \sum_{n} \big\| \jb{n}^s e^{itn^4+it|g_n^N|^2} \ft {v_{j}}(n) \big\|_{H^{\frac12+}}^2\notag \\
& =  \sum_{n} \jb{n}^{2s} \big\| \jb{\tau + n^4 + |g_n^N|^2}^{\frac12 +}  \ft v_{j}(n, \tau ) \big\|_{L^2_\tau}^2 \notag\\
& = \|v_j \|^2_{X^{s,\frac12+}_+(\o,N)} \les \dl^{-\eps} \| v_j \|_{X^{\frac{s}2,\frac12+}}^2
\label{Ivbound}
\end{align}

\noi
and 
\begin{align*}
\sum_{n} \big\|  a_{n} \big\|_{H^{\frac12-}}^2 
& = \|a \|^2_{X^{2s,\frac12+}_+(\o,N)} \les \dl^{-\eps} \| a \|_{X^{0,\frac12+}}^2
\end{align*}

\noi
for small $ \eps > 0$, 
outside   an exceptional set of probability  $< Ce^{-\frac1{\dl^c}}$.
Collecting estimates \eqref{I1}, \eqref{SumPhi}, \eqref{FFTT}, and \eqref{Ivbound}, we obtain 
\begin{align*}
\I(\o) \les \dl^{\frac{1}2 -}\prod_{j=1}^3 \|v_j\|_{X^{\frac{s}2,\frac12+}}
\end{align*}

\noi
outside   an exceptional set of probability  $< Ce^{-\frac1{\dl^c}}$.


\medskip

\noi
$\bullet$ {\bf Estimate on $\II$.} 
Without loss of  generality, we may assume $\II$ has only one term:
\begin{align*}
\II & =  \int_\R \eta_{_\dl}(t) \sum_n  \sum_{\G(n)} e^{it\Psi_N^\o (\bar n)}  \Big[ 
 \ind_{_{|n_1| > N_1}} e^{-itn_1^4}g_{n_1} \cj{\ft v_2 (n_2)} \ft v_3 (n_3) \cj{\ft a (n)} \Big]  dt.
\end{align*}

\noi
With  $b^{(j)}_n$ and $a_n$ as in \eqref{IR1}, 
 Parseval's identity yields
$$
\II = \sum_n \sum_{\G(n)} \frac{\ind_{_{|n_1|> N_1}} g_{n_1}}{\jb{n_2}^s \jb{n_3}^s \jb{n}^{2s}} 
\int_\R  \ft \eta_{_\dl} (\tau + \Phi(\bar n) - |g^N_{n_1}|^2) \FF(\cj{b^{(2)}_{n_2}} b^{(3)}_{n_3} \cj{a_n})(- \tau) d\tau .
$$

\noi
By Cauchy-Schwarz inequality in $\tau$ and then in $n,n_2,n_3$, 
we have 
\begin{align*}
\II & \leq 
\Bigg\| \sum_{\substack{n_1\\
(n_1,n_2,n_3)\in \G(n)}} \frac{ \ind_{_{|n_1|> N_1}} g_{n_1}}{\jb{n_2}^s \jb{n_3}^s \jb{n}^{2s}} 
\frac{ \ft \eta_{_\dl} (\tau + \Phi(\bar n) - |g^N_{n_1}|^2)}{\jb{\tau}^{\frac12-}}  
\Bigg\|_{\l^2_{n,n_2,n_3}L^2_\tau} \\
&  \hphantom{X} \times
\Big\|\jb{\tau}^{\frac12-} \FF(\cj{b^{(2)}_{n_2}} b^{(3)}_{n_3} \cj{a_n})(\tau) \Big\|_{\l^2_{n,n_2,n_3}L^2_\tau}
 \\
& 
\les  S^{s,\frac12 -,\dl}_{1,N}(\o) \Big\| \big\| \cj{b^{(2)}_{n_2}} b^{(3)}_{n_3} \cj{a_n}\big\|_{H^{\frac12-}_\tau} \Big\|_{\l^2_{n,n_2,n_3} } .
\end{align*}

\noi
where $S^{s,b, \dl}_{1,N}(\o)$ is defined in  \eqref{S1}.
Proceeding as in \eqref{FFTT} and \eqref{Ivbound},  we arrive at 
\begin{equation*}
\II
 \le    S^{s,\frac12 -,\dl}_{1,N}(\o)   \|v_2\|_{X^{s,\frac12+}_+(\o,N)} \|v_3\|_{X_+^{s,\frac12+}(\o,N)} \|a\|_{X_+^{2s,\frac12-}(\o,N)},
\end{equation*}

\noi
Then, by applying  Lemmas \ref{LEM:NRSum} and~\ref{LEM:RXsb}, we conclude that 
there exist
small $\ta, \be > 0$
and $s < 0$ close to $0$ such that 
\begin{align*}
\II (\o) \les \dl^{\ta} \jb{N_1}^{-\beta}
 \prod_{j = 2}^3 \|v_j\|_{X^{\frac{s}2,\frac12+}} 
\end{align*}

\noi
outside   an exceptional set of probability  $< Ce^{-\frac1{\dl^c}}$.

\medskip

\noi
$\bullet$ {\bf Estimate on $\III$.} 
Without loss of generality, 
we  assume that $\III$ has the following form:
\begin{align*}
\III & =   \int_\R \eta_{_\dl}(t) \sum_n  \sum_{\G(n)} e^{it\Psi_N^\o(\bar n)} e^{-it(n_1^4 - n_2^4)}  
\chi_{1,2} \cdot g_{n_1}   \cj{g_{n_2}} \ft v_{3}(n_3) \cj{\ft a(n)} dt
\end{align*}

\noi
where $\chi_{_{1,2}}: = \prod_{j = 1}^2 \ind_{_{|n_j|>N_j}}$.
%
By Parseval's identity as before, we have
\begin{align*}
\III = \sum_n  \sum_{\G(n)} \frac{\chi_{_{1,2}} \cdot  g_{n_1} \cj{g_{n_2}}}{\jb{n_3}^s \jb{n}^{2s}} 
\int_\R \frac{ \ft \eta_{_\dl} ( \tau + \Phi(\bar n) - |g_{n_1}^N|^2 + |g_{n_2}^N|^2)}{\jb{\tau}^{\frac12 -}}  
\Big( \jb{\tau}^{\frac12 -} \FF(b^{(3)}_{n_3} \cj{a_n})(-\tau) \Big) d\tau ,
\end{align*}
where $b^{(3)}_n$ and $a_n$ are as in \eqref{IR1}.
By  Cauchy-Schwarz inequality and proceeding as before
we  obtain 
\begin{align*}
\III & \leq \Bigg\| \sum_{\substack{n_1,n_2\\ (n_1,n_2,n_3)\in \G(n)}}
  \chi_{_{1,2}} \cdot  g_{n_1}  \cj{g_{n_2}} 
  \frac{ \ft \eta_{_\dl} ( \tau + \Phi(\bar n) - |g_{n_1}^N|^2 + |g_{n_2}^N|^2)}{\jb{n_3}^s \jb{n}^{2s}\jb{\tau}^{\frac12-}}   \Bigg\|_{ \l^2_{n,n_3}L^2_\tau} \\
&  \hphantom{X}\times
\Big\| \jb{\tau}^{\frac12 -} \FF(b^{(3)}_{n_3} a_n)(\tau) \Big\|_{ \l^2_{n,n_3}L^2_\tau}\\
&  \les S^{s,\frac12-,\dl}_{2,N}(\o)  \|v_3\|_{X_+^{s,\frac12+}(\o,N)} \|a\|_{X^{2s,\frac12-}_+(\o,N)} 
\end{align*}

\noi
where $S^{s,b, \dl}_{2,N}(\o)$ is defined in  \eqref{S2}.
Then, by applying  Lemmas \ref{LEM:NRSum} and~\ref{LEM:RXsb}
 we conclude that 
there exist
small $\ta, \be > 0$
and $s < 0$ close to $0$ such that 
\begin{align*}
\III (\o)  \les \dl^{\ta} \bigg(\prod_{j = 1}^2 \jb{N_1}^{-\beta}\bigg)
 \|v_3\|_{X^{\frac{s}2,\frac12+}} 
\end{align*}

\noi
outside   an exceptional set of probability  $< Ce^{-\frac1{\dl^c}}$.

\medskip

\noi
$\bullet$ {\bf Estimate on $\IV$.} 
Lastly, we consider  $\IV$.
We have
\begin{align*}
\IV & =  \int_{\R} \eta_{_\dl}(t) \sum_n \sum_{\G(n)} e^{it\Psi_N^\o(\bar n)}  e^{-it(n_1^4 - n_2^4 + n_3^4)}   \chi_{_{1,2,3}} \cdot g_{n_1} \cj{g_{n_2}} g_{n_3} \cj{\ft a (n)} dt \notag\\
& =  \sum_n \sum_{\G(n)} \frac{\chi_{_{1,2,3}} \cdot g_{n_1}  \cj{g_{n_2}} g_{n_3}}{\jb{n}^{2s}} \int_{\R} \eta_{_\dl} (t)  e^{it(\Psi_{3,N}^\o(\bar n) - \Phi(\bar n))}  \cj{a_n} dt, 
\end{align*}

\noi
where $\chi_{_{1,2,3}}  : = \prod_{j = 1}^3 \ind_{_{|n_j|>N_j}}$,  $\Psi_N^\o$ is as in \eqref{PsiN}, 
and $\Psi_{3,N} := |g_{n_1}^{N}|^2 -|g_{n_2}^{N}|^2 +|g_{n_3}^{N}|^2$.
By  applying Parseval's identity 
and Cauchy-Schwarz inequality as before, we have 
\begin{align*}
\IV & =  \sum_n \sum_{\G(n)}
\frac{\chi_{_{1,2,3}} \cdot g_{n_1}  \cj{g_{n_2}} g_{n_3}}{\jb{n}^{2s}}
\int_\R \ft \eta_{_\dl} ( \tau + \Phi(\bar n) - \Psi_{3,N}^\o(\bar n)) \cj{\ft{{ a_n}}(\tau)} d\tau\\
 & \le \Bigg\|  \sum_{\G(n)} \chi_{_{1,2,3}}\cdot   g_{n_1} \cj{g_{n_2}} g_{n_3} 
  \frac{\ft \eta_{_\dl}(\tau + \Phi(\bar n) - \Psi_{3,N}^\o(\bar n))}{\jb{n}^{2s}\jb{\tau}^{\frac12-}}  \Bigg\|_{ \l^2_n L^2_\tau} \Big\| \|a_n\|_{H^{\frac12-}} \Big\|_{\l_n^2} \notag\\
& \le S^{s,\frac12 -,\dl }_{3,N} (\o)  \|a\|_{X^{2s,\frac12-}_+(\o,N)}, 
\end{align*}

\noi
where $S^{s,b, \dl}_{3,N}(\o)$ is defined in  \eqref{S3}.
Then, by applying  Lemmas \ref{LEM:NRSum} and~\ref{LEM:RXsb}
 we conclude that 
there exist
small $\ta, \be > 0$
and $s < 0$ close to $0$ such that 
\begin{align*}
\IV  (\o)\les \dl^{\ta} \prod_{j = 1}^3 \jb{N_1}^{-\beta}
\end{align*}

\noi
outside   an exceptional set of probability  $< Ce^{-\frac1{\dl^c}}$.

\smallskip

This completes the  proof of Proposition \ref{PROP:I1}.


\section{Nonlinear estimate II: resonant part}
\label{SEC:Non2}
This section is devoted to the proof of Proposition \ref{PROP:I2}.  
Recall from~\eqref{EnN} and~\eqref{MultiI2N} that
\begin{align*} 
\wt {\mathcal{I}}_{2,N}^\o (w_1,w_2,w_3,w_4,w_5)(x, t) = \int_0^t S(t-t') 
\sum_{n\in \Z} e^{inx} \EE_n^N(w_1,w_2,w_3,w_4)(t') \ft w_5 (n,t') dt',
\end{align*}
where
\[
\EE_n^N(w_1,w_2,w_3,w_4) (t)= -2\Re i \int_0^t 
\sum_{\G(n)} e^{it'\Psi_N^\o(\bar n)} \ft w_1(n_1,t') \cj{\ft w_2(n_2,t')} \ft w_3 (n_3,t') \cj{\ft w_4(n,t')} dt'.
\]

Given $w_j$, 
let $v_j = w_j  - S(t) \pi_{N_j}^\perp (u_0^\o)$.
Then, 
we denote by $\wt v_j$
an extension of $v_j$  (viewed as a function on  the time interval $[-\dl, \dl]$)
and set 
\[\wt w_j = S(t) \pi_{N_j}^\perp (u_0^\o) + \wt v_j.\]

\noi
 Let  $s < 0< \beta$ 
 be sufficiently close to $0$.
By 
 the linear estimate (Lemma \ref{LEM:linear})
and  the definition \eqref{Xsb2} of the local-in-time space, 
it suffices to construct $\O_\dl \subset \O$
with 
 $P(\Omega_\dl^c) < e^{-\frac{1}{\dl^c}}$
   such that for each $\o \in \Omega_\dl$, we have
\begin{align}
\bigg\|
\chi_{_\dl} (t)  & 
 \sum_{n\in \Z} e^{inx}  \EE_n^N(  \wt w_1, \wt w_2,  \wt w_3,\wt w_4)(t) \ft {\wt w_5} (n,t)
\bigg\|_{X^{0,-\frac12+}}  \notag\\
&
 \le C\dl^{\theta} \prod_{j=1}^5 \Big( \jb{N_j}^{-\beta}  + \|\wt  v_j \|_{X^{\frac s2,\frac12+}} \Big)
\label{G1}
\end{align}

 \noi
 for any  extension $\wt v_j$ of $v_j$, $j = 1, \dots, 5$,
uniformly in     $N_j \in \Z_{\geq -1}$, $j = 1, \dots, 5$,
and  $N \geq N_0 (\o, \dl)$ for some $N_0(\o, \dl) \in \N$.
For simplicity of notations, 
we denote  $\wt v_j$ (and $\wt w_j$, respectively) by $v_j$ (and $w_j$, respectively)
in the following.
We also suppress the time dependence when it is clear from the context.

By the (continuous) trivial  embedding $L^2(\T\times \R)= X^{0, 0} \subset X^{0,-\frac12+}$
and  H\"older's inequality, 
we have
\begin{align*}
\text{LHS of \eqref{G1}} 
& \les   
\bigg\| \chi_{_\dl} (t) \bigg( \sum_{n\in \Z}  |\EE_n^N(w_1,w_2,w_3,w_4) (t)\ft w_5 (n, t)|^2 \bigg)^{\frac12} \bigg\|_{L^2_t}\\
&  \les \dl^{\frac12} 
\sup_{t\in [-\dl, \dl]} \bigg( \sum_{n\in \Z}  |\EE_n^N(w_1,w_2,w_3,w_4) (t)\ft w_5 (n,t)|^2 \bigg)^{\frac12}.
\end{align*}

\noi
Therefore,  in order to prove Proposition \ref{PROP:I2}, it suffices to prove
\begin{align}
\sup_{t\in [-\dl, \dl]}  \bigg(  & 
  \sum_{n\in \Z}   |\EE_n^N(w_1,  w_2,w_3,w_4) (t) \ft w_5 (n, t)|^2 \bigg)^{\frac12} \notag\\
&  \le C\dl^{-\frac 12 + } \prod_{j=1}^5 \Big( \jb{N_j}^{-\beta}  + \| v_j   \|_{X^{\frac s2,\frac12+}} \Big)
\label{G2}
\end{align}

\noi
with large probability, where $v_j$ is given by 
\begin{align*}
 v_j =  w_j - S(t) \pi_{N_j}^\perp (u_0^\o).
 \end{align*}

\medskip

\noi
{\bf Step (i): Elimination of $w_5$}. 
With $s < 0$ close to $0$, 
we have
\begin{align}
 \bigg( \sum_{n\in \Z} &   |\EE_n^N(w_1,w_2,w_3,w_4) \ft w_5 (n)|^2 \bigg)^{\frac12} \notag \\
& \le \bigg( \sum_{n\in \Z} \jb{n}^{-2s} |\EE_n^N(w_1,w_2,w_3,w_4) |^2 \bigg)^{\frac12} 
\cdot  \sup_n \big| \jb{n}^{s} \ind_{|n|> N_5} g_n(\o) \big| 
\notag \\
& \hphantom{X}
 + \bigg( \sum_{n\in \Z} \jb{n}^{-2s} |\EE_n^N(w_1,w_2,w_3,w_4) |^2 \bigg)^{\frac12} 
 \cdot \sup_n | \jb{n}^{s} \ft v_5(n)|.
\label{eD23}
\end{align}

\noi
By applying Lemma \ref{LEM:prob} with $\eps = - \frac{s}2 >0$, we conclude that 
\begin{align}
\label{eD22'}
\sup_n \big| \jb{n}^{s}  \ind_{|n|> N_5} g_n(\o) \big|  \le \jb{N_5}^{\frac{s}2} \dl^{0-},
\end{align}

\noi
outside an exceptional set of probability $< C e^{-\frac1{\dl^c}}$.
We also have
\begin{align}
\label{eD23'}
\sup_{t \in [-\dl, \dl]} \sup_n 
| \jb{n}^{s} \ft v_5(n, t)|
\les \| v_5  \|_{X^{s,\frac12+}}.
\end{align}

\noi
Therefore, 
we conclude from \eqref{eD23}, \eqref{eD22'}, and \eqref{eD23'}
that,  in order to prove \eqref{G2}, it  suffices to show the following estimate:
\begin{align}
\label{eD24}
 \sup_{t \in [-\dl, \dl]}  \bigg( & \sum_{n\in \Z}  \jb{n}^{-2s} |\EE_n^N(w_1,w_2,w_3,w_4) |^2 \bigg)^{\frac12}
\le C \dl^{-\frac 12 + } \prod_{j=1}^4 
\Big(  \jb{N_j}^{-\beta} + \|v_j \|_{X^{\frac{s}2,\frac12+}} \Big)
\end{align}

\noi
outside  an exceptional set of probability $< C e^{-\frac1{\dl^{c}}}$,
uniformly in     $N_j \in \Z_{\geq -1}$, $j = 1, \dots, 5$,
and  $N \geq N_0 (\o, \dl)$ for some $N_0(\o, \dl) \in \N$.

\medskip

\noi
{\bf Step (ii) Smoothing effect}. 
In the remaining part of this section, 
 we present the proof of \eqref{eD24}.
By expanding the product of 
\[
\ft w_j (n_j, t) =  \ft v_j (n_j, t) + e^{- itn_j^4} \ind_{|n_j| >N_j} g_{n_j},
\]

\noi
we can bound the left-hand side of \eqref{eD24} (without the supremum in time) by 
\begin{align}
\label{Res-4}
& \hphantom{X}
\bigg\|  \jb{n}^{-s}  \int_0^t  \sum_{\G(n)} 
 e^{it '\Psi^\o_N(\bar n)} 
\ft w_1(n_1, t')\cj{ \ft w_2(n_2, t')} \ft w_3(n_3, t') \cj{\ft w_4(n, t') }
dt' 
\bigg\|_{\l^2_n}\notag \\
& \les A+B+C+D+E,
\end{align}

\noi
where $A$, $B$, $C$, $D$, and $E$ are given by 
\begin{align}
A&: = 
\bigg\| \jb{n}^{-s}
\int_0^t \sum_{\G(n)}  e^{it'( \Psi^\o_N (\bar n) - \Phi(\bar n))} 
   \chi_{_{1,2,3.4}} \cdot g_{n_1} \cj{ g_{n_2}  }    g_{n_3}   \cj{ g_{n} } dt' \bigg\|_{\l^2_n} ,   \notag\\
B& : = \bigg\| \jb{n}^{-2s}
\int_0^t   \sum_{\G(n)}  e^{it'( \Psi_{3,N}^\o(\bar n) - \Phi(\bar n))}   
\chi_{_{1,2,3}} \cdot   g_{n_1} \cj{ g_{n_2}  }   g_{n_3}   \cj{ b^{(4)}_{n} } dt' \bigg\|_{\l^2_n} + \text{similar terms}, \notag\\
C &:= \Bigg\| 
\int_0^t  \sum_{\G(n)} \frac{ e^{it'( \Psi_{2,N}^\o(\bar n) - \Phi(\bar n))}   }{
\jb{n_3}^s \jb{n}^{2s}}
\chi_{_{1,2}} \cdot  g_{n_1} \cj{ g_{n_2}  }   b^{(3)}_{n_3}   \cj{ b^{(4)}_{n} } dt' \Bigg\|_{\l^2_n} + \text{similar terms}, \notag\\
D&: =\Bigg\| 
\int_0^t   \sum_{\G(n)} \frac{ e^{it'( |g^N_{n_1}|^2- \Phi(\bar n))}   }{
\jb{n_2}^s \jb{n_3}^s \jb{n}^{2s}}
\chi_{_1}\cdot  g_{n_1} \cj{ b^{(2)}_{n_2}  }   b^{(3)}_{n_3}   \cj{ b^{(4)}_{n} } dt' 
\Bigg\|_{\l^2_n} + \text{similar terms}, \notag\\
E&: =\Bigg\|
\int_0^t  \sum_{\G(n)} \frac{ e^{- it' \Phi(\bar n)} }{
\jb{n_1}^s \jb{n_2}^s \jb{n_3}^s \jb{n}^{2s}}
  b^{(1)}_{n_1} \cj{ b^{(2)}_{n_2}  }   b^{(3)}_{n_3}   \cj{ b^{(4)}_{n} } dt'
\Bigg\|_{\l^2_n}.\notag
\end{align}

\noi
Here, $b^{(j)}_n$ is as in  \eqref{IR1}, 
\begin{align*}
\chi_{_{1,\dots, k}}   & = \prod_{j =1}^k \ind_{|n_j|> N_j}, \quad k = 1, \dots, 4, 
\end{align*}

\noi
and
\begin{equation*}
\Psi_{k,N}^\o(\bar n)   = \sum_{j = 1}^k (-1)^{j+1} |g^N_{n_j}|^2, 
\quad k = 2, 3. 
\end{equation*}

\noi
In view of the restriction of the time variable
onto $[-\dl, \dl]$, 
we may freely insert 
the cutoff functions $\chi_{_\dl}(t)$ and $\eta_{_\dl}(t)$ in evaluating 
the terms $A$, $B$, $C$, $D$, and  $E$.
In the following, we prove~\eqref{eD24} by estimating each term 
on the right-hand side of  \eqref{Res-4}.

\medskip

\noi
{\bf (ii.1) Estimate on  $A$}.
Fix  $\kk , \eps > 0$ small.
By applying  Lemma \ref{LEM:prob}, 
we have  
\begin{align}
|g_n(\o)|\les \dl^{-\frac{\kk}2} \jb{n}^\eps
\label{G4}
\end{align}

\noi
outside  an exceptional set of probability 
$< C e^{-\frac1{\dl^c}}$.
Then, for such $\o$, we split $A(\o)$ into two parts:
\[
A(\o) = A_1(\o) + A_2(\o), 
\]

\noi
where $A_1(\o) $ denotes the contribution from 
the case $n_{\max} \les \dl^{-\kk}$.
Namely, we have
\[ A_1(\o) : = 
\bigg\| \jb{n}^{-s}
 \chi_{_\dl}(t)
\int_0^t  \sum_{\G(n)}  
\ind_{n_{\max} \les \dl^{-\kk}}
e^{it'( \Psi^\o_N (\bar n) - \Phi(\bar n))} 
   \chi_{_{1,2,3.4}} \cdot g_{n_1} \cj{ g_{n_2}  }    g_{n_3}   \cj{ g_{n} } dt' \bigg\|_{\l^2_n} .
\]

\noi
Note that if $\max(N_1, N_2, N_3, N_4) \gg \dl^{-\kk}$, then we have $A_1(\o) = 0$.
Otherwise, using \eqref{G4}, we have
\begin{align*}
A_1(\o) \les  \dl^{1 + s\kk - C\kk} \prod_{j=1}^4 \jb{N_j}^{-1}
\end{align*}

\noi
for some $C>0$.
This yields \eqref{eD24}.

Next, we consider  $A_2(\o)$.
Since $n_{\max} \gg \dl^{-\kk}$, 
we have $|g_n(\o)|\les \dl^{-\frac{\kk}2} \jb{n}^\eps \ll n_{\max}^{\frac12 + \eps}$.
Then, it follows from Lemma \ref{LEM:phase}
and \eqref{Psi2},  we have
\begin{align}
|\Psi_N^\o(\bar n) - \Phi(\bar n)| \sim 
\jb{\Phi(\bar n)}
\label{G5}
\end{align}

\noi
for  $(n_1,n_2,n_3) \in \G (n)$.
Thus, from \eqref{G4},  \eqref{G5}, and Lemma \ref{LEM:phase},  we obtain
\begin{align*}
A_2(\o)& = 
\bigg\| \jb{n}^{-s}   \sum_{\G(n)} \frac{ e^{it(  \Psi_N^\o(\bar n) - \Phi(\bar n))} - 1}{ \Psi_N^\o(\bar n) - \Phi(\bar n)}   \chi_{_{1,2,3.4}} \cdot   g_{n_1} \cj{ g_{n_2}  }    g_{n_3}   \cj{ g_{n} } \bigg\|_{\l^2_n} \notag\\
& \les   \dl^{-2\kk} 
\bigg( \prod_{j=1}^4 \jb{N_j}^{-\beta}\bigg)
\bigg\|     \sum_{\G(n)} \frac{ n_{\max}^{4 \be + 4\eps-s}}{ \jb{ \Phi(\bar n)}}  \chi_{_{1,2,3.4}}  \bigg\|_{\l^2_n} \notag\\
& \les   \dl^{-2\kk} 
\bigg( \prod_{j=1}^4 \jb{N_j}^{-\beta}\bigg)
\bigg\|     \sum_{\G(n)} \frac{ 1}{n_{\max}^{2-4 \be - 4\eps+s}
(n - n_1) (n- n_3)}  \bigg\|_{\l^2_n} \notag\\
& \les   \dl^{-2\kk} 
\prod_{j=1}^4 \jb{N_j}^{-\beta}, 
\end{align*}

\noi
provided that $\eps,  \beta,  - s>0$ are sufficiently small. 
This yields \eqref{eD24}.

\medskip

\noi
{\bf (ii.2) Estimate on $B$}. Without loss of generality,  we may assume that 
$B$ consists only of one term:
\begin{align*}
B  = \bigg\| \jb{n}^{-2s} \int_0^t \sum_{\G(n)} e^{it'(\Psi^\o_{3,N}(\bar n) - \Phi(\bar n))} \chi_{_{1,2,3}} \cdot
g_{n_1} \cj{g_{n_2}} g_{n_3} \cj{b^{(4)}_n} dt'
\bigg\|_{\l^2_n}.
\end{align*}

\noi
To exploit the oscillatory nature of the time integral, we rewrite the above integral as
\begin{align*}
 \int_\R \eta_{_\dl}(t') \sum_{\G(n)} e^{it'(\Psi^\o_{3,N}(\bar n) - \Phi(\bar n))} \chi_{_{1,2,3}}
 \cdot  g_{n_1} \cj{g_{n_2}} g_{n_3} \Big(\ind_{[0,t]}(t')\cj{b^{(4)}_n}(t')\Big) dt',
\end{align*}
where $\eta_{_\dl}$ is as  in \eqref{eta1}.
Then,  by Parseval's identity, the above expression is
\begin{align*}
  \sum_{\G(n)} &  \chi_{_{1,2,3}} \cdot g_{n_1} \cj{g_{n_2}} g_{n_3} 
 \int_\R \ft \eta_{_\dl}( \tau + \Phi(\bar n)  - \Psi^\o_{3,N}(\bar n)) \F_t(\ind_{[0,t]} \cj{b^{(4)}_n})(-\tau) d\tau 
 \\
& 
=    \sum_{\G(n)} \chi_{_{1,2,3}} \cdot  g_{n_1} \cj{g_{n_2}} g_{n_3}   
\int_\R \frac{ \ft \eta_{_\dl}(  \tau  + \Phi(\bar n)  - \Psi^\o_{3,N}(\bar n))}{ \jb{\tau}^{\frac12-}} \\
&\hphantom{XXXXXXXXXXXXX}
 \times \Big( \jb{\tau}^{\frac12-}\F_t(\ind_{[0,t]} \cj{b^{(4)}_n})(-\tau) \Big) d\tau .
\end{align*}

\noi
Therefore, by  Cauchy-Schwarz inequality in the $\tau$ variable
and Lemma \ref{LEM:bound}, we have
\begin{align*}
B & \le \bigg\|  \sum_{\G(n)}  \chi_{_{1,2,3}} \cdot  g_{n_1} \cj{g_{n_2}} g_{n_3} 
\frac{ \ft \eta_{_\dl}(  \tau  + \Phi(\bar n) - \Psi^\o_{3,N}(\bar n))}{\jb{n}^{2s} \jb{\tau}^{\frac12-}} 
\bigg\|_{\l^2_nL^2_{\tau}} 
 \Big\|  \|\ind_{[0,t]} (t') b^{(4)}_n(t') \|_{H_{t'}^{\frac12-}} \Big\|_{\l^\infty_n} \notag\\
&  \les 
 S^{s,\frac12-, \dl}_{3,N} (\o)   \Big\|  \| b^{(4)}_n\|_{H_t^{\frac12-}} \Big\|_{\l^\infty_n},
\end{align*}

\noi
where $S^{s,b,\dl}_{3,N} (\o)$ is defined  in \eqref{S3}.
Then, proceeding as in \eqref{Ivbound}, 
we obtain 
\begin{equation}\label{B-1}
B (\o) \les  S^{s,\frac12-, \dl}_{3,N} (\o)     \|v_4\|_{X_+^{s,\frac12-}(\o,N)}.
\end{equation}

\noi
Finally, 
 by applying Lemmas \ref{LEM:NRSum} and~\ref{LEM:RXsb} to  \eqref{B-1}, 
 we obtain the desired estimate \eqref{eD24} for the term $B$ 
outside  an exceptional set of probability $< C e^{-\frac1{\dl^{c}}}$.

\medskip

\noi
{\bf (ii.3) Estimate on $C$}.
Without loss of generality, we assume that $C$ consists  only of one term:
\begin{align*}
C &= \Bigg\| 
\int_\R  \sum_{\G(n)} 
\eta_{_\dl}(t')\frac{ e^{it'( \Psi_{2,N}^\o(\bar n) - \Phi(\bar n))}   }{
\jb{n_3}^s \jb{n}^{2s}}
\chi_{_{1,2}} \cdot  g_{n_1} \cj{ g_{n_2}  }  
\Big(\ind_{[0, t]}(t') 
 b^{(3)}_{n_3} (t')  \cj{ b^{(4)}_{n} (t')}\Big) dt' \Bigg\|_{\l^2_n}.
\end{align*}

\noi
By Parseval's identity, we have 
\begin{align*}
C & =  \Bigg\|   \sum_{\G(n)}  
 \chi_{_{1,2}} \cdot  g_{n_1} \cj{g_{n_2}} 
 \int_\R \frac{ \ft \eta_{_\dl} (\tau  + \Phi(\bar n) - \Psi_{2,N}^\o(\bar n))}{\jb{n_3}^{s}\jb{n}^{2s} \jb{\tau}^{\frac12-}} 
 \\
& \hphantom{XXXXXXXXXXX}
\times  \Big( \jb{\tau}^{\frac12-}\F_t(\ind_{[0,t]}  b^{(3)}_{n_3} \cj{b^{(4)}_n})(-\tau) \Big) d\tau 
\Bigg\|_{\l^2_n}.
\end{align*}

\noi
By  Cauchy-Schwarz inequality in $\tau$ and $n_3$
followed by  H\"older's inequality in $n$, we have
\begin{align}
C & \leq   \Bigg\|  \sum_{\substack{n_1,n_2\\ (n_1,n_2,n_3) \in \G(n)}} 
 \chi_{_{1,2}} \cdot   g_{n_1} \cj{g_{n_2}}  \frac{ \ft \eta_{_\dl}
  (\tau  + \Phi(\bar n)-\Psi_{2,N}^\o(\bar n))}{ \jb{n_3}^{s}\jb{n}^{2s} \jb{\tau}^{\frac12-}}
  \Bigg\|_{\l^2_{n, n_3}L^2_\tau } \notag \\
& \hphantom{X}
\times \sup_{n \in \Z} \Big\| \|\ind_{[0,t]}(t') b^{(3)}_{n_3} (t') \cj{b^{(4)}_n(t')}\|_{H_{t'}^{\frac12-}} \Big\|_{\l^2_{n_3}}.
\label{C1}
\end{align}

\noi
As for the second factor of \eqref{C1},
by applying Lemma \ref{LEM:bound} and then Lemma \ref{LEM:algebra} and 
proceeding as in~\eqref{Ivbound}, we have
\begin{align}
\label{C-1}
\sup_{n \in \Z} \Big\| \|\ind_{[0,t]}(t') b^{(3)}_{n_3} (t') \cj{b^{(4)}_n}(t')\|_{H_{t'}^{\frac12-}} \Big\|_{\l^2_{n_3}}
& \les 
 \Big\| \| b^{(3)}_{n_3} \|_{H^{\frac{1}{2}+}}
\| b^{(4)}_{n} \|_{H^{\frac{1}{2}+}} \Big\|_{\l^2_{n, n_3}}\notag\\
& \les \|v_3\|_{X_+^{s,\frac12+}(\o,N)} \|v_4\|_{X_+^{s,\frac12+}(\o,N)}.
\end{align}

\noi
Therefore, from \eqref{C1} and \eqref{C-1}, we obtain
\begin{align}
C (\o)\les   S^{s,\frac12-,\dl}_{2,N} (\o)
 \|v_3\|_{X_+^{s,\frac12+}(\o,N)} \|v_4\|_{X_+^{s,\frac12+}(\o,N)}, 
\label{C4}
\end{align}

\noi
where $S^{s,b,\dl}_{2,N} (\o)$ is defined  in \eqref{S2}.
Finally, 
 by applying Lemmas \ref{LEM:NRSum} and~\ref{LEM:RXsb} to  \eqref{C4}, 
 we obtain the desired estimate \eqref{eD24} for the term $C$ 
outside  an exceptional set of probability $< C e^{-\frac1{\dl^{c}}}$.

\medskip

\noi
{\bf (ii.4) Estimate on $D$}.
Without loss of generality, we assume that $D$ has only one term:
\begin{align*}
D& =\Bigg\| 
\int_0^t   \sum_{\G(n)} \eta_{_\dl}(t') \frac{ e^{it'( |g^N_{n_1}|^2- \Phi(\bar n))}   }{
\jb{n_2}^s \jb{n_3}^s \jb{n}^{2s}}
\chi_{_1}\cdot  g_{n_1} 
\Big( \ind_{[0, t]}(t') \cj{ b^{(2)}_{n_2} (t') }   b^{(3)}_{n_3}  (t') \cj{ b^{(4)}_{n}(t') }\Big) dt' 
\Bigg\|_{\l^2_n}.
\end{align*}

\noi
Proceeding as before
with  Parseval's identity and H\"older's inequality, 
we have
\begin{align*}
D & \leq   \Bigg\|  \sum_{\substack{n_1 \in \Z\\ (n_1,n_2,n_3) \in \G(n)}} 
 \chi_{_{1}} \cdot   g_{n_1}  \frac{ \ft \eta_{_\dl}
  (\tau  + \Phi(\bar n)-|g^N_{n_1}|^2)}{ \jb{n_2}^{s}\jb{n_3}^{s}\jb{n}^{2s} \jb{\tau}^{\frac12-}}
  \Bigg\|_{\l^2_{n, n_2, n_3}L^2_\tau } \notag \\
& \hphantom{X}
\times \sup_{n \in \Z} \Big\| \|
\ind_{[0,t]}(t') \cj{ b^{(2)}_{n_2} (t') }   b^{(3)}_{n_3}  (t') \cj{ b^{(4)}_{n}(t') }
\|_{H_{t'}^{\frac12-}} \Big\|_{\l^2_{n_2, n_3}}.
\end{align*}

\noi
Then, by estimating the  the second factor 
as in \eqref{C-1}
with Lemmas~\ref{LEM:algebra} and~\ref{LEM:bound}, 
 we obtain
\begin{align}
D (\o)\les   S^{s,\frac12-,\dl}_{1,N} (\o)
\prod_{j = 2}^4 \|v_j\|_{X_+^{s,\frac12+}(\o,N)}, 
\label{D4}
\end{align}

\noi
where $S^{s,b,\dl}_{1,N} (\o)$ is defined  in \eqref{S1}.
Finally, 
 by applying Lemmas \ref{LEM:NRSum} and~\ref{LEM:RXsb} to  \eqref{D4}, 
 we obtain the desired estimate \eqref{eD24} for the term $D$ 
outside  an exceptional set of probability $< C e^{-\frac1{\dl^{c}}}$.

\medskip

\noi
{\bf (ii.5) Estimate on $E$}.
We have
\begin{align}
E& =\Bigg\|
\int_0^t  \sum_{\G(n)} \eta_{_\dl}(t')  \frac{ e^{- it' \Phi(\bar n)} }{
\jb{n_1}^s \jb{n_2}^s \jb{n_3}^s \jb{n}^{2s}}
\Big( \ind_{[0, 1]}(t')  b^{(1)}_{n_1}(t') \cj{ b^{(2)}_{n_2} (t') }   b^{(3)}_{n_3}(t')   \cj{ b^{(4)}_{n}(t') } \Big)dt'
\Bigg\|_{\l^2_n}.\notag
\end{align}

\noi
Proceeding as before
with  Parseval's identity and H\"older's inequality, 
we have
\begin{align}
E & \leq  
\sup_{n \in \Z}  \Bigg\|    \frac{ \ft \eta_{_\dl}
  (\tau  + \Phi(\bar n))}{ \jb{n_1}^{s}\jb{n_2}^{s}\jb{n_3}^{s}\jb{n}^{2s} \jb{\tau}^{\frac12-}}
  \Bigg\|_{\l^2_{\G(n)}L^2_\tau } \notag \\
& \hphantom{X}
\times  \Big\| \| \ind_{[0,t]}(t')
 b^{(1)}_{n_1}  (t')
 \cj{ b^{(2)}_{n_2} (t') }   b^{(3)}_{n_3}  (t') \cj{ b^{(4)}_{n}(t') }
\|_{H_{t'}^{\frac12-}} \Big\|_{\l^2_{n, \G(n) }}, 
\label{E1a}
\end{align}

\noi
where the $\l^2_{\G(n)}$-norm is defined by  
\[ \| f_{n_1, n_2, n_3}\|_ {\l^2_{\G(n)}} = 
\bigg( \sum_{(n_1, n_2, n_3 ) \in \G(n) }
|f_{n_1, n_2, n_3}|^2 \bigg)^\frac{1}{2}.\]

\noi
By Lemma~\ref{LEM:GTV}
followed by Lemma \ref{LEM:phase}, 
we can bound the first factor on the right-hand side of~\eqref{E1a} by 
\begin{align*}
\sup_{n \in \Z}  \Bigg\|  
&  \frac{\dl \ft \eta (\dl (\tau - \Phi(\bar n)))}{\jb{n_1}^s \jb{n_2}^{s} \jb{n_3}^{s}\jb{n}^{2s} \jb{\tau}^{\frac12-}}
 \Bigg\|_{ \l^2_{n_1,n_2,n_3}L^2_\tau} 
   \les  
  \sup_{n \in \Z}  \Bigg\|  \frac{\dl^\frac{1}{2}}
  {n_{\max}^{5s} \jb{\tau}^{\frac12-}
\jb{\tau - \Phi(\bar n)}^\frac{1}{2}} \Bigg\|_{\l^2_{\G(n)} L^2_\tau} \notag \\
& \sim \dl^\frac{1}{2}
\bigg( \sum_{(n_1,n_2,n_3)\in \G(n)}  \frac{1}{n_{\max}^{10s}\jb{\Phi(\bar n)}^{1 -}}\bigg)^\frac{1}{2}
\les 1, 
\end{align*}

\noi
provided that $s < 0$ is sufficiently close to $0$.
The second factor on the right-hand side of~\eqref{E1a}
can be estimated 
as in \eqref{C-1}
with Lemmas~\ref{LEM:algebra} and~\ref{LEM:bound}.
Therefore, we obtain 
\begin{align}
\label{E1}
E(\o)\les  \prod_{j=1}^4 \|v_j\|_{X^{s,\frac12+}_+(\o,N)}^2.
\end{align} 

\noi
Finally, 
 by applying Lemma~\ref{LEM:RXsb} to  \eqref{E1}, 
 we obtain the desired estimate \eqref{eD24} for the term $E$ 
outside  an exceptional set of probability $< C e^{-\frac1{\dl^{c}}}$.

\smallskip

This completes the  proof of Proposition \ref{PROP:I2}.

\appendix
\section{Further probabilistic estimates}
\label{SEC:A}

In this appendix, we state 
and prove crucial probabilistic estimates.
These probabilistic estimates  play 
an important role in establishing Propositions \ref{PROP:I1} and \ref{PROP:I2}.
In Subsection~\ref{SUBSEC:A3}, we present the proof of Lemma~\ref{LEM:Z2}.

In the following, 
  $\{g_n\}_{n\in \Z}$ 
denotes  a sequence of independent standard complex-valued Gaussian random variables.
In particular, we have
\begin{align}
\E \big[g_n^k\, \cj {g_m^\l}\big] = \dl_{k \l}\dl_{nm} \cdot  k!
\label{Z2-5}
\end{align}
	
\noi
for any $k, \l \in \Z_{\geq 0}$ and $n, m \in \Z$.
The identity \eqref{Z2-5}  easily follows from 
a computation with the moment generating function
for the chi-square distribution of degree 2 (i.e.~$|g_n|^2 = (\Re g_n)^2 + (\Im g_n)^2$).

\subsection{Random $X^{s,b}$-space}\label{SUBSEC:A.1}
Given $N \in \N \cup\{\infty\}$, 
set  $g_n^N = \ind_{|n| \leq N} \cdot  g_n $ as in  \eqref{gN}
with the understanding that $\ind_{|n|\leq N} \equiv 1$ when $N = \infty$.
Then, 
we define
random versions
 $X^{s,b}_+(\o,N)$ and $X^{s,b}_-(\o,N)$ of the  $X^{s,b}$-space 
 by the norm:
 \begin{align}
\label{rxsb1}
\|u\|_{X^{s,b}_{\pm}(\o,N)} 
= \big\|\jb{n}^s \jb{\tau+n^4\pm|g_n^N(\o)|^2}^b  \ft u (n,\tau) \big\|_{\l^2_n L^2_{\tau}}.
\end{align}

\noi
When $N = \infty$, we simply set $X^{s, b}_{\pm, \o} = X^{s, b}_\pm (\o, \infty)$.
The following lemma shows that 
the random $X^{s, b}$-norm is controlled by the standard $X^{s, b}$-norm in \eqref{Xsb}
 with large probability.

\begin{lemma}
\label{LEM:RXsbA}

Let $\eta\in \S(\R)$ be a Schwartz function in time and $u\in X^{s,b}$ with 
 $s\in \R$ and  $b > 0 $. 
Then, there exists  $C > 0$ such that 
\begin{align}
\label{eRXsb}
\bigg\| \sup_{N \in \N \cup\{\infty\}}\| \eta  u\|_{X^{s,b}_\pm (\o,N)} \bigg\|_{L^p(\O)}\le Cp^{b+2} \|u\|_{X^{s,b}}
\end{align}

\noi
for all $p \ge 2$,
where the constant is independent of $u$.
As a consequence, 
there exist $c, C>0$ such that 
\begin{align}
\label{RX-0} P \bigg( 
 \sup_{N \in \N \cup\{\infty\}}\| \eta  u\|_{X^{s,b}_\pm (\o,N)} > K \|u\|_{X^{s,b}}\bigg)
 \leq C e^{ - K^\frac{1}{b+2}\|u\|_{X^{s,b}}^{-\frac{1}{b+2}}}
\end{align}

\noi
for any $K > 0$.

\end{lemma}

We present the proof of Lemma~\ref{LEM:RXsbA}
at the end of this subsection.
While  the tail estimate~\eqref{RX-0}
holds for each {\it fixed} $u \in X^{s, b}$, 
Lemma~\ref{LEM:RXsbA}
does not provide a uniform control in $u \in X^{s, b}$
and hence is not useful in the proof of 
 the main nonlinear estimates
(Propositions~\ref{PROP:I1} and~\ref{PROP:I2}).
By slightly losing spatial regularity, however, 
we can control the 
 random $X^{s, b}$-norm  by the standard $X^{\s, b}$-norm 
 (with $\s > s$) {\it uniformly} in $u \in X^{\s, b}$.
See Lemma \ref{LEM:RXsb} above.

\begin{lemma}
\label{LEM:RXsb2}
Let $\s > s$ and $b > 0$.
Then,
for each $K > 0$, there exists a set $\O_K\subset \O$ with $P(\O_K^c) < C e^{-cK^\frac{1}{b}}$
such that 
\begin{align}
  \sup_{N \in \N \cup\{\infty\}}\|  u\|_{X^{s,b}_\pm (\o,N)} \les  (1+K) \|u\|_{X^{\s,b}}
\label{RXsb2}
\end{align}

\noi
uniformly in $u \in X^{\s, b}$.

\end{lemma}

\begin{proof}
Fix $\eps > 0$ sufficiently small such that 
\[ \s \geq  s + 2b \eps.\]

\noi
By 
Lemma \ref{LEM:prob}, 
there exists $\O_K$ 
with $P(\O_K^c) < Ce^{-cK^\frac{1}{b}}$ such that 
\begin{align*}
\jb{\tau+n^4\pm |g_n^N(\o)|^2}^b
& \les 
\jb{\tau+n^4}^b + |g_n^N(\o)|^{2b}\\
& \les 
\jb{\tau+n^4}^b + K \jb{n}^{2b \eps}.
\end{align*}

\noi
This implies that 
\[  \sup_{N \in \N \cup\{\infty\}}\|  u\|_{X^{s,b}_\pm (\o,N)} 
\les  \|u\|_{X^{s,b}}  + K \|u\|_{X^{\s,0}}\]

\noi
for each $\o \in \O_K$, 
uniformly in $u \in X^{\s, b}$.
Then, the desired estimate \eqref{RXsb2} follows
from the monotonicity of the $X^{s, b}$-norm in $s$ and $b$.
\end{proof}

We now present the proof of  Lemma \ref{LEM:RXsbA}.

\begin{proof} [Proof of Lemma \ref{LEM:RXsbA}]
Trivially, we have
\begin{align*}
\sup_{N \in \N \cup\{\infty\}} \| \eta  u\|_{X^{s,b}_\pm (\o,N)}
\leq \| \eta  u\|_{X^{s,b}}
+ \| \eta  u\|_{X^{s,b}_\pm (\o,\infty)}.
\end{align*}

\noi
Since the multiplication by a smooth cutoff function $\eta$ is bounded in $X^{s, b}$, 
the estimate~\eqref{eRXsb} follows once we prove
\begin{align}
\label{RX-1b}
\Big\| \| \eta  u\|_{X^{s,b}_{\pm,  \o}} \Big\|_{L^p(\O)}\le Cp^{b+2} \|u\|_{X^{s,b}}.
\end{align}

\noi
The tail estimate \eqref{RX-0} follows from applying Lemma \ref{LEM:tail}
to \eqref{eRXsb}.

Let $v(t) = S(-t)u(t)$ denote the interaction representation of $u$
and set  $a_n(\tau) = \ft v(n, \tau)$.
Then, we have
\begin{align}
\label{RX-1}
\FF(\eta u)(n,\tau) = \int_\R \ft \eta (\tau_1 +n^4)  a_n (\tau -\tau_1) d\tau_1.
\end{align}

\noi
From the definition \eqref{rxsb1},  \eqref{RX-1}, 
and the triangle inequality
$\jb{\tau}^b \les \jb{\tau_1}^b + \jb{\tau -\tau_1}^b$
for $b \geq 0$, 
we have
\begin{align}
\| \eta  u\|_{X^{s,b}_{\pm, \o}}^2 
& =  \sum_n \int_\R \jb{n}^{2s} \jb {\tau }^{2b} |\FF(\eta u) (n,\tau - n^4 \mp |g_n|^2)|^2 d\tau
\notag \\
& =  \sum_n \int_\R \jb{n}^{2s} \jb {\tau }^{2b} 
\bigg| \int_\R \ft \eta (\tau_1 \mp |g_n|^2) a_n (\tau -\tau_1) d\tau_1 \bigg|^2 d\tau\notag \\
& \les   \sum_n \int_\R \jb{n}^{2s} 
\bigg| \int_\R  \jb{\tau_1 }^{b}  \ft \eta (\tau_1 \mp|g_n|^2) a_n (\tau -\tau_1) d\tau_1 \bigg|^2 d\tau \notag \\ 
& \hphantom{X} +  \sum_n \int_\R \jb{n}^{2s}
 \bigg| \int_\R \ft \eta (\tau_1 \mp |g_n|^2)  \jb {\tau-\tau_1 }^{b} a_n (\tau -\tau_1) d\tau_1 \bigg|^2 d\tau \notag\\
& = : \I + \II.
\label{RX-4}
\end{align}

Before proceeding further, 
we   claim the following inequality:
\begin{align}
\label{CM}
E_b (\tau) : = \bigg(\E \Big[ \jb{\tau}^{bp} |\ft \eta (\tau \mp |g_n|^2)|^p\Big] \bigg)^{\frac1p} 
\le C(b) \frac{p^{b+2}}{\jb{\tau}^{2}}.
\end{align}

\noi
We first use this estimate to bound  $\I$ and $\II$ in \eqref{RX-4}.
We present the proof of \eqref{CM} at the end of this proof.

By Minkowski's integral inequality, 
\eqref{CM}, and Young's inequality, we have
\begin{align}
\E \Big[ \I^{\frac{p}2}\Big] 
& \le 
\bigg(\sum_n \int_\R \jb{n}^{2s} \bigg| \int_\R E_b(\tau_1) |a_n(\tau-\tau_1)|  d\tau_1 \bigg|^2 d\tau \bigg)^{\frac{p}2}
\notag \\
& \les p^{(b+2)p} \bigg( \sum_n \jb{n}^{2s} \int_\R |a_n(\tau)|^2 d\tau\bigg)^{\frac{p}2} 
=  p^{(b+2)p} \|u\|_{X^{s,0}}^p.
\label{RX-5a}
\end{align}

\noi
Similarly, we have
\begin{align}
\E \Big[ \II^{\frac{p}2}\Big] 
&\le \bigg( \sum_n \int_\R \jb{n}^{2s} \bigg| \int_\R E_0(\tau_1) 
 \jb{\tau-\tau_1}^b |a_n(\tau-\tau_1)|  d\tau_1 \bigg|^2 d\tau \bigg)^{\frac{p}2}\notag \\
& \les p^{2p} \bigg( \sum_n \jb{n}^{2s} \int_\R \jb{\tau}^{2b} |a_n(\tau)|^2 d\tau\bigg)^{\frac{p}2} 
\les p^{2p} \|u\|_{X^{s,b}}.
\label{RX-5b}
\end{align}

\noi
Hence, \eqref{RX-1b} follows from \eqref{RX-4}, \eqref{RX-5a}, and \eqref{RX-5b}.

It remains to prove  \eqref{CM}. 
By the triangle inequality:
$$
\jb{\tau}\lesssim \langle \tau \mp |g_n|^2\rangle +|g_n|^2
$$ 
and  using the rapid decay of $\ft \eta \in \S(\R)$, we have
\begin{align*} 
\big\| \jb{\tau}^{b} \ft \eta (\tau \mp |g_n|^2)|\big\|_{L^p(\O)}
& \leq  \jb{\tau}^{-2}\big\| \jb{\tau}^{b+2} \ft \eta (\tau \mp |g_n|^2)\big\|_{L^p(\O)}\\
& \lesssim\jb{\tau}^{-2}\big(1+\|g_n\|_{L^{2(b+2)p}(\O)}^{2(b+2)}\big) \lesssim \frac{p^{b+2}}{\jb{\tau}^{2}}, 
\end{align*}

\noi
yielding \eqref{CM}.
This completes the proof of Lemma~\ref{LEM:RXsbA}.
\end{proof}

\subsection{Key tail estimates}

In the following, we present the proof of the key tail estimates (Lemma \ref{LEM:NRSum})
in establishing crucial nonlinear estimates (Propositions~\ref{PROP:I1} and~\ref{PROP:I2}).
Given $s, b \in \R$, $\dl > 0$, and $N \in \N$, 
we recall the definitions of 
 $S^{s,b, \dl}_{j,N}$,  $j = 1, 2, 3$, from \eqref{A1N}, \eqref{A2N}, and \eqref{A3N}
 (expressed in slightly different forms via Taylor expansions):
 \begin{align}
 S^{s,b, \dl}_{1,N}(f) 
&  = \Bigg\| 
  \sum_{\substack{n_1 \in \Z \\(n_1,n_2,n_3)\in \G(n)}} \ft f(n_1)   
\frac{  \ft \eta_{_\dl} (\tau + \Phi(\bar n) - | g_{n_1}^N |^2)}{  \jb{n_2}^s \jb{n_3}^s   \jb{n}^{2s}\jb{\tau}^{b}} \Bigg\|_{\l^2_{n,n_2,n_3}L^2_\tau}, 
\label{XA1N}\\
 S^{s,b, \dl}_{2,N}(f_1,f_2) 
&  =  
\Bigg\| 
\sum_{k_1, k_2 = 0}^\infty
\sum_{\substack{n_1, n_2 \in \Z\\ (n_1,n_2,n_3)\in \G(n)}}  \ft f_1 (n_1)  \cj{\ft f_2 (n_2) }
\notag\\
& \hphantom{XXXXXX}
\times \prod_{j = 1}^2 \frac{ |{g}^N_{n_j}|^{2k_j}}{k_j!}
\cdot  \frac{\dd^{k_1+k_2} \ft \eta_{_\dl} (\tau + \Phi(\bar n))}{   \jb{n_3}^s  \jb{n}^{2s} \jb{\tau}^{b}}  \Bigg\|_{\l^2_{n,n_3}L^2_\tau } , 
\notag\\
 S^{s,b, \dl}_{3,N}(f_1,f_2,f_3) 
&  =   \Bigg\| 
 \sum_{k_1, k_2, k_3 = 0}^\infty
  \sum_{\G(n)}  \ft f_1 (n_1) \cj{\ft f_2 (n_2) }\ft f_3 (n_3)  
\notag\\
& \hphantom{XXXXXX}
\times
\prod_{j = 1}^3 \frac{ |{g}^N_{n_j}|^{2k_j}}{k_j!}\cdot 
\frac{\dd^{k_1+k_2+k_3} \ft \eta_{_\dl}(\tau + \Phi(\bar n))
}{ \jb{n}^{2s} \jb{\tau}^{b}}  \Bigg\|_{\l^2_n L^2_{\tau}}.
\notag
\end{align}

\noi
Here,  $\eta \in C^\infty_c(\mathbb{R})$ 
denotes  a smooth non-negative cutoff function supported on $[-2, 2]$ with $\eta \equiv 1$ on $[-1, 1]$, 
 and the notations  $\G(n)$, $\Phi(\bar n)$, and $\Psi_N^\o(\bar n)$ are as in  \eqref{Gam0}, \eqref{phi1}, and~\eqref{PsiN}, respectively. 
We also recall that there is only one term in the summation over $n_1$ in~\eqref{XA1N}.

For simplicity of notations, we
set
\begin{align*}
S^{s,b, \dl}_{1,N}(\o) & := S^{s,b, \dl}_{1,N}(\pi_{N_1}^\perp(u_0^\o)),\\
S^{s,b, \dl}_{2,N}(\o) & := S^{s,b, \dl}_{2,N}(\pi_{N_1}^\perp(u_0^\o), \pi_{N_2}^\perp(u_0^\o)), \\
S^{s,b, \dl}_{3,N}(\o) & := S^{s,b, \dl}_{3,N}(\pi_{N_1}^\perp(u_0^\o), \pi_{N_2}^\perp(u_0^\o), \pi_{N_3}^\perp(u_0^\o) ), 
\end{align*}

\noi
where  $u_0^\o$ is the white noise in \eqref{white00}
and $\pi_{N_j}^\perp$
denotes the frequency projection onto the frequencies $\{ |n| >N_j \}$ as in \eqref{pi2}
with the convention that $\pi_{-1}^\perp = \Id$.
With the notations defined above, 
we have the following tail estimates
for these random functionals
(Lemma~\ref{LEM:NRSum}).

\begin{lemma}
\label{LEM:NRSum1}
Let $s < 0$ ,  $b < \frac 12$, and $\be > 0$
such that $s$ and $\be$ are sufficiently close to $0$
and  $b$ is sufficiently close to $\frac 12$.
Then, there exist $c, \kk>0$
and small $\dl_0 > 0$  such that the following statements holds.

\smallskip
\noi
\textup{(i)} We have 
\begin{align}
\label{A1bound}
&P\bigg( \bigg\{\o\in \Omega: 
\sup_{N \in \N} \sup_{N_1 \in \Z_{\geq -1}} 
\jb{N_1}^\be |S^{s,b, \dl}_{1,N}(\o)|>   \dl^\kk \bigg\} \bigg) <  e^{-\frac{1}{\dl^c}}
\end{align}
\noi
for any $0 < \dl < \dl_0$.

\smallskip
\noi
\textup{(ii)} Let $k = 2, 3$.
Given $0 < \dl < \dl_0$, 
define the sets $\A_k$ by 
\begin{align*}
\A_k := \bigg\{ \o\in \Omega: 
\ & \text{there exists } N_0 = N_0(\o, \dl) \in \N
\text{ such that }\\
& \sup_{N \ge N_0} \sup_{\substack{N_j \in \Z_{\geq -1}\\ j = 1, \cdots, k}} 
\bigg(\prod_{j = 1}^k \jb{N_j}^\be \bigg)
|S^{s,b, \dl}_{k,N}(\o)|
\leq  \dl^\kk \bigg\}.
\end{align*}

\noi
Then, we have 
\begin{align}
\label{A2bound}
&P( \A_k^c) < e^{-\frac{1}{\dl^c}}
\end{align}

\noi
for any $0 < \dl < \dl_0$.

\end{lemma}

\begin{proof}

In the following, we take $s <0$ and $\be > 0$ both sufficiently close to $0$
and $b < \frac 12$ sufficiently close to $\frac 12$.

We first prove \eqref{A1bound}. 
Fix  $K \gg  1$.
Given small $\eps > 0$, 
it follows from  Lemma~\ref{LEM:prob} that
there exists $\Omega_K\subset \Omega$ with 
\begin{align}
P(\Omega_K^c) \le e^{-cK^2}
\label{H0aa}
\end{align}

\noi
such that we have
\begin{align}
|g_n^N(\o)| \leq K \jb{n}^{\eps}
\label{H0}
\end{align}
for any $\o\in \Omega_K$, 
 any $n \in \Z$, and any $N \in \N$.
We 
separately consider the following two cases:
\[ \text{(i)~$n_{\max}^\frac{1}{2} \les K$ 
\qquad and \qquad (ii) $n_{\max}^\frac{1}{2} \gg K$,}\]

\noi 
where $n_{\max}$ is as in \eqref{Nmax}.
Suppose that  $n_{\max}^\frac{1}{2} \les K $.
By crudely estimating the contribution in this case with \eqref{H0}, 
$N_1 < |n_1| \les K^2$, and $\ft \eta_{_\dl}(\tau) = \dl \ft \eta (\dl \tau)$, 
we have 
\begin{align}
\sup_{N \in \N}
 \sup_{N_1 \in \Z_{\geq -1}}
&  \jb{N_1}^{\be}|S_{1,N}^{s,b, \dl}(\o)|
 \les 
 K^{1-8s+2\be +2\eps} \bigg\|
\ind_{n_{\max} \les K } \frac{\ft \eta_{_\dl}(\tau + \Phi(\bar n) -|g_{n_1}^N(\o)|^2)}{
    \jb{\tau}^{b}} \bigg\|_{\l^2_{n, n_2, n_3}L^2_\tau} \notag \\
& \les 
 K^{1-8s+2\be +2\eps} \bigg\|
\ind_{n_{\max} \les K } \frac{\dl}
{    \jb{\tau}^{b} \dl^{\frac{1}{2} - b + \eps}
\jb{\tau + \Phi(\bar n) -|g_{n_1}^N(\o)|^2}^{\frac{1}{2} - b + \eps}
} \bigg\|_{\l^2_{n, n_2, n_3}L^2_\tau} \notag \\
& \les 
\dl^{\frac{1}{2} + b - \eps}  K^{4-8s+2\be+2\eps} 
  \ll
\dl^{\frac{1}{2} + b - \eps} K^5, 
\label{H0a}
\end{align}

\noi
provided that 
 $K\gg 1$
 and $s$, $\be$, and $\eps$ are all sufficiently close to $0$.

Next, we consider the case  $n_{\max}^\frac{1}{2} \gg K$.
In this case, we have
\[
\big| \Phi(\bar n) -|g_{n_1}^N(\o)|^2 \big| \sim | \Phi(\bar n)  |
\]

\noi
uniformly 
for any $\o \in \Omega_K$,  $\bar n = (n_1, n_2, n_3, n) \in \Z^4$, and $N \in \N$.  
Then, by Lemma \ref{LEM:GTV}, 
we have 
\begin{align}
\bigg\|  
\frac{\ft \eta_{_\dl}(\tau + \Phi(\bar n) - |g_{n_1}^N(\o)|^2)}{\jb{\tau}^{b}} \bigg\|_{L^2_\tau} 
& \les
\bigg( \int 
\frac{\dl^2 } {\jb{\tau}^{2b}
\dl^{2b} \jb{\tau + \Phi(\bar n) - |g_{n_1}^N(\o)|^2}^{2b}} d\tau \bigg)^\frac{1}{2} \notag \\
& \les \frac{\dl^{1-b}} {\jb{\Phi(\bar n)  - |g_{n_1}^N(\o)|^2}^{2b-\frac 12}} \sim \frac{\dl^{1-b}}{\jb{\Phi(\bar n) }^{2b-\frac 12}}
\label{H0b}
\end{align}

\noi
for $b < \frac 12$ sufficiently close to $\frac 12$.
Hence, from \eqref{H0b} and 
Lemma \ref{LEM:phase}, we have
\begin{align}
\sup_{N \in \N}
& \sup_{N_1 \in \Z_{\geq -1}}
 \jb{N_1}^\be |S_{1,N}^{s,b, \dl}(\o)| \notag\\
& \les 
\dl^{1-b} K \bigg\| 
\ind_{(n_1, n_2, n_3) \in \G(n)}
\cdot 
 \ind_{|n_1| \geq N_1}
\frac{\jb{n_1}^\be}{ n_{\max}^{4b-1+4s - \eps } \jb{n-n_1}^{2b-\frac 12} \jb{n - n_3}^{2b-\frac 12}} 
\bigg\|_{\l^2_{n, n_2, n_3}} \notag \\
& \les   \dl^{1-b}K, 
\label{H0c}
\end{align}

\noi
provided that 
 $s$, $\be$, and $\eps$ are all sufficiently close to $0$
and that $b < \frac 12$ is sufficiently close to $\frac 12$.
Hence, by choosing $K = \dl^{-\frac{c}{2}}$ for some small $c > 0$, the bound \eqref{A1bound} follows from \eqref{H0aa}, \eqref{H0a}, and \eqref{H0c}.

Let us now turn to the proof of \eqref{A2bound} for $k = 2$. We have
\begin{align*}
S_{2,N}^{s,b, \dl} (\o)
&  =  \Bigg\|  
\sum_{k_1, k_2 = 0}^\infty
\sum_{\substack{n_1, n_2 \in \Z\\ (n_1,n_2,n_3)\in \G(n)}}  
\chi_{_{1,2}} 
 \prod_{j = 1}^2 \frac{ |{g}^N_{n_j}|^{2k_j}g_{n_j}^*}{k_j!}
\frac{\dd^{k_1+k_2} \ft \eta_{_\dl} (\tau + \Phi(\bar n))}
{\jb{n_3}^s\jb{n}^{2s}
\jb{\tau}^{b}}  \Bigg\|_{ \l^2_{n,n_3}L^2_\tau},
\end{align*}

\noi
where $g_{n_j}^*$ is as in \eqref{Z2-1a}
and  
 $\chi_{_{1,2}} = \prod_{j = 1}^2 \ind_{_{|n_j|>N_j}}$.
By Minkowski's integral inequality
and Lemma~\ref{LEM:Z2}
with \eqref{eta1},  we have
\begin{align}
\|S_{2,N}^{s,b, \dl}\|_{L^p(\O) }
\leq p \dl \sum_{k_1, k_2 = 0}^\infty
(Cp\dl)^{k_1+k_2}
 \bigg\|  
\chi_{_{1,2}} 
\frac{\dd^{k_1+k_2} \ft \eta (\dl (\tau + \Phi(\bar n)))}
{\jb{n_3}^s\jb{n}^{2s}
\jb{\tau}^{b}}  \bigg\|_{ \l^2_{n,\G(\bar n)}L^2_\tau}.
\label{H1}
\end{align}

\noi
We 
separately consider the following two cases:
\[ \text{(i)~$\jb{\tau} \ges |\Phi(\bar n)|$
\qquad and \qquad (ii) $\jb{\tau} \ll |\Phi(\bar n)|$.}\]

First, suppose that $\jb{\tau} \ges |\Phi(\bar n)|$.
By Plancherel's identity with 
\eqref{Xa1}, we have 
\begin{align}
\dl^\frac{1}{2}
\big\| \dd^{k} \ft \eta (\dl (\tau + \Phi(\bar n)))\big\|_{L^2_\tau}\leq C^k
\label{H2}
\end{align}

\noi
for any $k \in \Z_{\geq 0}$.
Then, from \eqref{H1}, \eqref{H2},  Lemma \ref{LEM:phase}, 
and choosing $p = \dl^{-\ta}$ for some $\ta > 0$
such that $Cp\dl < 1$ as in \eqref{Xa3}, we obtain
\begin{align}
\|S_{2,N}^{s,b, \dl}\|_{L^p(\O) }
& \leq p \dl^\frac{1}{2} \jb{N_1}^{-2\be} \jb{N_2}^{-2\be} \sum_{k_1, k_2 = 0}^\infty
(Cp\dl)^{k_1+k_2} \notag \\
& \hphantom{X}
\times \bigg(\sum_{n \in \Z} \sum_{\G(n)}
\frac{\jb{n_1}^{4\be}\jb{n_2}^{4\be}}
{\jb{n_3}^{2s}\jb{n}^{4s}
n_{\max}^{4b} (n - n_1)^{2b} (n-n_3)^{2b}}
\bigg)^\frac{1}{2}\notag \\
& \leq Cp \dl^\frac{1}{2} \jb{N_1}^{-2\be} \jb{N_2}^{-2\be},
\label{H3}
\end{align}

\noi
provided that 
 $s$ and $\be$ are all sufficiently close to $0$
and that $b < \frac 12$ is sufficiently close to $\frac 12$.

Next, we consider the case $\jb{\tau} \ll |\Phi(\bar n)|$.
By Hausdorff-Young's inequality, we have
\begin{align*}
 \big\| \dd^{k} \ft \eta (\dl (\tau + \Phi(\bar n)))\big\|_{L^\infty_\tau}
& \leq  \big\| (-it)^k \eta(t)\big\|_{L^1_t}
\leq C^k,\\
 \big\| \dl (\tau + \Phi(\bar n)) \dd^{k} \ft \eta (\dl (\tau + \Phi(\bar n)))\big\|_{L^\infty_\tau}
& \leq  \big\| \dt \big((-it)^k \eta(t)\big)\big\|_{L^1_t}
\leq C^k
\end{align*}

\noi
for any $k \geq 0$.
By interpolating the two estimates above, we have
\begin{align}
 \big\| \dl^{\frac{1}{2}} (\tau + \Phi(\bar n))^{\frac 12 } \dd^{k} \ft \eta (\dl (\tau + \Phi(\bar n)))\big\|_{L^\infty_\tau}
\leq C^k
\label{H4}
\end{align}

\noi
for any $k \geq 0$.
Then, from \eqref{H4},  Lemma \ref{LEM:phase}
and  choosing $p = \dl^{-\ta}$ as above, 
we obtain
\begin{align}
\|S_{2,N}^{s,b, \dl}\|_{L^p(\O) }
& \leq p \dl^\frac{1}{2} \jb{N_1}^{-2\be} \jb{N_2}^{-2\be}
\sum_{k_1, k_2 = 0}^\infty
(Cp\dl)^{k_1+k_2} \notag\\
& \hphantom{X}
\times  \bigg\|  
\chi_{_{1,2}} 
\frac{\jb{n_1}^{4\be}\jb{n_2}^{4\be}}
{\jb{n_3}^s\jb{n}^{2s}
|\Phi(\bar n)|^{\frac 12-\eps}  \jb{\tau}^{b+\eps }}  \bigg\|_{ \l^2_{n,\G(\bar n)}L^2_\tau}\notag\\
& \leq C p \dl^\frac{1}{2} \jb{N_1}^{-2\be} \jb{N_2}^{-2\be},
\label{H5}
\end{align}

\noi
provided that 
 $s$, $\be$, and $\eps$ are all sufficiently close to $0$
and that $b < \frac 12$ is sufficiently close to $\frac 12$
such that $b + \eps > \frac 12$.
By applying Chebyshev's inequality  with  \eqref{H3} and \eqref{H5}
and choosing $\ld = C p^2 \dl^\frac{1}{2}$
with $p = \dl^{-\ta} $, we obtain
\begin{align}
P\Big( \jb{N_1}^\be \jb{N_2}^\be |S_{2,N, N_1, N_2}^{s,b, \dl}| > \ld \Big)
& \leq \frac{1}{\jb{N_1}^{\be p} \jb{N_2}^{\be p}}C^p \ld^{-p} p^p \dl^\frac{p}{2}\notag\\
& = \frac{1}{\jb{N_1}^{\be p} \jb{N_2}^{\be p}}e^{-p \ln p}
=  \frac{1}{\jb{N_1}^{\be p} \jb{N_2}^{\be p}}e^{-\frac{1}{\dl^c}}.
\label{H6}
\end{align}

\noi
Here, we added subscripts $N_1$ and $N_2$ in 
$S_{2,N, N_1, N_2}^{s,b, \dl}$ to show its dependence on $N_1$ and $N_2$
explicitly.
Now, by summing \eqref{H6} over $N_1, N_2 \in \Z_{\geq -1}$
we obtain 
\begin{align*}
P\bigg( \sup_{\substack{N_j \in \Z_{\geq -1}\\ j = 1, 2}} \jb{N_1}^\be \jb{N_2}^\be 
|S_{2,N, N_1, N_2}^{s,b, \dl}| > \dl^{\frac 12 - 2\ta} \bigg)
& \leq C e^{-\frac{1}{\dl^c}}
\end{align*}

\noi
for any $ 0 < \dl < \dl_0$, where $\dl_0 > 0$
is defined by  $ \be \dl_0^{-\ta} = 1$.

Let $M \geq N \geq 1$.
Then, by slightly modifying the computation above
with the definition~\eqref{gN} of $g_n^N$ and Minkowski's inequality
(on the $\l^2_{n,n_3}L^2_\tau$-norm), we also have
\begin{align*}
\|S_{2,M, N_1, N_2}^{s,b, \dl} - S_{2,N, N_1, N_2}^{s,b, \dl}\|_{L^p(\O) }
& \leq C p \dl^\frac{1}{2} N^{-\be} \jb{N_1}^{-2\be} \jb{N_2}^{-2\be}, 
\end{align*}

\noi
since we must have $n_{\max} \geq N$ to have a non-zero contribution
to the left-hand side above.
This shows that $\big\{S_{2,N, N_1, N_2}^{s,b, \dl}\big\}_{N \in \N}$
forms a Cauchy sequence in $L^p(\O)$ for any $p \geq 1$ and converges to some limit
$S_{2,\infty, N_1, N_2}^{s,b, \dl}$, satisfying 
\begin{align*}
\|S_{2,\infty, N_1, N_2}^{s,b, \dl} - S_{2,N, N_1, N_2}^{s,b, \dl}\|_{L^p(\O) }
& \leq C p \dl^\frac{1}{2} N^{-\be} \jb{N_1}^{-2\be} \jb{N_2}^{-2\be}
\end{align*}

\noi
and
\begin{align}
P\bigg( \sup_{\substack{N_j \in \Z_{\geq-1}\\ j = 1, 2}} \jb{N_1}^\be \jb{N_2}^\be 
|S_{2,\infty, N_1, N_2}^{s,b, \dl}| > \dl^{\frac 12 - 2\ta} \bigg)
& \leq C e^{-\frac{1}{\dl^c}}
\label{H10}
\end{align}

\noi
for any $ 0 < \dl < \dl_0$.

By repeating the computation in \eqref{H6}, we then obtain
\begin{align}
P\bigg( \sup_{\substack{N_j \in \Z_{\geq0}\\ j = 1, 2}} \jb{N_1}^\be \jb{N_2}^\be 
|S_{2,\infty, N_1, N_2}^{s,b, \dl} - S_{2,N, N_1, N_2}^{s,b, \dl}| > \dl^{\frac 12 - 2\ta} \bigg)
& \leq \frac{C}{N^{\be p}} e^{-\frac{1 }{\dl^c}}
\label{H11}
\end{align}

\noi
for any $ 0 < \dl < \dl_0$ (by possibly making $\dl_0$ smaller but independent of $N \in \N$).
Given $\l \in \N$
sufficiently large, 
by choosing $\l = \dl^{2\ta - \frac 12}$, it follows from \eqref{H11} that 
\begin{align*}
\sum_{N = 1}^\infty & P\bigg( \sup_{\substack{N_j \in \Z_{\geq0}\\ j = 1, 2}} \jb{N_1}^\be \jb{N_2}^\be 
|S_{2,\infty, N_1, N_2}^{s,b, \dl} - S_{2,N, N_1, N_2}^{s,b, \dl}| > \frac{1}{\l} \bigg)\notag\\
&  \leq \sum_{N = 1}^\infty \frac{C(\l)}{N^{\be p}} 
 < \infty,
\end{align*}

\noi
since $\be p > 1$.
Therefore, we conclude from the Borel-Cantelli lemma that 
there exists $\O_\l$ with $P(\O_\l) = 1$
such that 
for each $\o \in \O_\l$, 
there exists $N_0 = N_0(\o) \in \N$ such that 
\[ \sup_{\substack{N_j \in \Z_{\geq0}\\ j = 1, 2}} \jb{N_1}^\be \jb{N_2}^\be 
|S_{2,\infty, N_1, N_2}^{s,b, \dl} - S_{2,N, N_1, N_2}^{s,b, \dl}| \leq  \frac{1}{\l}\]

\noi
for any $N \geq N_0$.
By setting $\Sigma = \bigcap_{\l = 1}^\infty \O_\l$, 
we have $P(\Si) = 1$.
This shows that, as $N \to \infty$,  
$S_{2,N, N_1, N_2}^{s,b, \dl}$ converges almost surely to 
$S_{2,\infty, N_1, N_2}^{s,b, \dl}$ 
with respect to the metric:
\[ d(f_{N_1, N_2}, g_{N_1, N_2})
: = \sup_{\substack{N_j \in \Z_{\geq-1}\\ j = 1, 2}} \jb{N_1}^\be \jb{N_2}^\be 
|f_{N_1, N_2} -  g_{N_1, N_2}|.\]
Combining this almost sure convergence with \eqref{H10}, 
we obtain \eqref{A2bound} when $k = 2$.

The proof of \eqref{A2bound} for $k = 3$ follows
in an analogous manner
and hence we omit details.
\end{proof}

\subsection{Proof of Lemma \ref{LEM:Z2}}
\label{SUBSEC:A3}

We conclude this appendix by presenting
the proof of Lemma~\ref{LEM:Z2}.

First, we consider the case  $|\A| = 1$.
\noi
By Stirling's formula: $k! \sim \sqrt k \big(\frac k e\big)^k$, 
there exist $C_0, C>0$ such that 
\begin{align}
\frac{(2k+1)!}{(k!)^2} \leq C_0^k \sqrt k \leq C^k
\label{Z2-7}
\end{align}

\noi
for any $k \in \Z_{\geq 0}$.
Hence, 
the desired estimate \eqref{Z2-2}
follows from 
the Wiener chaos estimate (Lemma \ref{LEM:hyp}), 
\eqref{Z2-5}, and \eqref{Z2-7}.

The proof when $|\A| \geq 2$ follows in a similar manner, using
an estimate such as \eqref{Z2-7}.
In the following, we only present the proof when $|\A| = 3$, 
namely, $\A = \{1, 2, 3\}$, 
since the proof for the case $|\A| = 2$ follows in an analogous manner.
In this case, by the Wiener chaos estimate (Lemma~\ref{LEM:hyp}) with \eqref{Z2-1b}, we have
\begin{equation}
 \|\Si_{n} \|_{L^p(\O)} \leq  (p-1)^{ k + \frac{3}{2}}\|\Si_{n}\|_{L^2(\O)}.
\label{Z2-3}
 \end{equation}

\noi
In the following, we estimate $\|\Si_{n}\|_{L^2(\O)}$.
From \eqref{Z2-1}, we have
\begin{align}
\|\Si_{n}\|_{L^2(\O)}
& = 
 \frac{1}{ k_1! k_2!  k_3!}
\Bigg\|
   \sum_{   (n_1, n_2, n_3) \in \G(n)}
 \sum_{(\wt n_1, \wt n_2, \wt n_3) \in \G(n)}
c^{\bar k}_{n_1, n_2, n_3}
\cj {c^{ \bar k}_{\wt n_1, \wt n_2, \wt n_3}} 
\notag \\
& \hphantom{XXXXXXXXXXXXXX}
\times
\prod_{j = 1}^3 |g_{n_j}|^{2k_j}g^*_{n_j}
\prod_{\wt j = 1}^3 \cj{ |g_{{\wt n}_{\wt j}}|^{2{ k}_j} g^*_{\wt n_{\wt j}}}\Bigg\|_{L^2(\O)}.
\label{Z2-4}
\end{align}

Recall from \eqref{Z2-5} that under the conditions
$n_2 \ne n_1, n_3$
and $\wt n_2 \ne \wt n_1, \wt n_3$,
the right-hand side of 
\eqref{Z2-4} yields zero contribution unless
 $n_2 = \wt n_2$.
Hence, we assume 
 $n_2 = \wt n_2$ 
  in the following.

 \medskip
 
 \noi
 $\bullet$ {\bf Case 1:}
 $n_1 \ne n_3$.
 \quad 
Note that we must have $n_1 = \wt n_1 \ne \wt n_3$ or $n_1 = \wt n_3 \ne \wt n_1$ in this case.
Otherwise, 
the right-hand side of 
\eqref{Z2-4} yields zero contribution.

We first consider the case $n_1 = \wt n_1 \ne \wt n_3$.
In this case, we have $n_3 = \wt n_3$.
Then, from~\eqref{Z2-5}, we obtain
\begin{align}
\text{RHS of } \eqref{Z2-4}
& \leq 
 \frac{1}{ k_1! k_2!  k_3!}
\Bigg(
   \sum_{\G(n)}
|c^{\bar k}_{n_1, n_2, n_3}|^2 
\prod_{j = 1}^3 (2k_j+1)!\Bigg)^\frac{1}{2}\notag \\
& \leq C^k  
\bigg(
  \sum_{\G(n)} 
|c^{\bar k}_{n_1, n_2, n_3}|^2 
\bigg)^\frac{1}{2}.
\label{Z2-8}
\end{align}

Next,  we consider the case $n_1 = \wt n_3 \ne \wt n_1$.
In this case, we have $n_3 = \wt n_1$.
Then,  from~\eqref{Z2-5} and \eqref{Z2-7}, 
we obtain 
\begin{align}
\text{RHS of } \eqref{Z2-4}
& \leq 
 \frac{1}{ k_1! k_2!  k_3!} 
\Bigg(
   \sum_{\G(n)}
|c^{\bar k}_{n_1, n_2, n_3}|^2 
[(k_1 + k_3 + 1)! ]^2 (2k_2+1)!\Bigg)^\frac{1}{2}.
\label{Z2-9}
\end{align}

\noi
We claim that 
\begin{align}
\frac{(k_1 + k_3+1)!}{k_1! k_3!}  \leq C^{k_1 + k_3}
\label{Z2-10}
\end{align}

\noi
for some $C > 0$.
Hence, from \eqref{Z2-9} with \eqref{Z2-7} and \eqref{Z2-10}, we obtain 
\begin{align}
\text{RHS of } \eqref{Z2-4}
& \leq C^k 
\bigg(
  \sum_{\G(n)}
|c^{\bar k}_{n_1, n_2, n_3}|^2 
\bigg)^\frac{1}{2}.
\label{Z2-11}
\end{align}

Hence, it remains to prove \eqref{Z2-9}.
Without loss of generality, assume $k_1 \leq k_3$.
Then,  by Stirling's formula, we have
\begin{align}
\frac{(k_1 + k_3+1)!}{k_1! k_3!}  
\leq
 C^{ k_3}
 \frac{(k_1 + k_3)^\frac{3}{2}}{\sqrt{k_1k_3}}
 \frac{(k_1+ k_3)^{k_1}}{k_1^{k_1}}
\leq
 C^{k_1 +  k_3}
\Bigg[\bigg( 1+ \frac{k_3}{k_1}\bigg)^\frac{k_1}{k_3}\Bigg]^{k_3}.
\label{Z2-12}
\end{align}

\noi
Then, \eqref{Z2-10} follows from \eqref{Z2-12}
once we note that $\lim_{x\to \infty} (1 + x)^\frac{1}{x} = 1$.

\medskip
 
 \noi
 $\bullet$ {\bf Case 2:}
 $n_1 = n_3$.
 \quad 
In this case,  we must have $n_1 = n_3 = \wt n_1 = \wt n_3$.
Proceeding as before with 
\eqref{Z2-5}, we have 
\begin{align}
\text{RHS of } \eqref{Z2-4}
& \leq 
 \frac{1}{ k_1! k_2!  k_3!} 
\Bigg(
   \sum_{\G(n)}
|c^{\bar k}_{n_1, n_2, n_3}|^2 
(2k_1 + 2k_3 + 2)!  (2k_2+1)!\Bigg)^\frac{1}{2}\notag\\
& \leq C^k 
\bigg(
  \sum_{\G(n)}
|c^{\bar k}_{n_1, n_2, n_3}|^2 
\bigg)^\frac{1}{2}, 
\label{Z2-13}
\end{align}

\noi
where we used 
\begin{align}
\frac{(2k_1 + 2k_3+2)!}{(k_1!)^2 (k_3!)^2}  \leq C^{k_1 + k_3}
\label{Z2-14}
\end{align}

\noi
in the second inequality.
The proof of \eqref{Z2-14} is analogous to that of \eqref{Z2-10}
and thus we omit details.

Putting \eqref{Z2-3}, \eqref{Z2-8}, \eqref{Z2-11}, and \eqref{Z2-13} together, 
we obtain \eqref{Z2-2} when $\A = \{1, 2, 3\}$.
This completes the proof of Lemma~\ref{LEM:Z2}.



\begin{ackno}\rm
T.O.~and Y.W.~were supported by the European Research Council (grant no.~637995 ``ProbDynDispEq''). 
N.T. was supported by the ANR grant ODA (ANR-18-CE40-0020-01).
\end{ackno}


\end{document}